\documentclass[reqno,12pt,twoside]{ip-journal}
\makeatletter
\usepackage{mathdots}
\usepackage[pdfusetitle, colorlinks=false, bookmarks=false]{hyperref}
\hypersetup{pdfauthor={Rongqing Ye, Elad Zelingher},
	pdfsubject={Representation theory},
	pdfkeywords={Exterior square, Finite fields, p-adic groups, Level zero representations, Gauss sums}}

\usepackage{cleveref}

\numberwithin{equation}{section}
\numberwithin{figure}{section}
\theoremstyle{definition}
\newtheorem{defn}{\protect\definitionname}[section]
\theoremstyle{plain}
\newtheorem{prop}[defn]{\protect\propositionname}
\theoremstyle{plain}
\newtheorem{thm}[defn]{\protect\theoremname}
\theoremstyle{plain}
\newtheorem{lem}[defn]{\protect\lemmaname}
\theoremstyle{remark}
\newtheorem{rem}[defn]{\protect\remarkname}
\theoremstyle{remark}

\theoremstyle{plain}
\newtheorem{cor}[defn]{\protect\corollaryname}
\theoremstyle{plain}
\newtheorem*{thm*}{\protect\theoremname}

\makeatother

\providecommand{\claimname}{Claim}
\providecommand{\definitionname}{Definition}
\providecommand{\lemmaname}{Lemma}
\providecommand{\propositionname}{Proposition}
\providecommand{\remarkname}{Remark}
\providecommand{\corollaryname}{Corollary}
\providecommand{\theoremname}{Theorem}

\providecommand{\claimname}{Claim}
\providecommand{\corollaryname}{Corollary}
\providecommand{\definitionname}{Definition}
\providecommand{\lemmaname}{Lemma}
\providecommand{\propositionname}{Proposition}
\providecommand{\remarkname}{Remark}
\providecommand{\theoremname}{Theorem}
\providecommand{\theoremname}{Theorem}

\usepackage{scalerel}
\usepackage{stackengine,wasysym}

\newcommand\reallywidetilde[1]{\ThisStyle{%
		\setbox0=\hbox{$\SavedStyle#1$}%
		\stackengine{-.1\LMpt}{$\SavedStyle#1$}{%
			\stretchto{\scaleto{\SavedStyle\mkern.2mu\AC}{.5150\wd0}}{.6\ht0}%
		}{O}{c}{F}{T}{S}%
}}

\def\undertilde#1{\mathord{\vtop{\ialign{##\crcr
				$\hfil\displaystyle{#1}\hfil$\crcr\noalign{\kern1.5pt\nointerlineskip}
				$\hfil\tilde{}\hfil$\crcr\noalign{\kern1.5pt}}}}}

\newcommand{\fieldCharacter}{\psi}
\newcommand{\UpperTriangularAdditive}{\mathcal{B}}
\newcommand{\weylElement}[1]{w_{#1}}
\newcommand{\permutationMatrix}[1]{\sigma_{#1}}
\newcommand{\oddPermutationMatrix}{\permutationMatrix{2m+1}}
\newcommand{\evenPermutationMatrix}{\permutationMatrix{2m}}
\newcommand{\ShalikaSubgroup}[1]{S_{#1}}
\newcommand{\ShalikaSubgroupEven}{\ShalikaSubgroup{2m}}
\newcommand{\ShalikaSubgroupOdd}{\ShalikaSubgroup{2m + 1}}
\newcommand{\ShalikaSubgroupAction}{\rho}
\newcommand{\ShalikaCharacter}{\Psi}
\newcommand{\NilpotentLowerTriangular}{\mathcal{N}^{-}}
\newcommand{\JS}[2]{J_{#1, #2}}
\newcommand{\DualJS}[2]{\tilde{J}_{#1, #2}}
\newcommand{\rowvector}[1]{\varepsilon_{#1}}
\newcommand{\firstrowvector}{\rowvector{1}}
\newcommand{\lastrowvector}{\varepsilon}
\newcommand{\exteriorSquareLFunction}[2]{L\left(#1, #2, \wedge^2 \right)}
\newcommand{\ShalikaDiagonalElement}[1]{\begin{pmatrix}
		#1 &\\		
		& #1
\end{pmatrix}}
\newcommand{\ShalikaUnipotentElement}[1]{\begin{pmatrix}
		\IdentityMatrix{m} & #1\\		
		& \IdentityMatrix{m}
\end{pmatrix}}
\newcommand{\SmallShalikaDiagonalElement}[1]{\left(\begin{smallmatrix}
		#1 &\\		
		& #1
	\end{smallmatrix}\right)}

\newcommand{\SmallShalikaUnipotentElement}[1]{\left(\begin{smallmatrix}
		\IdentityMatrix{m} & #1\\		
		& \IdentityMatrix{m}
	\end{smallmatrix}\right)}

\newcommand{\SmallShalikaDiagonalElementOdd}[1]{\left(\begin{smallmatrix}
		#1 & &\\		
		& #1 &\\
		& & 1
	\end{smallmatrix}\right)}

\newcommand{\ShalikaDiagonalElementOdd}[1]{\begin{pmatrix}
		#1 & &\\		
		& #1 &\\
		& & 1
\end{pmatrix}}

\newcommand{\ShalikaUnipotentElementOdd}[1]{\begin{pmatrix}
		\IdentityMatrix{m} & #1 &\\		
		& \IdentityMatrix{m} & \\
		& & 1
\end{pmatrix}}

\newcommand{\SmallShalikaUnipotentElementOdd}[1]{\left(\begin{smallmatrix}
		\IdentityMatrix{m} & #1 &\\		
		& \IdentityMatrix{m} & \\
		& & 1
	\end{smallmatrix}\right)}

\newcommand{\ShalikaLowerUnipotentElementOdd}[1]{\begin{pmatrix}
		\IdentityMatrix{m} & &\\		
		& \IdentityMatrix{m} & \\
		& #1 & 1
\end{pmatrix}}

\newcommand{\ShalikaUpperRightUnipotentElementOdd}[1]{\begin{pmatrix}
		\IdentityMatrix{m} & & #1\\		
		& \IdentityMatrix{m} & \\
		&  & 1
\end{pmatrix}}

\newcommand{\SmallShalikaUpperRightUnipotentElementOdd}[1]{\left(\begin{smallmatrix}
		\IdentityMatrix{m} & & #1\\		
		& \IdentityMatrix{m} & \\
		&  & 1
	\end{smallmatrix}\right)}

\newcommand{\SmallShalikaLowerUnipotentElementOdd}[1]{\left(\begin{smallmatrix}
		\IdentityMatrix{m} & &\\		
		& \IdentityMatrix{m} & \\
		& #1 & 1
	\end{smallmatrix}\right)}

\newcommand{\WhittakerModelOfResidueFieldRepresentation}{\Whittaker \left(\residueFieldRepresentation, \residueFieldCharacter\right)}
\newcommand{\WhittakerModelOfLocalFieldRepresentation}{\Whittaker \left(\localFieldRepresentation, \fieldCharacter\right)}
\newcommand{\gammaFactorOfLocalField}[1]{\gamma \left(#1, \localFieldRepresentation, \wedge^2, \fieldCharacter \right)}
\newcommand{\epsilonFactorOfLocalField}[1]{\epsilon \left(#1, \localFieldRepresentation, \wedge^2, \fieldCharacter \right)}
\newcommand{\gammaFactorOfResidueField}{\gamma \left(\residueFieldRepresentation, \wedge^2, \residueFieldCharacter \right)}
\newcommand{\modifiedGammaFactorOfResidueField}[1]{\gamma \left(#1, \residueFieldRepresentation, \wedge^2, \residueFieldCharacter \right)}
\newcommand{\JSOfResidueFieldRepresentation}[2]{\JS{\residueFieldRepresentation}{\residueFieldCharacter} \left(#1, #2\right)}
\newcommand{\JSOfLocalFieldRepresentation}[3]{\JS{\localFieldRepresentation}{\fieldCharacter} \left(#1,#2, #3\right)}
\newcommand{\DualJSOfLocalFieldRepresentation}[3]{\DualJS{\localFieldRepresentation}{\fieldCharacter} \left(#1,#2, #3\right)}
\newcommand{\DualJSOfResidueFieldRepresentation}[2]{\DualJS{\residueFieldRepresentation}{\residueFieldCharacter} \left(#1, #2\right)}
\newcommand{\realPartHalfRightPlaneLocalFieldRepresentation}{r_{\localFieldRepresentation, \wedge^2}}
\newcommand{\realPartHalfRightPlaneLocalFieldDualRepresentation}{r_{\Contragradient{\localFieldRepresentation}, \wedge^2}}

\newcommand{\WhittakerModelOfFiniteFieldRepresentation}{\Whittaker \left(\finiteFieldRepresentation, \fieldCharacter\right)}
\newcommand{\gammaFactorOfFiniteField}{\gamma \left(\finiteFieldRepresentation, \wedge^2, \fieldCharacter \right)}
\newcommand{\gammaFactorOfContragredientFiniteField}{\gamma \left(\Contragradient{\finiteFieldRepresentation}, \wedge^2, \fieldCharacter^{-1} \right)}
\newcommand{\JSOfFiniteFieldRepresentation}[2]{\JS{\finiteFieldRepresentation}{\fieldCharacter} \left(#1, #2\right)}
\newcommand{\DualJSOfFiniteFieldRepresentation}[2]{\DualJS{\finiteFieldRepresentation}{\fieldCharacter} \left(#1, #2\right)}

\newcommand{\zIntegers}{\mathbb{Z}}
\newcommand{\rReal}{\mathbb{R}}
\newcommand{\cComplex}{\mathbb{C}}
\newcommand{\multiplicativegroup}[1]{#1^{\ast}}
\newcommand{\RealPart}{\mathrm{Re}}
\newcommand{\Hom}{\mathrm{Hom}}
\newcommand{\Span}{\mathrm{span}}
\newcommand{\Supp}{\mathrm{supp}}
\newcommand{\Stab}{\mathrm{stab}}
\newcommand{\idmap}{\mathrm{id}}
\newcommand{\conjugate}[1]{\overline{#1}}
\newcommand{\indicatorFunction}[1]{\delta_{#1}}
\newcommand{\isomorphic}{\cong}
\newcommand{\sizeof}[1]{\left|#1\right|}

\newcommand{\innerproduct}[2]{\left(#1,#2\right)}
\newcommand{\Norm}[1]{\left\Vert #1\right\Vert }
\newcommand{\standardForm}[2]{\left\langle #1,#2\right\rangle}

\newcommand{\centralCharacter}[1]{\omega_{#1}}
\newcommand{\Ind}[3]{\mathrm{Ind}_{#1}^{#2}\left(#3\right)}
\newcommand{\CompactInd}[3]{\mathrm{ind}_{#1}^{#2}\left(#3\right)}

\newcommand{\UnipotentSubgroup}{N}

\newcommand{\UnipotentRadical}[1]{N_{#1}}
\newcommand{\LeviSubgroup}[1]{M_{#1}}

\newcommand{\rightHaarMeasureModulus}[1]{\delta_{#1}^{-1}}

\newcommand{\Schwartz}{\mathcal{S}}
\newcommand{\Whittaker}{\mathcal{W}}
\newcommand{\residueFieldWhittakerFunctional}{T_0}
\newcommand{\localFieldWhittakerFunctional}{T}
\newcommand{\Contragradient}[1]{\widetilde{#1}}

\newcommand{\FourierTransformWithRespectToCharacter}[2]{\mathcal{F}_{#2}#1}
\newcommand{\rightTranslation}{\rho}
\newcommand{\leftTranslation}{\lambda}
\newcommand{\underlyingVectorSpace}[1]{V_{#1}}
\newcommand{\representationDeclaration}[1]{\left(#1, V_{#1}\right)}
\newcommand{\representationLFunction}[2]{L\left(#1,#2\right)}
\newcommand{\besselFunction}{\mathcal{B}}
\newcommand{\besselFunctionOfFiniteFieldRepresentation}{\besselFunction_{\finiteFieldRepresentation, \fieldCharacter}}
\newcommand{\maximalCompactSubgroup}{K}
\newcommand{\borelSubgroup}{B}

\newcommand{\uniformizer}{\varpi}
\newcommand{\integersring}{\mathfrak{o}}
\newcommand{\maximalideal}{\mathfrak{p}}
\newcommand{\residueField}{\mathfrak{f}}
\newcommand{\localField}{F}
\newcommand{\localFieldRepresentation}{\pi}
\newcommand{\quotientMap}{\nu}
\newcommand{\residueFieldCharacter}{\fieldCharacter_0}
\newcommand{\residueFieldShalikaCharacter}{\ShalikaCharacter_0}
\newcommand{\residueFieldRepresentation}{\pi_0}
\newcommand{\multiplicativeMeasure}[1]{{d^{\times}{#1}}}
\newcommand{\abs}[1]{\left|#1\right|}

\newcommand{\rquot}[2]{{#1}\slash{#2}}
\newcommand{\lquot}[2]{{#1}\backslash{#2}}
\newcommand{\grpIndex}[2]{\left[#1:#2\right]}

\newcommand{\transpose}[1]{\, {}^{t}#1}
\newcommand{\inverseTranspose}[1]{#1^{\iota}}
\newcommand{\IdentityMatrix}[1]{I_{#1}}
\newcommand{\diag}{\mathrm{diag}}
\newcommand{\antidiag}{\operatorname{\mathrm{antidiag}}}
\newcommand{\trace}{\mathrm{tr}}
\newcommand{\GL}[2]{\mathrm{GL}_{#1}\left(#2\right)}
\newcommand{\SquareMat}[2]{M_{#1}\left(#2\right)}
\newcommand{\Mat}[3]{M_{#1 \times #2}\left(#3\right)}
\newcommand{\columnOf}{\mathrm{column}}
\newcommand{\standardColumnVector}[1]{e_{#1}}
\newcommand{\diagonalSubgroup}{A}
\newcommand{\smallDiagTwo}[2]{\left(\begin{smallmatrix}
		#1 & \\
		& #2
	\end{smallmatrix}\right)}
\newcommand{\diagTwo}[2]{\begin{pmatrix}
		#1 & \\
		& #2
\end{pmatrix}}

\newcommand{\FiniteField}[1]{\mathbb{F}_{#1}}
\newcommand{\FieldNorm}{\mathrm{N}}
\newcommand{\FieldTrace}{\mathrm{Tr}}
\newcommand{\cuspidalCharacter}{\theta}
\newcommand{\finiteFieldRepresentation}{\pi}
\newcommand{\finiteField}{\mathbb{F}}
\newcommand{\finiteFieldExtension}[1]{\finiteField_{#1}}
\newcommand{\FieldExtension}[2]{{#1} \slash {#2}}
\newcommand{\algebraicClosure}[1]{\overline{#1}}
\newcommand{\representationCharacter}[1]{\chi_{#1}}

\newcommand{\lift}[1]{\mathcal{L}\left({#1}\right)}

\newcommand{\NonstandardLeviSubgroup}[1]{L_{#1}}
\newcommand{\LeviPermutation}[2]{\sigma_{#1,#2}}
\newcommand{\LeviPermutationMatrix}[2]{w_{#1,#2}}
\newcommand{\TwistedLeviSubgroup}[3][n]{H_{#2, #3}^{\left(#1\right)}}
\newcommand{\CharacterStabilizerSubgroup}[3][n]{S_{#2, #3}^{\left(#1\right)}}
\newcommand{\MirabolicUnipotentRadical}[1]{U_{#1}}
\newcommand{\Mirabolic}[1]{P_{#1}}
\newcommand{\MirabolicRepresentation}{\sigma}
\newcommand{\MirabolicPlusFunctor}{\Phi^{+}}
\newcommand{\MirabolicRepresentationVectorSpace}{\underlyingVectorSpace{\MirabolicRepresentation}}
\newcommand{\MirabolicTensorRepresentation}{\MirabolicRepresentation'}
\newcommand{\MirabolicCharacter}{\Psi}
\newcommand{\columnPermutationMatrix}[1]{w_{#1}}

\title[Exterior square gamma factors for cuspidal representations of $\mathrm{GL}_n$]{Exterior square gamma factors for cuspidal representations of $\mathrm{GL}_n$: finite field analogs and level zero representations}

\author{Rongqing Ye}
\address{Department of Mathematics, Purdue University, West Lafayette, IN 47907, USA}
\email{ye271@purdue.edu}

\author{Elad Zelingher}
\address{Department of Mathematics, Yale University, New Haven, CT 06510, USA}
\email{elad.zelingher@yale.edu}


\begin{document}

\begin{abstract}
We follow Jacquet-Shalika \cite{jacquet11exterior}, Matringe \cite{matringe2014linear} and Cogdell-Matringe \cite{CogdellMatringe15} to define exterior square gamma factors for irreducible cuspidal representations of $\GL{n}{\FiniteField{q}}$. These exterior square gamma factors are expressed in terms of Bessel functions, or in terms of the regular characters associated with the cuspidal representations. We also relate our exterior square gamma factors over finite fields to those over local fields through level zero representations.
\end{abstract}

\maketitle

\section{Introduction}\label{sec:introduction}

Let $\localField$ be a $p$-adic local field of characteristic zero with residue field $\residueField$. Fix a non-trivial additive character $\fieldCharacter$ of $\localField$. In their work \cite{jacquet11exterior}, Jacquet and Shalika define important integrals which we call local Jacquet-Shalika integrals, see \cite[Sections 7, 9.3]{jacquet11exterior}. These Jacquet-Shalika  integrals enable them to introduce integral representations for the exterior square $L$-function $\exteriorSquareLFunction{s}{\localFieldRepresentation}$ of a (generic) representation $\pi$ of $\GL{n}{\localField}$. Later, Matringe \cite{matringe2014linear} and Cogdell-Matringe \cite{CogdellMatringe15} prove local functional equations for these local exterior square $L$-functions, in which local factors $\gammaFactorOfLocalField{s}$ and $\epsilonFactorOfLocalField{s}$ play an important role. These local factors are related via the following equation:

\begin{equation}\label{eqn:relation-of-local-factors}
\gammaFactorOfLocalField{s} = \frac{\epsilonFactorOfLocalField{s}\exteriorSquareLFunction{1-s}{\Contragradient{\localFieldRepresentation}}}{\exteriorSquareLFunction{s}{\localFieldRepresentation}}.
\end{equation}

If $\localFieldRepresentation$ is an irreducible supercuspidal representation of $\GL{n}{F}$, then it is generic, and the local factors $\gammaFactorOfLocalField{s}$ and $\epsilonFactorOfLocalField{s}$ are defined. By type theory of Bushnell and Kutzko \cite{BushnellKutzko93}, $\localFieldRepresentation$ is constructed from some maximal simple type $(J, \lambda)$, see \cite[Section 6]{BushnellKutzko93} for details. Partially, $\lambda$ comes from some irreducible cuspidal representation $\residueFieldRepresentation$ of $\GL{m}{\mathfrak{e}}$, where $\mathfrak{e}$ is some finite extension of $\residueField$. In this paper, we are interested in defining the exterior square gamma factor $\gammaFactorOfResidueField$ for some non-trivial additive character $\residueFieldCharacter$ of $\mathfrak{e}$, and we relate $\localFieldRepresentation$ to $\residueFieldRepresentation$ via their exterior square gamma factors, in the case where $\localFieldRepresentation$ is of level zero. To be concrete, the main result of this paper is that if $\localFieldRepresentation$ is a level zero representation constructed from $\residueFieldRepresentation$, a cuspidal representation of $\GL{n}{\residueField}$ which does not admit a Shalika vector, then $\gammaFactorOfLocalField{s} = \gammaFactorOfResidueField$.

In \Cref{sec:jacquet-shalika-integral}, we work with a general finite field $\finiteField$ and a non-trivial additive character $\fieldCharacter$. For a generic representation $\finiteFieldRepresentation$ of $\GL{n}{\finiteField}$ where $n = 2m$ or $n = 2m+1$, we follow \cite{jacquet11exterior} to define the Jacquet-Shalika integral $\JSOfFiniteFieldRepresentation{W}{\phi}$ and its dual $\DualJSOfFiniteFieldRepresentation{W}{\phi}$, where $W$ is a Whittaker function of $\finiteFieldRepresentation$ and $\phi$ is a complex valued function on $\finiteField^m$. Using ideas from \cite{matringe2014linear,CogdellMatringe15}, we define the exterior square gamma factor $\gammaFactorOfFiniteField$ of $\finiteFieldRepresentation$ in

\begin{thm*}
Let $\finiteFieldRepresentation$ be an irreducible cuspidal representation of $\GL{n}{\finiteField}$ that does not admit a Shalika vector. Then there exists a non-zero constant $\gammaFactorOfFiniteField$ such that
$$\DualJS{\finiteFieldRepresentation}{\fieldCharacter}\left(W, \phi\right) = \gammaFactorOfFiniteField \cdot \JSOfFiniteFieldRepresentation{W}{\phi},$$
for any Whittaker function $W \in \WhittakerModelOfFiniteFieldRepresentation$ and any complex valued function $\phi$ on $\finiteField^m$. Here $n = 2m$ or $n = 2m+1$.
\end{thm*}

This theorem is presented as \Cref{thm:functional-equation-finite-field} in \Cref{sec:jacquet-shalika-integral}. Its proof is separated into the even and odd cases, since the Jacquet-Shalika integrals differ in these cases. The proof relies heavily on some multiplicity one theorems, whose details can be found in \cref{sec:multiplicity-one}. If $\finiteFieldRepresentation$ admits a Shalika vector, we need to borrow results from the local field case in order to define the exterior square gamma factor, which is discussed in \cref{thm:modified-functional-equation}. By choosing some suitable $W$ and $\phi$, we are able to express $\gammaFactorOfFiniteField$ as a summation of the Bessel function $\besselFunctionOfFiniteFieldRepresentation$ over some tori of $\GL{n}{\finiteField}$, see \cref{thm:mellin-inverse-even-case} and \cref{thm:mellin-inverse-odd-case}. The ability to write local factors closely related to the exterior square gamma factor in terms of integrals of partial Bessel functions over tori as in \cite[Proposition 4.6]{CogdellShahidiTsai14}, is one of the key ingredients that Cogdell, Shahidi and Tsai use in order to prove the stability of exterior square gamma factors for local fields.

In \Cref{subsec:computation-regular-character}, we perform some computations using the correspondence between irreducible cuspidal representations of $\GL{n}{\finiteField}$ and equivalence classes of regular multiplicative characters of $\multiplicativegroup{\finiteFieldExtension{n}}$, where $\FieldExtension{\finiteFieldExtension{n}}{\finiteField}$ is a field extension of degree $n$. This correspondence can be seen as an analog of the Langlands correspondence for finite fields. Let $\finiteFieldRepresentation$ be an irreducible cuspidal representation of $\GL{n}{\finiteField}$ associated with a regular character $\cuspidalCharacter$ of $\multiplicativegroup{\finiteFieldExtension{n}}$, as in \cite[Section 6]{gel1970representations}. We are able to express $\gammaFactorOfFiniteField$ explicitly in terms of $\cuspidalCharacter$ and $\fieldCharacter$ only, for $n = 2, 3, 4$. 

In \Cref{sec:gamma-factor}, we assume that $\localFieldRepresentation$ is a level zero representation of $\GL{n}{F}$, where $\localField$ is a $p$-adic local field of characteristic zero with residue field $\residueField$. Let $\fieldCharacter$ be a non-trivial additive character of $\localField$, which descends to a non-trivial additive character $\residueFieldCharacter$ of $\residueField$. By type theory of Bushnell and Kutzko, $\localFieldRepresentation$ is constructed from a maximal simple type $(\GL{n}{\integersring}, \residueFieldRepresentation)$, where $\integersring$ is the ring of integers of $\localField$ and $\residueFieldRepresentation$ is an irreducible cuspidal representation of $\GL{n}{\residueField}$. Our main result is in \Cref{thm:equality-of-gamma-factors-non-shalika-vector}, which states that if $\residueFieldRepresentation$ does not admit a Shalika vector, then
$$\gammaFactorOfLocalField{s} = \gammaFactorOfResidueField.$$
$\gammaFactorOfResidueField$ is a non-zero constant, thus the above equation implies that $\gammaFactorOfLocalField{s}$ is a non-zero constant independent of $s$. The case where $\residueFieldRepresentation$ admits a Shalika vector is also treated in \Cref{thm:equality-of-gamma-factors-with-shalika-vector}. We are also able to specify the local exterior square factors $\exteriorSquareLFunction{s}{\localFieldRepresentation}$ and $\epsilonFactorOfLocalField{s}$ explicitly along with the equalities of the exterior square gamma factors.

This paper grows out of part of a thesis project \cite{yephd19} of the first author, and the master's thesis \cite{Zelingher17} of the second author.

\section{The Jacquet-Shalika integral over a finite field}\label{sec:jacquet-shalika-integral}

In this section, we define analogs of the local Jacquet-Shalika integrals \cite[Section 7]{jacquet11exterior}, \cite[Section 3.2]{CogdellMatringe15} over a finite field. We then prove that they satisfy a functional equation, which defines an important invariant, the exterior square gamma factor. This functional equation is valid only under some assumption on the relevant representation, which we restate in several equivalent ways. We then express the exterior square gamma factor in terms of the Bessel function of Gelfand \cite{gel1970representations}. Finally, we express the exterior square gamma factor in terms of the multiplicative regular character associated with the representation for $n = 2,3,4$.

\subsection{Preliminaries and notations}\label{subsection:finite-field-notations}

\subsubsection{Notations}

Let $\finiteField$ be a finite field and denote $q = \sizeof{\finiteField}$. Fix an algebraic closure $\algebraicClosure{\finiteField}$ of $\finiteField$.  For every positive integer $m$, we denote by $\finiteFieldExtension{m}$ the unique field extension of $\finiteField$ in $\algebraicClosure{\finiteField}$ of degree $m$.

Let $\fieldCharacter:\finiteField \rightarrow \multiplicativegroup{\cComplex}$ be a non-trivial additive character.

For a non-negative integer $m$, we denote $\Schwartz \left( \finiteField^m \right) = \left\{ f:\finiteField^{m} \rightarrow \cComplex \right\}$, the space of complex valued functions on $\finiteField^m$.

If $\phi \in \Schwartz \left( \finiteField^m \right)$, we define its Fourier transform with respect to $\fieldCharacter$ by the formula
$$\FourierTransformWithRespectToCharacter{\phi}{\fieldCharacter}\left(y\right) =  q^{-\frac{m}{2}} \sum_{x \in \finiteField^m} {f\left(x\right) \fieldCharacter\left(\standardForm{x}{y}\right)},$$
where $\standardForm{x}{y}$ is the standard bilinear form on $\finiteField^m$. We have the following Fourier inversion formulas:
\begin{equation*}
	\begin{split}
		\FourierTransformWithRespectToCharacter{\FourierTransformWithRespectToCharacter{\phi}{\fieldCharacter}}{\fieldCharacter}\left(x\right) = \phi\left(-x\right), \\
		\FourierTransformWithRespectToCharacter{\FourierTransformWithRespectToCharacter{\phi}{\fieldCharacter}}{\fieldCharacter^{-1}}\left(x\right) = \phi\left(x\right).
	\end{split}
\end{equation*}

For an irreducible generic representation $\finiteFieldRepresentation$ of $ \GL{n}{\finiteField}$, we denote by $\WhittakerModelOfFiniteFieldRepresentation$ the Whittaker model of $\finiteFieldRepresentation$ with respect to the character $\fieldCharacter$.

For $W \in \WhittakerModelOfFiniteFieldRepresentation$, we define $\Contragradient{W} \in \Whittaker\left(\Contragradient{\finiteFieldRepresentation}, \fieldCharacter^{-1}\right)$ by $\Contragradient{W}\left(g\right) = W\left( \weylElement{n} \inverseTranspose{g} \right),$ where $\inverseTranspose{g} = \transpose{g^{-1}}$ and $\weylElement{n} = \left(\begin{smallmatrix}
			& & 1 \\
			& \iddots &   \\
			1 &  &
		\end{smallmatrix}\right)$.

Let $n_1, \dots, n_r \ge 1$ with $n_1 + \dots + n_r = n $ and $ g_i \in \GL{n_i}{\finiteField} $ for every $ i $. We denote
$$ \antidiag\left(g_1, \dots, g_r\right) =
	\begin{pmatrix}
		      &         &     & g_1 \\
		      &         & g_2 &     \\
		      & \iddots &     &     \\
		g_{r} &         &     &
	\end{pmatrix} \in \GL{n}{\finiteField}.$$

\subsubsection{The Bessel function}

Let $\representationDeclaration{\finiteFieldRepresentation}$ be an irreducible generic representation of $\GL{n}{\finiteField}$. By \cite[Propositions 4.2, 4.3]{gel1970representations} there exists a unique element $\besselFunctionOfFiniteFieldRepresentation \in \WhittakerModelOfFiniteFieldRepresentation$  satisfying  $\besselFunctionOfFiniteFieldRepresentation \left( \IdentityMatrix{n}\right) = 1$ and $ \besselFunctionOfFiniteFieldRepresentation \left( u_1gu_2 \right) = \fieldCharacter\left( u_1 \right)\fieldCharacter\left( u_2 \right) \besselFunctionOfFiniteFieldRepresentation \left(g \right) $, for every $ g \in \GL{n}{\finiteField}$ and $u_1, u_2 \in N_n$, where $N_n$ is the upper triangular unipotent subgroup of $\GL{n}{\finiteField}$. We gather some properties of $\besselFunctionOfFiniteFieldRepresentation$.

\begin{thm}[{\cite[Proposition 4.9]{gel1970representations}}]\label{thm:support-of-bessel-function}
	Let $w, d \in \GL{n}{\finiteField}$, where $w$ is a permutation matrix and $d$ is a diagonal matrix. Suppose that $\besselFunctionOfFiniteFieldRepresentation \left(dw\right) \ne 0$. Then $dw = \antidiag \left(\lambda_{1} \IdentityMatrix{n_1}, \dots, \lambda_{r} \IdentityMatrix{n_r} \right)$, where $n_1 + \dots + n_r = n$ and $\lambda_{1}, \dots, \lambda_{r} \in \multiplicativegroup{\finiteField}$.
\end{thm}

\begin{thm}[{\cite[Proposition 4.5]{gel1970representations}}]\label{thm:bessel-function-in-terms-of-representation-character}
	$$ \besselFunctionOfFiniteFieldRepresentation\left(g\right) = \frac{1}{\sizeof{\UnipotentRadical{n}}} \sum_{u \in \UnipotentRadical{n}} { \trace \left(\finiteFieldRepresentation\left(gu\right)\right) \fieldCharacter^{-1} \left(u\right)}.$$
\end{thm}

\begin{prop}[{\cite[Proposition 2.15]{Nien14} or \cite[Proposition 6.1.2]{Roditty10}}]\label{prop:argument-inverse-and-conjugate-of-Bessel-function}
	$$ \besselFunctionOfFiniteFieldRepresentation \left( g^{-1} \right) = \conjugate{\besselFunctionOfFiniteFieldRepresentation \left( g \right)}. $$
\end{prop}

Since representations of finite groups are unitary, we get the following

\begin{cor}[{\cite[Proposition 3.5]{Nien14}}]
	$$ \conjugate{\besselFunctionOfFiniteFieldRepresentation \left( g \right)} = \besselFunction_{\Contragradient{\finiteFieldRepresentation}, \fieldCharacter^{-1}} \left( g \right). $$
\end{cor}

\subsection{The Jacquet-Shalika integral}\label{subsec:jacquet-shalika-integral}

In this section, we define the Jacquet-Shalika integral analog over the finite field $\finiteField$.

\subsubsection{The even case}

Let $\representationDeclaration{\finiteFieldRepresentation}$ be an irreducible generic representation of $\GL{2m}{\finiteField}$.
For $W \in \WhittakerModelOfFiniteFieldRepresentation$ and $\phi \in \Schwartz \left( \finiteField^m \right)$ we define the Jacquet-Shalika integral as

$$ \JSOfFiniteFieldRepresentation{W}{\phi} = \frac{1}{\grpIndex{G}{N} \grpIndex{M}{\UpperTriangularAdditive}} \sum_{g \in \lquot{\UnipotentSubgroup}{G}} \sum_{X \in \lquot{\UpperTriangularAdditive}{M}} W\left( \evenPermutationMatrix \ShalikaUnipotentElement{X} \ShalikaDiagonalElement{g}
	\right) \fieldCharacter\left(-\trace X\right) \phi \left( \lastrowvector g \right).$$
Here $G = \GL{m}{\finiteField}$, $N \le G$ is the upper triangular unipotent subgroup, $M = \SquareMat{m}{\finiteField}$, $\UpperTriangularAdditive \le M$ is the upper triangular matrix subspace, $\lastrowvector = \rowvector{m} = \begin{pmatrix}
		0 & \cdots & 0 & 1
	\end{pmatrix}$. $\evenPermutationMatrix$ is the column permutation matrix corresponding to the permutation

$$\sigma = \begin{pmatrix}
		1 & 2 & 3 & \cdots & m    & \mid & m+1 & m+2 & \cdots & 2m \\
		1 & 3 & 5 & \cdots & 2m-1 & \mid & 2   & 4   & \cdots & 2m
	\end{pmatrix}.$$
We also define the dual Jacquet-Shalika integral as
$$ \DualJS{\finiteFieldRepresentation}{\fieldCharacter}\left(W, \phi\right) = \JS{\Contragradient{\finiteFieldRepresentation}}{\fieldCharacter^{-1}}\left(\Contragradient{\finiteFieldRepresentation} \begin{pmatrix}
		                   & \IdentityMatrix{m} \\
		\IdentityMatrix{m} &
	\end{pmatrix} \Contragradient{W}, \FourierTransformWithRespectToCharacter{\phi}{\fieldCharacter} \right).$$
It is immediate from the definition that
\begin{prop}[Double-duality]\label{prop:double-duality-of-jacquet-shalika-integral-even}
	$$\DualJS{\Contragradient{\finiteFieldRepresentation}}{\fieldCharacter^{-1}}\left( \Contragradient{\finiteFieldRepresentation} \begin{pmatrix}
			 & \IdentityMatrix{m} \\
			\IdentityMatrix{m}
		\end{pmatrix} \Contragradient{W}, \FourierTransformWithRespectToCharacter{\phi}{\fieldCharacter} \right) = \JSOfFiniteFieldRepresentation{W}{\phi}.$$
\end{prop}
One can show the following
\begin{prop}\label{prop:dual-jacquet-shalika-formula-even}
	$$ \DualJSOfFiniteFieldRepresentation{W}{\phi} = \frac{1}{\grpIndex{G}{N} \grpIndex{M}{\UpperTriangularAdditive}} \sum_{g \in \lquot{\UnipotentSubgroup}{G}} \sum_{X \in \lquot{\UpperTriangularAdditive}{M}} W\left( \evenPermutationMatrix \ShalikaUnipotentElement{X}	\ShalikaDiagonalElement{g}
		\right) \fieldCharacter\left(-\trace X\right) \FourierTransformWithRespectToCharacter{\phi}{\fieldCharacter} \left( \firstrowvector \inverseTranspose{g} \right),$$
	where $\firstrowvector=\begin{pmatrix}
			1 & 0 & \cdots & 0
		\end{pmatrix}.$
\end{prop}

The following proposition is a simple exercise that follows from \cref{thm:support-of-bessel-function}. It will also follow as a special case of \cref{lem:support-of-bessel-function-jacquet-shalika-integral}.

\begin{prop}\label{prop:jacquet-shalika-integral-of-bessel-function-even-case}
	Let $W = \grpIndex{G}{N} \grpIndex{M}{\UpperTriangularAdditive} \finiteFieldRepresentation \left( \evenPermutationMatrix^{-1} \right) \besselFunctionOfFiniteFieldRepresentation$, and let $$\phi \left(x\right) = \indicatorFunction{\lastrowvector}\left(x\right) = \begin{cases}
	1 & x=\lastrowvector \\
	0 & \text{else}
	\end{cases}.$$
	Then $\JSOfFiniteFieldRepresentation{W}{\phi} = 1$.
\end{prop}

\begin{defn}\label{defn:shalika-subgroup-even}
	\textbf{The Shalika subgroup} $\ShalikaSubgroupEven \le \GL{2m}{\finiteField}$ is defined as
	$$ \ShalikaSubgroupEven = \left\{ \begin{pmatrix}
			g & X \\
			  & g
		\end{pmatrix}
		\mid g \in \GL{m}{\finiteField}, \, X \in \SquareMat{m}{\finiteField}
		\right\}.$$
	$\ShalikaSubgroupEven$ acts on $ \Schwartz \left( \finiteField^m \right)$ by $$\left(\ShalikaSubgroupAction \begin{pmatrix}
				g & X \\
				  & g
			\end{pmatrix}\phi\right)\left(y\right) = \phi(y g).$$
	We define a character $\ShalikaCharacter : \ShalikaSubgroupEven \rightarrow \multiplicativegroup{\cComplex}$ by $$\ShalikaCharacter \begin{pmatrix}
			g & X \\
			  & g
		\end{pmatrix} = \fieldCharacter\left( \trace\left( X g^{-1} \right) \right).$$
\end{defn}

One easily checks that

\begin{prop}[Equivariance property]\label{prop:equivariance-of-jacquet-shalika-integral-even-case}
	Let $s \in \ShalikaSubgroupEven$ and $B \in \left\{ \JS{\finiteFieldRepresentation}{\fieldCharacter}, \DualJS{\finiteFieldRepresentation}{\fieldCharacter} \right\} $, we have
	$$B\left( \finiteFieldRepresentation \left(s\right) W,
		\rightTranslation\left( s \right) \phi\right)  = \ShalikaCharacter\left( s \right) B\left( W, \phi\right).$$
\end{prop}

\begin{defn}\label{defn:shalika-vector}
	A vector $0 \ne v \in \underlyingVectorSpace{\finiteFieldRepresentation}$ is called a \textbf{Shalika vector}, if for every $s \in \ShalikaSubgroupEven$ we have
	$$\finiteFieldRepresentation \left( s \right)v = \ShalikaCharacter \left( s \right)v.$$
\end{defn}

\begin{rem}\label{rem:shalika-vector-trivial-character}
	Suppose that $\finiteFieldRepresentation$ admits a Shalika vector $v$. Then for every $a \in \multiplicativegroup{\finiteField}$ we have $aI_{2m} \in \ShalikaSubgroupEven$ and $\centralCharacter{\finiteFieldRepresentation}\left(a\right)v = \finiteFieldRepresentation \left( a \IdentityMatrix{2m} \right) v = v$, and therefore $\finiteFieldRepresentation$ has trivial central character.
\end{rem}

\subsubsection{The odd case}

Let $\representationDeclaration{\finiteFieldRepresentation}$ be an irreducible generic representation of $\GL{2m+1}{\finiteField}$. For $W \in \WhittakerModelOfFiniteFieldRepresentation$ and $\phi \in \Schwartz \left( \finiteField^m \right)$ we define the Jacquet-Shalika integral as
\begin{equation*}
	\begin{split}
		\JSOfFiniteFieldRepresentation{W}{\phi} = & \frac{1}{\grpIndex{G}{N} \grpIndex{M}{\UpperTriangularAdditive}\sizeof{\Mat{1}{m}{\finiteField}}} \sum_{g \in \lquot{\UnipotentSubgroup}{G}} \sum_{X \in \lquot{\UpperTriangularAdditive}{M}} \sum_{Z \in \Mat{1}{m}{\finiteField}} \\
		&W\left( \oddPermutationMatrix \ShalikaUnipotentElementOdd{X} \ShalikaDiagonalElementOdd{g} \ShalikaLowerUnipotentElementOdd{Z}
		\right) \fieldCharacter\left(-\trace X\right) \phi \left( Z \right).
	\end{split}
\end{equation*}
Here the notations are the same as in the even case, except for $\oddPermutationMatrix$ which is the column permutation matrix corresponding to the permutation

$$\sigma = \begin{pmatrix}
		1 & 2 & 3 & \cdots & m    & \mid & m+1 & m+2 & \cdots & 2m & \mid & 2m+1 \\
		1 & 3 & 5 & \cdots & 2m-1 & \mid & 2   & 4   & \cdots & 2m & \mid & 2m+1
	\end{pmatrix}.$$
We also define the dual Jacquet-Shalika integral as

$$\DualJS{\finiteFieldRepresentation}{\fieldCharacter}\left(W, \phi\right) = \JS{\Contragradient{\finiteFieldRepresentation}}{\fieldCharacter^{-1}}\left(\Contragradient{\finiteFieldRepresentation} \begin{pmatrix}
		                   & \IdentityMatrix{m} &   \\
		\IdentityMatrix{m} &                    &   \\
		                   &                    & 1
	\end{pmatrix} \Contragradient{W}, \FourierTransformWithRespectToCharacter{\phi}{\fieldCharacter} \right).$$
Again, we get immediately from the definition the following analog of \cref{prop:double-duality-of-jacquet-shalika-integral-even}:
\begin{prop}[Double-duality]\label{prop:double-duality-of-jacquet-shalika-integral-odd}
	$$\DualJS{\Contragradient{\finiteFieldRepresentation}}{\fieldCharacter^{-1}}\left( \Contragradient{\finiteFieldRepresentation} \begin{pmatrix}
			                   & \IdentityMatrix{m} &   \\
			\IdentityMatrix{m} &                        \\
			                   &                    & 1
		\end{pmatrix} \Contragradient{W}, \FourierTransformWithRespectToCharacter{\phi}{\fieldCharacter} \right) = \JSOfFiniteFieldRepresentation{W}{\phi}.$$
\end{prop}

One can show the following two propositions which are similar to \cref{prop:dual-jacquet-shalika-formula-even} and \ref{prop:jacquet-shalika-integral-of-bessel-function-even-case}, respectively.

\begin{prop}\label{prop:dual-jacquet-shalika-formula-odd}
	\begin{equation*}
		\begin{split}
			\DualJS{\finiteFieldRepresentation}{\fieldCharacter}&\left(W, \phi\right) = \frac{1}{\grpIndex{G}{N} \grpIndex{M}{\UpperTriangularAdditive}\sizeof{\Mat{1}{m}{\finiteField}} } \sum_{g \in \lquot{\UnipotentSubgroup}{G}} \sum_{X \in \lquot{\UpperTriangularAdditive}{M}} \sum_{Z \in \Mat{1}{m}{\finiteField}} \\
			&W\left(
			\begin{pmatrix}& 1 \\ I_{2m} & \end{pmatrix}
			\oddPermutationMatrix \ShalikaUnipotentElementOdd{X}
			\ShalikaDiagonalElementOdd{g}
			\ShalikaUpperRightUnipotentElementOdd{-\transpose{Z}}
			\right)
			\fieldCharacter\left(-\trace X\right)
			\FourierTransformWithRespectToCharacter{\phi}{\fieldCharacter} \left( Z \right).
		\end{split}
	\end{equation*}
\end{prop}

\begin{prop}\label{prop:jacquet-shalika-integral-of-bessel-function-odd-case}
	Let $W = \grpIndex{G}{N} \grpIndex{M}{\UpperTriangularAdditive} \sizeof{\Mat{1}{m}{\finiteField}} \finiteFieldRepresentation \left( \oddPermutationMatrix^{-1} \right) \besselFunctionOfFiniteFieldRepresentation$,
	and $\phi = \indicatorFunction{0}$ be the indicator function of $0 \in \finiteField^m$. Then
	$\JSOfFiniteFieldRepresentation{W}{\phi} = 1.$
\end{prop}

\begin{defn}
	\textbf{The Shalika subgroup} $\ShalikaSubgroupOdd \le \GL{2m+1}{\finiteField}$ is defined as
	\begin{equation*}
		\ShalikaSubgroupOdd = \left\{ \begin{pmatrix}
			g & X & Y \\
			  & g &   \\
			  & Z & 1
		\end{pmatrix}
		\mid g \in \GL{m}{\finiteField}, \, X \in \SquareMat{m}{\finiteField}, \,
		Y \in \Mat{m}{1}{\finiteField}, \, Z \in \Mat{1}{m}{\finiteField}
		\right\}.
	\end{equation*}
	Let $\Mirabolic{2m+1} \le \GL{2m+1}{\finiteField}$ be the Mirabolic subgroup, i.e., the subgroup of matrices having $\rowvector{2m+1} = \begin{pmatrix}0 & \cdots & 0 & 1\end{pmatrix}$ as their last row. We define a character $\ShalikaCharacter : \ShalikaSubgroupOdd \cap \Mirabolic{2m+1} \rightarrow \multiplicativegroup{\cComplex}$ by
	$$\ShalikaCharacter \begin{pmatrix}
			g & X & Y \\
			  & g &   \\
			  &   & 1
		\end{pmatrix}
		= \fieldCharacter\left( \trace\left( X g^{-1} \right) \right).$$
	The Shalika subgroup acts on $\Schwartz \left(\finiteField^m\right)$ by the following relations \cite[Proposition 3.1]{CogdellMatringe15}
	\begin{itemize}
		\item $\ShalikaSubgroupAction \ShalikaDiagonalElementOdd{g} \phi \left(x\right) = \phi\left(xg\right)$.
		\item $\ShalikaSubgroupAction \ShalikaUnipotentElementOdd{z_0} \phi \left(x\right)  = \fieldCharacter\left(-\trace z_0\right) \phi \left(x\right)$.
		\item $\ShalikaSubgroupAction \ShalikaUpperRightUnipotentElementOdd{y_0} \phi \left(x\right) = \fieldCharacter\left(\standardForm{x}{y_0}\right) \phi \left(x\right) $.
		\item $\ShalikaSubgroupAction \ShalikaLowerUnipotentElementOdd{x_0} \phi \left(x\right) = \phi \left(x + x_0\right).$
	\end{itemize}
\end{defn}
As in \cite[Proposition 3.1]{CogdellMatringe15}, one has
\begin{prop}\label{prop:shalika-subgroup-odd-representation}
	The action of $\ShalikaSubgroupOdd$ on $\Schwartz \left(\finiteField^m\right)$ is equivalent to $\Ind{\Mirabolic{2m+1}\cap \ShalikaSubgroupOdd}{\ShalikaSubgroupOdd}{\ShalikaCharacter^{-1}}$ given by mapping $f \in \Ind{\Mirabolic{2m+1}\cap \ShalikaSubgroupOdd}{\ShalikaSubgroupOdd}{\ShalikaCharacter^{-1}}$ to $\phi \in \Schwartz \left(\finiteField^m\right)$, defined as $\phi \left(x\right) = f \SmallShalikaLowerUnipotentElementOdd{x}$.
\end{prop}

As in \cite[Lemma 3.2, 3.3]{CogdellMatringe15}, one easily checks that we have the following analog of \cref{prop:equivariance-of-jacquet-shalika-integral-even-case}:
\begin{prop}[Equivariance property]\label{prop:equivariance-of-jacquet-shalika-integral-odd-case}
	Let $B \in \left\{ \JS{\finiteFieldRepresentation}{\fieldCharacter}, \DualJS{\finiteFieldRepresentation}{\fieldCharacter} \right\} $. Then for every $s \in \ShalikaSubgroupOdd$, $$B\left({\finiteFieldRepresentation \left(s\right) W},{\ShalikaSubgroupAction \left(s\right) \phi}\right) = B\left({W},{\phi}\right).$$
\end{prop}

Since in the odd case, the character $\ShalikaCharacter$ is defined only on the subgroup $\ShalikaSubgroupOdd \cap \Mirabolic{2m+1}$, we don't have a definition for a Shalika vector in this case.

\subsection{The functional equation}

In this section, we prove the functional equation satisfied by the Jacquet-Shalika integral. This allows us to define the exterior square gamma factor of an irreducible cuspidal representation of $\GL{n}{\finiteField}$.

\begin{thm}[The functional equation]\label{thm:functional-equation-finite-field}
	Let $n = 2m$ or $n = 2m + 1$, and let $\finiteFieldRepresentation$ be an irreducible cuspidal representation of $\GL{n}{\finiteField}$. If $n$ is even, suppose that $\finiteFieldRepresentation$ does not admit a Shalika vector. Then there exists a non-zero constant $ \gammaFactorOfFiniteField \in \multiplicativegroup{\cComplex}$, such that for every $W \in \WhittakerModelOfFiniteFieldRepresentation$ and $\phi \in \Schwartz\left( \finiteField^m \right) $, we have
	$$\DualJS{\finiteFieldRepresentation}{\fieldCharacter}\left(W, \phi\right) = \gammaFactorOfFiniteField \cdot \JSOfFiniteFieldRepresentation{W}{\phi}.$$
\end{thm}

The proof of this theorem is based on the proofs for the local field case of the local Jacquet-Shalika integrals \cite[Section 4]{matringe2014linear}, \cite[Section 3.3]{CogdellMatringe15}. The Shalika subgroup $\ShalikaSubgroup{n}$ plays an important role in the proof. We treat the even case and the odd case separately. In both cases, the main idea is to show that the space of $\ShalikaSubgroup{n}$ equivariant bilinear forms $$B : \WhittakerModelOfFiniteFieldRepresentation \times \Schwartz\left(\finiteField^m\right) \rightarrow \cComplex$$ is at most one dimensional. Since $\JS{\finiteFieldRepresentation}{\fieldCharacter}$ and $\DualJS{\finiteFieldRepresentation}{\fieldCharacter}$ define non-zero elements of this space (see Propositions \ref{prop:jacquet-shalika-integral-of-bessel-function-even-case} and \ref{prop:equivariance-of-jacquet-shalika-integral-even-case} for the even case, and Propositions \ref{prop:jacquet-shalika-integral-of-bessel-function-odd-case} and \ref{prop:equivariance-of-jacquet-shalika-integral-odd-case} for the odd case), the theorem follows.

\begin{defn}\label{defn:exterior-square-gamma-factor-even}
	We call $\gammaFactorOfFiniteField$ in the above theorem the \textbf{exterior square gamma factor} of $\finiteFieldRepresentation$ with respect to the character $\fieldCharacter$.
\end{defn}

\begin{rem}\label{rem:exterior-square-gamma-factor-is-unitrary}
	By double duality (\cref{prop:double-duality-of-jacquet-shalika-integral-even} and \cref{prop:double-duality-of-jacquet-shalika-integral-odd}), we have that $$\gammaFactorOfFiniteField \cdot \gammaFactorOfContragredientFiniteField = 1.$$
	Substituting in the functional equation the functions from \cref{prop:jacquet-shalika-integral-of-bessel-function-even-case} and \cref{prop:jacquet-shalika-integral-of-bessel-function-odd-case} in the even case and the odd case respectively, and using the fact that $\besselFunction_{\Contragradient{\finiteFieldRepresentation}, \fieldCharacter^{-1}} = \conjugate{\besselFunctionOfFiniteFieldRepresentation}$, we get that $\conjugate{\gammaFactorOfFiniteField} = \gammaFactorOfContragredientFiniteField$, and therefore $\abs{\gammaFactorOfFiniteField} = 1$.
\end{rem}

\subsubsection{The even case}

The proof of the functional equation in the even case relies on the following lemma.

\begin{lem}\label{lem:multiplicity-one-shalika-subgroup-even}
	Let $\representationDeclaration{\finiteFieldRepresentation}$ be an irreducible cuspidal representation of $\GL{2m}{\finiteField}$. Then $$\dim_\cComplex{\Hom_{\ShalikaSubgroupEven \cap \Mirabolic{2m}}\left(\finiteFieldRepresentation, \ShalikaCharacter\right)} \le 1.$$
	Here $\Mirabolic{2m} \le \GL{2m}{\finiteField}$ is the mirabolic subgroup.
\end{lem}

\begin{proof}
	We define a homomorphism $$ \Lambda : \Hom_{\ShalikaSubgroupEven \cap \Mirabolic{2m}}\left(\finiteFieldRepresentation,\ShalikaCharacter\right) \rightarrow \Hom_{\LeviSubgroup{m,m} \cap \Mirabolic{2m}}\left(\finiteFieldRepresentation,1\right)$$ by $$\Lambda \left(L\right)\left(v\right) = \frac{1}{\sizeof{\GL{m}{\finiteField}}} \sum_{g \in \GL{m}{\finiteField}}{L\left( \finiteFieldRepresentation\begin{pmatrix}
					g &                    \\
					  & \IdentityMatrix{m}
				\end{pmatrix} v\right)},$$
	where $\LeviSubgroup{m,m} \le \GL{2m}{\finiteField}$ is the Levi subgroup corresponding to the partition $\left(m,m\right)$. We claim that $\Lambda$ is injective: suppose that $L \ne 0$ and that $v \in \underlyingVectorSpace{\finiteFieldRepresentation}$ satisfies $L\left(v\right) \ne 0$. We define for a function $\Phi : \SquareMat{m}{\finiteField} \rightarrow \cComplex$, a vector $v_\Phi \in \underlyingVectorSpace{\finiteFieldRepresentation}$ by
	$$v_\Phi = \frac{1}{\sqrt{\sizeof{\SquareMat{m}{\finiteField}}}}\sum_{X \in \SquareMat{m}{\finiteField}}{\Phi\left(X\right) \finiteFieldRepresentation \ShalikaUnipotentElement{X}}v.$$
	Then $$\Lambda\left(L\right)v_\Phi = \frac{1}{\sizeof{\GL{m}{\finiteField}}}\sum_{g \in \GL{m}{\finiteField}}{\FourierTransformWithRespectToCharacter{\Phi}{\fieldCharacter}\left(g\right)}L\left( \finiteFieldRepresentation\begin{pmatrix}
				g &                    \\
				  & \IdentityMatrix{m}
			\end{pmatrix} v\right).$$
	Choosing $\Phi$ such that $\FourierTransformWithRespectToCharacter{\Phi}{\fieldCharacter} = \indicatorFunction{\IdentityMatrix{m}}$, we get that $\Lambda\left(L\right)v_\Phi =
		\frac{1}{\sizeof{\GL{m}{\finiteField}}} v \ne 0$, and therefore $\Lambda \left( L \right) \ne 0$.
	Thus, $\Lambda$ is injective. The lemma then follows from the multiplicity one theorem below.
\end{proof}

\begin{thm}\label{thm:multiplicity-one-theorem-even-case}
	Let $\representationDeclaration{\finiteFieldRepresentation}$ be an irreducible cuspidal representation of $\GL{2m}{\finiteField}$. Then
	$$\dim_\cComplex{\Hom_{\LeviSubgroup{m,m} \cap \Mirabolic{2m}}\left(\finiteFieldRepresentation, 1\right)} \le 1.$$
	Here $\LeviSubgroup{m,m} \le \GL{2m}{\finiteField}$ is the Levi subgroup corresponding to the partition $\left(m,m\right)$.
\end{thm}
The proof of \cref{thm:multiplicity-one-theorem-even-case} will be given in the appendix, as it is a detour from the main line of the paper.  \cref{thm:multiplicity-one-theorem-even-case} is restated and proved in \cref{thm:multiplicity-one-even-case}, while the one for the odd case is \cref{thm:multiplicity-one-odd-case}.

\begin{proof}[Proof of \cref{thm:functional-equation-finite-field}, $n=2m$]
	The idea is to show that $\dim_\cComplex \Hom_{\ShalikaSubgroupEven}\left(\finiteFieldRepresentation \otimes \Schwartz\left(\finiteField^m\right), \ShalikaCharacter \right) \le 1$, and since $\JS{\finiteFieldRepresentation}{\fieldCharacter}$ and $\DualJS{\finiteFieldRepresentation}{\fieldCharacter}$ define non-zero elements of $\Hom_{\ShalikaSubgroupEven}\left(\finiteFieldRepresentation \otimes \Schwartz\left(\finiteField^m\right), \ShalikaCharacter \right)$, it follows that such a constant exists.

	We first claim that the restriction map $$\Hom_{\ShalikaSubgroupEven}\left(\finiteFieldRepresentation \otimes \Schwartz\left(\finiteField^m\right), \ShalikaCharacter \right) \rightarrow \Hom_{\ShalikaSubgroupEven}\left(\finiteFieldRepresentation \otimes \Schwartz\left(\finiteField^m \setminus \left\{0\right\} \right), \ShalikaCharacter \right)$$ is injective, where $\Schwartz\left(\finiteField^m \setminus \left\{0\right\} \right)$ is realized as a subspace of $\Schwartz \left(\finiteField^m\right)$ by the set of elements of $\Schwartz \left(\finiteField^m\right)$, vanishing at zero.

	Suppose that this restriction map is not injective. Then there exists $ 0 \ne B : \underlyingVectorSpace{\finiteFieldRepresentation} \times \Schwartz\left(\finiteField^m\right) \rightarrow \cComplex$ satisfying $$B\left(\finiteFieldRepresentation\left(s\right) v,\rightTranslation \left(s\right) \phi\right) = \ShalikaCharacter \left(s\right) B\left(v, \phi\right)$$ for every $s \in \ShalikaSubgroupEven$, such that $b \in \Contragradient{\underlyingVectorSpace{\finiteFieldRepresentation}}$ defined by $b\left(v\right)=B \left(v , \indicatorFunction{0} \right)$ is not the zero functional, where $\indicatorFunction{0}$ is the indicator function of $0$ in $\finiteField^m$. Let $\innerproduct{\cdot}{\cdot}$ be an inner product on $\underlyingVectorSpace{\finiteFieldRepresentation}$ with respect to which $\finiteFieldRepresentation$ is unitary, and let $0 \ne v_0 \in \underlyingVectorSpace{\finiteFieldRepresentation}$, such that $b \left( v \right) = \innerproduct{v}{v_0}$. Then since $\rightTranslation \left(s\right) \indicatorFunction{0} = \indicatorFunction{0}$ for every $s \in \ShalikaSubgroupEven$, it follows from the equivariance property of $B$ that
    $$b \left(\finiteFieldRepresentation\left(s\right)v \right) = B \left(\finiteFieldRepresentation\left(s\right) v, \indicatorFunction{0} \right) = B \left(\finiteFieldRepresentation\left(s\right) v, \rightTranslation\left(s\right) \indicatorFunction{0} \right) = \ShalikaCharacter \left(s\right) B \left(v, \indicatorFunction{0}\right) = \ShalikaCharacter \left(s\right)b \left(v\right).$$
    Thus, for any $s \in \ShalikaSubgroupEven$ and $v \in \underlyingVectorSpace{\finiteFieldRepresentation}$, we have
    $$\innerproduct{v}{\finiteFieldRepresentation\left(s\right)v_0} = \innerproduct{\finiteFieldRepresentation(s^{-1})v}{v_0} = b \left(\finiteFieldRepresentation(s^{-1})v\right) = \ShalikaCharacter(s^{-1})b(v) = \conjugate{\ShalikaCharacter\left(s\right)}\innerproduct{v}{v_0} = \innerproduct{v}{\ShalikaCharacter\left(s\right)v_0}.$$
    This shows that $v_0$ is a Shalika vector, which is a contradiction.

	Next we write a sequence of isomorphisms from which it follows that $$\dim_{\cComplex}{\Hom_{\ShalikaSubgroupEven}\left(\finiteFieldRepresentation \otimes \Schwartz\left(\finiteField^m \setminus \left\{0\right\} \right),  \ShalikaCharacter\right)} \le 1.$$
	Since $\lquot{\left(\ShalikaSubgroupEven \cap \Mirabolic{2m}\right)}{\ShalikaSubgroupEven} \isomorphic \finiteField^m \setminus \left\{ 0 \right\}$ by the map $\left(\begin{smallmatrix}
				g & X\\
				& g\end{smallmatrix}\right) \mapsto \lastrowvector g$, we have that $\Schwartz\left(\finiteField^m \setminus \left\{0\right\} \right) \isomorphic \Ind{\ShalikaSubgroupEven \cap \Mirabolic{2m}}{\ShalikaSubgroupEven}{1}$ and we have the following isomorphisms:

	\begin{equation*}
		\begin{split}
			\Hom_{\ShalikaSubgroupEven}\left(\finiteFieldRepresentation \otimes \Ind{\ShalikaSubgroupEven \cap \Mirabolic{2m}}{\ShalikaSubgroupEven}{1}, \ShalikaCharacter\right) &\isomorphic  \Hom_{\ShalikaSubgroupEven}\left(\ShalikaCharacter^{-1} \otimes \finiteFieldRepresentation, \reallywidetilde{\Ind{\ShalikaSubgroupEven \cap \Mirabolic{2m}}{\ShalikaSubgroupEven}{1}} \right) \\
			& \isomorphic \Hom_{\ShalikaSubgroupEven}\left(\ShalikaCharacter^{-1} \otimes \finiteFieldRepresentation, {\Ind{\ShalikaSubgroupEven \cap \Mirabolic{2m}}{\ShalikaSubgroupEven}{\Contragradient{1}}} \right).
		\end{split}
	\end{equation*}

	Thanks to Frobenius reciprocity, we get that $$\Hom_{\ShalikaSubgroupEven}\left(\ShalikaCharacter^{-1} \otimes \finiteFieldRepresentation, {\Ind{\ShalikaSubgroupEven \cap \Mirabolic{2m}}{\ShalikaSubgroupEven}{1}} \right) \isomorphic \Hom_{\ShalikaSubgroupEven \cap \Mirabolic{2m}}\left(\ShalikaCharacter^{-1} \otimes \finiteFieldRepresentation, 1 \right) \isomorphic \Hom_{\ShalikaSubgroupEven \cap \Mirabolic{2m}}\left( \finiteFieldRepresentation, \ShalikaCharacter \right),$$
	and by \cref{lem:multiplicity-one-shalika-subgroup-even}, the last space has dimension $\le 1$.
\end{proof}

\begin{rem}
	As seen in the proof, the proof fails if $\finiteFieldRepresentation$ admits a Shalika vector. In this case, a modified functional equation is valid. This is discussed in \cref{thm:modified-functional-equation}.
\end{rem}

\subsubsection{The odd case}
Similarly to the even case, the proof relies on the following lemma:

\begin{lem}\label{lem:multiplicity-one-shalika-subgroup-odd}
	Let $\representationDeclaration{\finiteFieldRepresentation}$ be an irreducible cuspidal representation of $\GL{2m+1}{\finiteField}$. Then
	$$\dim_{\cComplex}{\Hom_{\ShalikaSubgroupOdd \cap \Mirabolic{2m + 1}}} \left(\finiteFieldRepresentation, \ShalikaCharacter \right) \le 1,$$
	where $\Mirabolic{2m+1}$ is the mirabolic subgroup.
\end{lem}
\begin{proof}
	Let $\NonstandardLeviSubgroup{2 m + 1}$ be the following maximal (non-standard) Levi subgroup corresponding to the partition $\left(m+1, m\right)$:
	$$
		\NonstandardLeviSubgroup{2m + 1} = \left\{
		\begin{pmatrix}
			g_1 &     & u       \\
			    & g_2 &         \\
			v   &     & \lambda
		\end{pmatrix} \mid u \in \Mat{m}{1}{\finiteField}, v \in \Mat{1}{m}{\finiteField}, \lambda \in \finiteField, g_1, g_2 \in \GL{m}{\finiteField}
		\right\} \cap \GL{2m+1}{\finiteField}.
	$$
	We define a homomorphism $$\Lambda : \Hom_{\ShalikaSubgroupOdd \cap \Mirabolic{2m + 1}} \left(\finiteFieldRepresentation, \ShalikaCharacter \right) \rightarrow \Hom_{\NonstandardLeviSubgroup{2m + 1} \cap \Mirabolic{2m + 1}} \left(\finiteFieldRepresentation, 1 \right),$$ by
	$$ \Lambda\left(L\right)\left(v\right) = \frac{1}{\sizeof{\GL{m}{\finiteField}}}\sum_{g \in \GL{m}{\finiteField}}{L \left(\finiteFieldRepresentation\diagTwo{g}{\IdentityMatrix{m+1}} v \right)}.$$
	We claim that $\Lambda$ is injective. Let $L \ne 0$ and let $v \in \underlyingVectorSpace{\finiteFieldRepresentation}$ such that $L \left(v\right) \ne 0$. For a function $\Phi : \SquareMat{m}{\finiteField} \rightarrow \cComplex$, we define $$v_{\Phi} = \frac{1}{\sqrt{\sizeof{\SquareMat{m}{\finiteField}}}} \sum_{X \in \SquareMat{m}{\finiteField}}{\Phi \left(X\right) \finiteFieldRepresentation \ShalikaUnipotentElementOdd{X}}v.$$
	Then $$\Lambda\left(L\right)\left(v_\Phi\right) = \frac{1}{\sizeof{\GL{m}{\finiteField}}} \sum_{g \in \GL{m}{\finiteField}} \FourierTransformWithRespectToCharacter{\Phi}{\fieldCharacter} \left(g\right) L\left( \finiteFieldRepresentation\diagTwo{g}{\IdentityMatrix{m+1}}v \right).$$
	As in the even case, choosing $\Phi$ with $\FourierTransformWithRespectToCharacter{\Phi}{\fieldCharacter} = \indicatorFunction{\IdentityMatrix{m}}$, we get that $\Lambda \left(L\right) \left(v\right) \ne 0$, thus $\Lambda$ is injective. The lemma now follows from the following multiplicity one theorem.
\end{proof}
\begin{thm}\label{thm:multiplicity-one-theorem-odd-case}
	Let $\representationDeclaration{\finiteFieldRepresentation}$ be an irreducible cuspidal representation of $\GL{2m+1}{\finiteField}$. Then
	$$\dim_{\cComplex}\Hom_{\NonstandardLeviSubgroup{2m+1} \cap \Mirabolic{2m+1}}\left(\finiteFieldRepresentation, 1\right) \le 1.$$
\end{thm}
As \cref{thm:multiplicity-one-theorem-even-case}, the proof of \cref{thm:multiplicity-one-theorem-odd-case} will be given in the appendix, as it detours from the main line of the paper.

\begin{proof}[Proof of \cref{thm:functional-equation-finite-field}, $n=2m+1$]
	By \cref{prop:equivariance-of-jacquet-shalika-integral-even-case},
	$\JS{\finiteFieldRepresentation}{\fieldCharacter}, \DualJS{\finiteFieldRepresentation}{\fieldCharacter}$ define non-zero elements of $ \Hom_{\ShalikaSubgroupOdd}(\finiteFieldRepresentation \otimes \Schwartz(\finiteField^m), 1)$. Therefore, to conclude the theorem, it suffices to prove
	$$\dim \Hom_{\ShalikaSubgroupOdd}(\finiteFieldRepresentation \otimes \Schwartz(\finiteField^m), 1) \le 1.$$

	By \cref{prop:shalika-subgroup-odd-representation}, the action of $\ShalikaSubgroupOdd$ on $\Schwartz\left( \finiteField^m \right)$ is equivalent to the representation $\Ind{\ShalikaSubgroupOdd \cap \Mirabolic{2m+1}}{\ShalikaSubgroupOdd}{\ShalikaCharacter^{-1}}$.

	Using the same steps as in the even case, we have
	\begin{align*}
		\Hom_{\ShalikaSubgroupOdd}\left(\finiteFieldRepresentation \otimes \Ind{\ShalikaSubgroupOdd \cap \Mirabolic{2m+1}}{\ShalikaSubgroupOdd}{\ShalikaCharacter^{-1}}, 1\right) & \isomorphic \Hom_{\ShalikaSubgroupOdd}\left(\finiteFieldRepresentation, \reallywidetilde{\Ind{\ShalikaSubgroupOdd \cap \Mirabolic{2m+1}}{\ShalikaSubgroupOdd}{\ShalikaCharacter^{-1}}}\right)   \\
		                                                                                                                                                                          & \isomorphic \Hom_{\ShalikaSubgroupOdd}\left(\finiteFieldRepresentation, {\Ind{\ShalikaSubgroupOdd \cap \Mirabolic{2m+1}}{\ShalikaSubgroupOdd}{\Contragradient{\ShalikaCharacter^{-1}}}}\right),
	\end{align*}
	By Frobenius reciprocity, we have that the last space is isomorphic to $\Hom_{\ShalikaSubgroupOdd \cap \Mirabolic{2m + 1}}\left(\finiteFieldRepresentation, \ShalikaCharacter\right)$, and therefore by \cref{lem:multiplicity-one-shalika-subgroup-odd}, the dimension of this space is at most $1$.
\end{proof}

\subsubsection{Equivalent conditions for admitting a Shalika vector}

Let $\representationDeclaration{\finiteFieldRepresentation}$ be an irreducible cuspidal representation of $\GL{2m}{\finiteField}$. In this section we state equivalent conditions for $\finiteFieldRepresentation$ to admit a Shalika vector.

\begin{prop}
	The representation $\finiteFieldRepresentation$ admits a Shalika vector if and only if there exists $W \in \WhittakerModelOfFiniteFieldRepresentation$ such that $\JSOfFiniteFieldRepresentation{W}{1} \ne 0$, where $1$ denotes the constant function valued $1$ on $\finiteField^m$.
\end{prop}

\begin{proof}
	Suppose that there exists $W \in \WhittakerModelOfFiniteFieldRepresentation$ such that $\JSOfFiniteFieldRepresentation{W}{1} \ne 0$. Define $W_0 \in \WhittakerModelOfFiniteFieldRepresentation$ by
	$$ W_0 \left( h \right) =
		\frac{1}{\grpIndex{G}{\UnipotentSubgroup}} \frac{1}{\grpIndex{M}{\UpperTriangularAdditive}} \sum_{g \in \lquot{\UnipotentSubgroup}{G}}\sum_{X \in \lquot{\UpperTriangularAdditive}{M}} W \left( h \ShalikaUnipotentElement{X} \ShalikaDiagonalElement{g}
		\right)
		\fieldCharacter \left( -\trace X \right).$$ Then $W_0 \ne 0$ as $W_0 \left(\evenPermutationMatrix\right) = \JSOfFiniteFieldRepresentation{W}{1}$. By changing variables, we can show that
        $$W_0 \left( h s \right) = \ShalikaCharacter \left( s \right) W_0 \left( h \right),$$ for every $s \in \ShalikaSubgroupEven$,
        thus $W_0$ is a Shalika vector.

	For the other direction, assume $\finiteFieldRepresentation$ admits a non-zero Shalika vector $W_0 \in \WhittakerModelOfFiniteFieldRepresentation$. Choose an inner product $\innerproduct{\cdot}{\cdot}$ on $\WhittakerModelOfFiniteFieldRepresentation$, with respect to which $\finiteFieldRepresentation$ is unitary. Then $W_0$ defines a non-zero element $T_{W_0} \in \Hom_{\ShalikaSubgroupEven}\left(\finiteFieldRepresentation, \ShalikaCharacter\right)$ by $T_{W_0}\left(W'\right) = \innerproduct{W'}{W_0}$.
	We have that $$\Hom_{\ShalikaSubgroupEven}\left(\finiteFieldRepresentation, \ShalikaCharacter\right) \subseteq \Hom_{\ShalikaSubgroupEven \cap \Mirabolic{2m}}\left(\finiteFieldRepresentation, \ShalikaCharacter\right).$$ By \cref{lem:multiplicity-one-shalika-subgroup-even}, we have in this case that $0 \ne \Hom_{\ShalikaSubgroupEven}\left(\finiteFieldRepresentation, \ShalikaCharacter\right) = \Hom_{\ShalikaSubgroupEven \cap \Mirabolic{2m}}\left(\finiteFieldRepresentation, \ShalikaCharacter\right)$. We define $W \in \WhittakerModelOfFiniteFieldRepresentation$ by
	$$W \left(h\right) =
		\sum_{g \in \lquot{\UnipotentSubgroup}{P}}\sum_{X \in \lquot{\UpperTriangularAdditive}{M}}
		{
			\besselFunctionOfFiniteFieldRepresentation \left(h
			\ShalikaUnipotentElement{X}
			\ShalikaDiagonalElement{g}
			\evenPermutationMatrix^{-1}
			\right)
			\fieldCharacter \left( -\trace X \right)
		},$$
	where $P = \Mirabolic{m}\left(\finiteField\right) = \left\{ g \in \GL{m}{\finiteField} \mid \rowvector{m} g = \rowvector{m} \right\} $ is the mirabolic subgroup.
	We can show that for $s \in \ShalikaSubgroupEven \cap \Mirabolic{2m}$,
	$$W\left(hs\right) = \ShalikaCharacter \left( s \right)W \left(h\right).$$
	Therefore, $W$  defines an element  $T_W \in \Hom_{\ShalikaSubgroupEven \cap \Mirabolic{2m}}\left(\finiteFieldRepresentation, \ShalikaCharacter\right)$ by $T_W \left( W' \right) = \innerproduct{W'}{W}$. Thus $T_W \in \Hom_{\ShalikaSubgroupEven}\left(\finiteFieldRepresentation, \ShalikaCharacter\right)$. This implies that $W$ is a Shalika vector. We can also show that $\JSOfFiniteFieldRepresentation{W}{1} = W\left( \evenPermutationMatrix \right)$. By \cref{prop:jacquet-shalika-integral-of-bessel-function-even-case}, we have that $W \left(\evenPermutationMatrix\right) = 1$, hence $\JSOfFiniteFieldRepresentation{W}{1} = W\left( \evenPermutationMatrix \right) \ne 0$.
\end{proof}

\begin{cor}
	If $\finiteFieldRepresentation$ admits a Shalika vector, then the functional equation in \cref{thm:functional-equation-finite-field} does not hold.
\end{cor}
\begin{proof}
	Let $W \in \WhittakerModelOfFiniteFieldRepresentation$, such that $\JSOfFiniteFieldRepresentation{W}{1} = 1$. Then since $\FourierTransformWithRespectToCharacter{1}{\fieldCharacter} = q^{\frac{m}{2}} \cdot \indicatorFunction{0}$, we have that $\DualJSOfFiniteFieldRepresentation{W}{1} = 0$, and therefore no such non-zero constant exists.
\end{proof}

Suppose that $\finiteFieldRepresentation$ is associated with the regular character $ \cuspidalCharacter : \multiplicativegroup{\finiteFieldExtension{2m}} \rightarrow  \multiplicativegroup{\cComplex}$ \cite{green1955characters}. We now recall the work of Prasad \cite{prasad2000space} in order to classify when $\finiteFieldRepresentation$ admits a Shalika vector, in terms of $\cuspidalCharacter$.

\newcommand{\prasadRepresentation}{\finiteFieldRepresentation_{\UnipotentRadical{m,m}, \ShalikaCharacter}}

Let $\UnipotentRadical{m,m}$ be the unipotent radical of $\GL{2m}{\finiteField}$ corresponding to the partition $\left(m,m\right)$. Denote by $\underlyingVectorSpace{\prasadRepresentation}$ the twisted Jacquet submodule of $\UnipotentRadical{m,m}$ with respect to the character $\ShalikaCharacter$:
$$
	\underlyingVectorSpace{\prasadRepresentation} = \left\{ v \in \underlyingVectorSpace{\finiteFieldRepresentation} \mid \forall X \in \SquareMat{m}{\finiteField}, \, \finiteFieldRepresentation \ShalikaUnipotentElement{X} v = \fieldCharacter \left(\trace X\right) v  \right\}.
$$
Then $\underlyingVectorSpace{\prasadRepresentation}$ is a subspace invariant under the action $\prasadRepresentation$ of $\GL{m}{\finiteField}$, defined by $\prasadRepresentation \left(g\right) v = \finiteFieldRepresentation \SmallShalikaDiagonalElement{g} v$, where $g \in \GL{m}{\finiteField}$, $v \in \underlyingVectorSpace{\prasadRepresentation}$. Prasad shows
\begin{thm}[{\cite[Theorem 1]{prasad2000space}}]
	$ \prasadRepresentation \isomorphic \Ind{\multiplicativegroup{\finiteFieldExtension{m}}}{\GL{m}{\finiteField}}{\cuspidalCharacter \restriction_{\multiplicativegroup{\finiteFieldExtension{m}}}}.$
\end{thm}
From the definition of a Shalika vector and the definition of the representation $\prasadRepresentation$, we see that $\finiteFieldRepresentation$ admits a Shalika vector if and only if $\Hom_{\GL{m}{\finiteField}}\left(1, \prasadRepresentation\right) \ne 0$. From Frobenius reciprocity we have that $\Hom_{\GL{m}{\finiteField}}\left(1, \prasadRepresentation\right) \isomorphic \Hom_{\multiplicativegroup{\finiteFieldExtension{m}}}\left(1, \cuspidalCharacter\restriction_{\multiplicativegroup{\finiteFieldExtension{m}}} \right)$, and the last space is non-zero if and only if $\cuspidalCharacter \restriction_{\multiplicativegroup{\finiteFieldExtension{m}}} = 1$, and then it is one dimensional. Therefore we get the following
\begin{cor}
	$\finiteFieldRepresentation$ admits a Shalika vector if and only if $\cuspidalCharacter \restriction_{\multiplicativegroup{\finiteFieldExtension{m}}} = 1$, and in this case, the space of Shalika vectors is one dimensional.
\end{cor}

We conclude this section with a theorem.

\begin{thm}[Equivalent conditions for admitting a Shalika vector]\label{thm:equivalent-conditions-for-a-shalika-vector}
	Suppose $\finiteFieldRepresentation$ is an irreducible cuspidal representation of $\GL{2m}{\finiteField}$ associated with the regular character $$ \theta : \multiplicativegroup{\finiteFieldExtension{2m}} \rightarrow \multiplicativegroup{\cComplex}.$$
	The following are equivalent.
	\begin{enumerate}
		\item $\finiteFieldRepresentation$ admits a Shalika vector.
		\item There exists $W \in \WhittakerModelOfFiniteFieldRepresentation$, such that $ \JSOfFiniteFieldRepresentation{W}{1} \ne 0$.
		\item $\cuspidalCharacter \restriction_{\multiplicativegroup{\finiteFieldExtension{m}}} = 1$.
	\end{enumerate}
	Moreover, in these cases, the space of Shalika vectors is one-dimensional.
\end{thm}

\subsection{An expression for the exterior square gamma factor}\label{subsec:melin-transform}

Let $\finiteFieldRepresentation$ be an irreducible cuspidal representation of $\GL{n}{\finiteField}$. If $n$ is even, suppose that $\finiteFieldRepresentation$ does not admit a Shalika vector.

In this section we express the exterior square gamma factor of $\finiteFieldRepresentation$ in terms of the Bessel function. One may use \cref{prop:jacquet-shalika-integral-of-bessel-function-even-case} and \cref{prop:jacquet-shalika-integral-of-bessel-function-odd-case} in order to get such an expression, but the Jacquet-Shalika integral runs over too many elements, some of which are not in the support of the Bessel function. We find a more accurate expression which involves only elements of the form of \cref{thm:support-of-bessel-function}.

Our main results of this section are the following theorems:

\begin{thm}\label{thm:mellin-inverse-even-case}
	If $n = 2m$, the exterior square gamma factor is given by the formula
	\begin{equation*}\begin{split}\gammaFactorOfFiniteField = q^{-\frac{m}{2} + 2\binom{m}{2}} \sum_{\substack{
			m_1, \dots, m_r \ge 1\\
			m_1 + \dots + m_r = m\\
			\lambda_1, \dots, \lambda_r \in \multiplicativegroup{\finiteField}
		}} & q^{-\sum_{i=1}^{r}{2 \binom{m_i}{2}}} \cdot {\besselFunctionOfFiniteFieldRepresentation} \left(\antidiag\left(\lambda_1 \IdentityMatrix{2 m_1}, \dots, \lambda_r \IdentityMatrix{2 m_r}\right)^{-1}\right)  \cdot \\
		& \cdot \fieldCharacter\left( \lambda_{r} \cdot \indicatorFunction{m_r,1} \right),\end{split}\end{equation*} where $\indicatorFunction{m_r,1} = \begin{cases}
			1 & m_r = 1,   \\
			0 & m_r \ne 1.
		\end{cases}$
\end{thm}

\begin{thm}\label{thm:mellin-inverse-odd-case}
	If $n = 2m + 1$, the exterior square gamma factor is given by the formula
	$$
		\gammaFactorOfFiniteField = q^{\frac{m}{2}+2 \binom{m}{2}} \sum_{\substack{
				m_1, \dots, m_r \ge 1\\
				m_1 + \dots + m_r = m\\
				\lambda_1, \dots, \lambda_r \in \multiplicativegroup{\finiteField}
			}}
		q^{-\sum_{i=1}^{r}{2 \binom{m_i}{2}}} \cdot {\besselFunctionOfFiniteFieldRepresentation} \left(\antidiag\left(\lambda_1 \IdentityMatrix{2 m_1}, \dots, \lambda_r \IdentityMatrix{2 m_r}, 1 \IdentityMatrix{1}\right)^{-1}\right).
	$$
\end{thm}

The following lemma indicates which representatives for $g$ and $X$ in the Jacquet-Shalika integral contribute to the sum. It is a key to the proofs of \cref{thm:mellin-inverse-even-case} and \cref{thm:mellin-inverse-odd-case}.

\begin{lem}\label{lem:support-of-bessel-function-jacquet-shalika-integral}
	Let $g \in \GL{m}{\finiteField}$ and $X \in \NilpotentLowerTriangular_m \left({\finiteField}\right) $ a lower triangular nilpotent matrix (i.e., a lower triangular matrix with zeros on its diagonal). Suppose that $g =  w d u $, where $u \in \UnipotentRadical{m}\left( \finiteField \right)$, $w$ is a permutation matrix and $d$ is a diagonal matrix. Denote by $\tau$ the permutation defined by the columns of $w$. Write $X = \left(x_{i j}\right)$. Suppose that 
	\begin{equation}\label{eqn:assumption-of-shape}
	\evenPermutationMatrix \ShalikaUnipotentElement{X}		\ShalikaDiagonalElement{g}
		\evenPermutationMatrix^{-1} \in \UnipotentSubgroup_{2m} \antidiag \left(\lambda_1 \IdentityMatrix{n_1}, \dots, \lambda_r \IdentityMatrix{n_r}  \right) \UnipotentSubgroup_{2m},
	\end{equation}
	where $\lambda_1,\dots,\lambda_r \in \multiplicativegroup{\finiteField}$, and $n_1 + \dots + n_r = 2m$. Then
	\begin{enumerate}
		\item $wd = \antidiag\left(\lambda_1 \IdentityMatrix{m_1}, \dots, \lambda_r \IdentityMatrix{m_r}\right)$, where $n_i = 2m_i$, for every $1 \le i \le r$ (and therefore $m_1 + \dots + m_r = m$).
		\item $x_{ij} = 0$ for every $\left(i,j\right)$ satisfying $j < i$ and $\tau^{-1}\left(j\right) < \tau^{-1}\left(i\right)$.
	\end{enumerate}
	Furthermore, in this case, $\evenPermutationMatrix \SmallShalikaUnipotentElement{X}		\SmallShalikaDiagonalElement{g}
		\evenPermutationMatrix^{-1} = \antidiag \left(\lambda_1 \IdentityMatrix{2 m_1}, \dots, \lambda_r \IdentityMatrix{2 m_r}  \right) \cdot v$, where $v \in \UnipotentSubgroup_{2m}(\finiteField)$ is an upper triangular unipotent matrix with zeros right above its diagonal.
\end{lem}

\begin{proof}
	Since $\evenPermutationMatrix \SmallShalikaDiagonalElement{u} \evenPermutationMatrix^{-1} \in \UnipotentSubgroup_{2m}$, we have by \cref{eqn:assumption-of-shape} that $$\evenPermutationMatrix \ShalikaUnipotentElement{X} \ShalikaDiagonalElement{wd} \evenPermutationMatrix^{-1} \in \UnipotentSubgroup_{2m} \antidiag\left(\lambda_1 \IdentityMatrix{n_1}, \dots, \lambda_r \IdentityMatrix{n_r} \right) \UnipotentSubgroup_{2m}.$$

	Let $w' = \evenPermutationMatrix \SmallShalikaDiagonalElement{w} \evenPermutationMatrix^{-1}$. Then $w'$ is a column permutation matrix of the permutation $\tau'$, where $\tau'\left(2j\right) = 2\tau\left(j\right)$, and $\tau' \left(2j - 1\right) = 2 \tau \left(j\right) - 1$, for every $1 \le j \le m$, therefore $$w' = \begin{pmatrix}
			\standardColumnVector{2 \tau\left(1\right) - 1} & \standardColumnVector{2 \tau\left(1\right)} & \dots & \standardColumnVector{2 \tau\left(m\right) - 1} & \standardColumnVector{2 \tau\left(m\right)}
		\end{pmatrix},$$
	where $\standardColumnVector{i}$ is the $i$-th standard column vector in $\finiteField^{2m}$.

	Let $d = \diag \left(a_1, \dots, a_m\right)$ and  $d' = \evenPermutationMatrix \SmallShalikaDiagonalElement{d} \evenPermutationMatrix^{-1}$. Then $d' = \diag\left(a_1, a_1, a_2, a_2, \dots, a_m, a_m\right)$.

	Let $U_X = \evenPermutationMatrix \SmallShalikaUnipotentElement{X} \evenPermutationMatrix^{-1}$. Since $X = \left(x_{i j}\right)$ is a lower triangular nilpotent matrix, $U_X$ is a lower triangular unipotent matrix, and its columns are given by $\columnOf_{2j-1}\left(U_X\right) = \standardColumnVector{2j - 1}$ and $$ \columnOf_{2j}\left(U_X\right) = \standardColumnVector{2j} + \sum_{i = j + 1}^{m}{x_{i j} \standardColumnVector{2i - 1}}. $$

	Let $Z = \evenPermutationMatrix \SmallShalikaUnipotentElement{X} \SmallShalikaDiagonalElement{wd} \evenPermutationMatrix^{-1} = U_X w' d'$. We get that $\columnOf_{2j-1} \left(Z \right) = a_j \standardColumnVector{2 \tau\left(j \right) - 1}$ and $$ \columnOf_{2j}\left(Z\right) = a_j \standardColumnVector{2\tau \left(j\right)} + \sum_{i = \tau \left(j\right) + 1}^{m}{a_j x_{i \tau\left(j\right)} \standardColumnVector{2i - 1}}. $$

	We claim that $Z$ has the form $$ Z = \begin{pmatrix}
			0                                  & 0                                        & \dots   & 0                                & \lambda_1 \IdentityMatrix{2 m_1} \\
			0                                  & 0                                        & \dots   & \lambda_2 \IdentityMatrix{2 m_2} & \ast                             \\
			0                                  & 0                                        & \iddots & \ast                             & \ast                             \\
			0                                  & \lambda_{r-1} \IdentityMatrix{2 m_{r-1}} & \cdots  & \ast                             & \ast                             \\
			\lambda_{r} \IdentityMatrix{2 m_r} & \ast                                     & \cdots  & \ast                             & \ast                             \\
		\end{pmatrix},$$
	where $2m_i = n_i$, for every $1 \le i \le r$.

	The proof of the claim is by applying row and column reduction carefully, in order to obtain the diagonal and permutation matrices that are involved in the Bruhat decomposition of $Z$. Every $h \in \GL{n}{\finiteField}$ has a Bruhat decomposition $h = u_1 w_h d_h u_2$, where $u_1, u_2 \in \UnipotentSubgroup_n(\finiteField)$, $w_h$ is a permutation matrix and $d_h$ is a diagonal matrix. Such $w_h$ and $d_h$ are unique. Throughout the text, we refer to the uniqueness of $w_h$ and $d_h$ by the \emph{uniqueness of the Bruhat decomposition}.
	We will first show that $Z$ is of the form $$Z = \begin{pmatrix}
			0                                & \ast \\
			\lambda_r \IdentityMatrix{2 m_r} & \ast
		\end{pmatrix}.$$
	By the above description of the columns of $Z$, we have that the first column of $Z$ is $a_1 \cdot e_{2 \tau \left(1\right) - 1}$. Let $2m - 2\tau\left(1\right) + 2 = 2m_r$. We show by induction that for every $1 \le l \le 2m_r$, $\columnOf_{l}\left(Z\right) = \lambda_r \standardColumnVector{2m - 2m_r + l}$. For $l = 1$, we have $\columnOf_{1}\left(Z\right) = a_1 \standardColumnVector{2m - 2m_r + 1}$. By assumption $Z \in \UnipotentSubgroup_{2m} \antidiag \left(\lambda_1 \IdentityMatrix{n_1}, \dots, \lambda_r \IdentityMatrix{n_r}  \right) \UnipotentSubgroup_{2m}$ and therefore by the uniqueness of the Bruhat decomposition of $Z$, we must have $a_1 = \lambda_r$. Assume that the claim is true for all columns before $l$, that is to say, $\columnOf_{i}\left(Z\right) = \lambda_r \standardColumnVector{2m - 2m_r + i}$ for $1 \le i < l$. Since $\columnOf_{l-1}\left(Z\right) = \lambda_r \standardColumnVector{2m - 2m_r + l-1}$, we expect $\columnOf_{l}\left(Z\right) = \lambda_r \standardColumnVector{2m - 2m_r + l}$.

	If $l = 2j-1$ is odd, then since $\columnOf_{2j-1} \left(Z \right) = a_j \standardColumnVector{2 \tau\left(j \right) - 1}$, in order for $Z$ to have the required Bruhat decomposition $Z = u_1 \antidiag \left(\lambda_1 \IdentityMatrix{n_1}, \dots, \lambda_r \IdentityMatrix{n_r}  \right) u_2$ for some $u_1, u_2 \in \UnipotentSubgroup_{2m}(\finiteField)$, we must have that $a_j = \lambda_r$ and $2 \tau\left(j \right) - 1 = 2m - 2m_r + 2j - 1$, i.e., $\tau \left(j\right) = m - m_r + j$.

	If $l = 2j$ is even, then we have that $$ \columnOf_{2j}\left(Z\right) = \lambda_r \standardColumnVector{2m - 2 m_r + 2 j} + \sum_{i = m - m_r + j + 1}^{m}{\lambda_r x_{i, m - m_r + j} \standardColumnVector{2i - 1}}.$$
	If any of the $x_{i, m - m_r + j}$ is non-zero, then by applying row reduction (by multiplying from the left by an upper triangular unipotent matrix that annihilates all the elements above the lowest element of $\columnOf_{2j}\left(Z\right)$), we get that the Bruhat decomposition of $Z$ is not of the required form. Therefore we get that $\columnOf_{2j}\left(Z\right) = \lambda_r \standardColumnVector{2m - 2 m_r + 2 j}$, as required.

	We have shown that for every $1 \le l \le 2m_r$, $\columnOf_{l}\left(Z\right) = \lambda_r \standardColumnVector{2m - 2m_r + l}$. By the uniqueness of the Bruhat decomposition of $Z$, we must have $n_r = 2m_r$.

	We now have that $$Z = \begin{pmatrix}
			0                                & Z' \\
			\lambda_r \IdentityMatrix{2 m_r} & A
		\end{pmatrix},$$ where $Z' \in \GL{2m - 2m_r}{\finiteField}$ and $A \in \Mat{2 m_r}{\left( 2m - 2m_r \right)}{\finiteField} $.
	We have that $$Z \cdot \begin{pmatrix}
			\IdentityMatrix{2m_r} & -\lambda_r^{-1} A          \\
			0                     & \IdentityMatrix{2m - 2m_r}
		\end{pmatrix} = \begin{pmatrix}
			0                                & Z' \\
			\lambda_r \IdentityMatrix{2 m_r} & 0
		\end{pmatrix}.$$
	By uniqueness of the Bruhat decomposition of $Z$, we must have that $$Z' \in \UnipotentSubgroup_{2m-2m_r} \antidiag \left(\lambda_1\IdentityMatrix{n_1}, \dots, \lambda_{n_{r-1}} \IdentityMatrix{n_{r - 1}} \right) \UnipotentSubgroup_{2m-2m_r}.$$ Since $Z'$ inherits an analogous column description from $Z$ but is of smaller size than that of $Z$, we get by induction on the size of the matrix, i.e., by repeating the above steps applied to $Z$, that $$ Z' = \begin{pmatrix}
			0                                        & 0                                        & \dots   & 0                                & \lambda_1 \IdentityMatrix{2 m_1} \\
			0                                        & 0                                        & \dots   & \lambda_2 \IdentityMatrix{2 m_2} & \ast                             \\
			0                                        & 0                                        & \iddots & \ast                             & \ast                             \\
			0                                        & \lambda_{r-2} \IdentityMatrix{2 m_{r-2}} & \cdots  & \ast                             & \ast                             \\
			\lambda_{r-1} \IdentityMatrix{2 m_{r-1}} & \ast                                     & \cdots  & \ast                             & \ast                             \\
		\end{pmatrix},$$ and that $n_i = 2m_i$ for every $1 \le i \le r -1$, and therefore $Z$ has the desired form.

	By the above claim, we get that the Bruhat decomposition of $Z$ is $Z = w_Z d_Z u_Z$, where $w_Z d_Z = \antidiag\left(\lambda_1 \IdentityMatrix{2 m_1}, \dots, \lambda_r \IdentityMatrix{2 m_r} \right)$ ($w_Z$ is a permutation matrix, $d_Z$ is a diagonal matrix, and $u_Z$ is an upper triangular unipotent matrix). We conclude from $\evenPermutationMatrix \SmallShalikaDiagonalElement{wd} \evenPermutationMatrix^{-1} = w_Z d_Z$ that $w d = \antidiag \left( \lambda_1 \IdentityMatrix{m_1}, \dots, \lambda_{r} \IdentityMatrix{m_r} \right)$, which is the first part of the lemma. We also conclude from the induction process of the above claim that $x_{\tau \left(i\right) \tau \left(j\right)} = 0$ for every $i > j$ with $\tau' \left(2i - 1 \right) > \tau'\left(2j\right)$, which is equivalent to $x_{i j} = 0$ for every $i > j$ with $\tau^{-1} \left(i \right) > \tau^{-1} \left(j \right)$. This finishes the second part of the lemma.

	Finally, we write $Z = w' d' \left(d'^{-1} w'^{-1} U_X w' d'\right) $. We claim that $w'^{-1} U_X w'$ is an upper triangular matrix with zeros right above its diagonal: the only non-zero non-diagonal components of $U_X$ are located in the positions of the form $\left(2i-1, 2j\right)$ with values $x_{i j}$ for $j < i$, and these move in the conjugation $w'^{-1} U_X w'$ to $\left(\tau'^{-1}\left(2i-1\right), \tau'^{-1}\left(2j\right)\right) = \left(2\tau^{-1}\left(i\right) - 1, 2\tau^{-1}\left(j\right)\right)$. If $2\tau^{-1}\left(j\right)<2\tau^{-1} \left(i \right) - 1$, then we get that $x_{ij} = 0$, and therefore $d'^{-1} w'^{-1} U_X w' d'$ is an upper unipotent matrix with zeros right above its diagonal. Since $$\evenPermutationMatrix \ShalikaUnipotentElement{X} \ShalikaDiagonalElement{g} \evenPermutationMatrix^{-1} = Z \evenPermutationMatrix \ShalikaDiagonalElement{u} \evenPermutationMatrix^{-1},$$ and since $\evenPermutationMatrix \SmallShalikaDiagonalElement{u} \evenPermutationMatrix^{-1}$ is an upper triangular unipotent matrix with zeros right above its diagonal, it follows that $\evenPermutationMatrix \SmallShalikaUnipotentElement{X}		\SmallShalikaDiagonalElement{g}
		\evenPermutationMatrix^{-1} = \antidiag \left(\lambda_1 \IdentityMatrix{2 m_1}, \dots, \lambda_r \IdentityMatrix{2 m_r}  \right) \cdot v$, for some upper triangular unipotent matrix $v \in \UnipotentSubgroup_{2m}(\finiteField)$ with zeros right above its diagonal.

\end{proof}

We need one more lemma, regarding the number of elements involved in the Jacquet-Shalika integral, such that $g$ has a given Bruhat decomposition. As we'll see in the proofs of \cref{thm:mellin-inverse-even-case} and \cref{thm:mellin-inverse-odd-case}, in order for a coset $g \in \lquot{\UnipotentSubgroup}{G}$ to contribute to the Jacquet-Shalika integral, we must have $g \in  \UnipotentSubgroup \antidiag\left(\lambda_1 \IdentityMatrix{m_1}, \dots, \lambda_r \IdentityMatrix{m_r} \right) \UnipotentSubgroup$. Let $wd = \antidiag\left(\lambda_1 \IdentityMatrix{m_1}, \dots, \lambda_r \IdentityMatrix{m_r} \right)$. Given $g \in \UnipotentSubgroup w d \UnipotentSubgroup$ and $X \in \NilpotentLowerTriangular$ as in \cref{lem:support-of-bessel-function-jacquet-shalika-integral}, we'll see in the following proofs that the summand of the Jacquet-Shalika integral on a special choice of functions, depends only on $w d$. To evaluate the Jacquet-Shalika integral, we should count the number of cosets in the set $\left\{N w d u \mid u \in \UnipotentSubgroup \right\} \subseteq \lquot{\UnipotentSubgroup}{G}$ and the number of options for a matrix $X \in \NilpotentLowerTriangular$ satisfying the condition of \cref{lem:support-of-bessel-function-jacquet-shalika-integral}.

\begin{lem}\label{lem:jacquet-shalika-support-orbits-size}
	Let $g = wd = \antidiag \left( \lambda_1 \IdentityMatrix{m_1}, \dots, \lambda_r \IdentityMatrix{m_r}  \right)$, where $w = \antidiag \left( \IdentityMatrix{m_1}, \dots, \IdentityMatrix{m_r}  \right)$ is a permutation matrix, $d = \diag \left(\lambda_r \IdentityMatrix{m_r}, \dots, \lambda_1 \IdentityMatrix{m_1}\right)$ is a diagonal matrix, $m_1 + \dots + m_r = m$, $\lambda_1, \dots, \lambda_r \in \multiplicativegroup{\finiteField}$.
	\begin{enumerate}
		\item Consider the right action of the upper triangular unipotent subgroup $\UnipotentSubgroup = \UnipotentSubgroup_m$ on $\lquot{\UnipotentSubgroup}{G}$, where $G = \GL{m}{\finiteField}$. Then the orbit of $\UnipotentSubgroup g$ is of size $q^{\binom{m}{2} - \sum_{i=1}^{r}{\binom{m_i}{2}}}$.
		\item Let $\tau$ be the permutation corresponding to columns of $w$. Then the set $$\left\{X \in \NilpotentLowerTriangular_m  \mid x_{ij}=0, \, \forall 1 \le j < i \le m, \text{ s.t. } \tau^{-1}\left(j\right) < \tau^{-1}\left(i\right) \right\}$$ is of cardinality $q^{\binom{m}{2} - \sum_{i=1}^{r}{\binom{m_i}{2}}}$.
	\end{enumerate}
\end{lem}
\begin{proof}
	By the orbit-stabilizer theorem, we have that the orbit of $\UnipotentSubgroup w d$ is of size $ \grpIndex{\UnipotentSubgroup}{\Stab_\UnipotentSubgroup \left(N wd\right)} $. We have that $$\Stab_\UnipotentSubgroup \left(N wd\right) = \UnipotentSubgroup \cap w^{-1} \UnipotentSubgroup w = \left\{ \left( u_{ij} \right) \in \UnipotentSubgroup \mid u_{ij} = 0,\, \forall i<j \text{ s.t. } \tau^{-1}\left(i\right) > \tau^{-1}\left(j\right) \right\}.$$
	Therefore $$\log_q{\sizeof{\Stab_\UnipotentSubgroup \left(N wd\right)}} = \sizeof{\left\{ \left(i,j\right) \mid i<j  \text{ and } \tau^{-1}\left(i\right) < \tau^{-1}\left(j\right) \right\}} = \sum_{i=1}^r{\binom{m_i}{2}},$$
	and the first part is proved. The second part follows from the fact that
	$$\sizeof{\left\{ \left(i,j\right) \mid j<i  \text{ and } \tau^{-1}\left(j\right) < \tau^{-1}\left(i\right) \right\}} = \sum_{i=1}^r{\binom{m_i}{2}}.$$
\end{proof}

We are now ready to prove \cref{thm:mellin-inverse-even-case} and \cref{thm:mellin-inverse-odd-case}.

\begin{proof}[Proof of \cref{thm:mellin-inverse-even-case}]
	We have that $ \FourierTransformWithRespectToCharacter{\indicatorFunction{-\firstrowvector}}{\fieldCharacter}\left( x \right) = q^{-\frac{m}{2}} \fieldCharacter \left( \standardForm{-x}{\firstrowvector} \right) = q^{-\frac{m}{2}} \fieldCharacter \left( -x_1 \right) $ and therefore by the Fourier inversion formula, if $\phi \left(x\right) = \fieldCharacter \left(-x_1\right)$, then $\FourierTransformWithRespectToCharacter{\phi}{\fieldCharacter} = q^{\frac{m}{2}} \indicatorFunction{\firstrowvector}$.

	We compute $\DualJSOfFiniteFieldRepresentation{W}{\phi}$, for $W = \grpIndex{G}{\UnipotentSubgroup} \grpIndex{M}{\UpperTriangularAdditive} \finiteFieldRepresentation \left( \evenPermutationMatrix^{-1} \right) \besselFunctionOfFiniteFieldRepresentation$ and $\phi \left(x\right) = \fieldCharacter \left(-x_1\right) $. We have from \cref{prop:dual-jacquet-shalika-formula-even} that
	$$ q^{-\frac{m}{2}} \DualJSOfFiniteFieldRepresentation{W}{\phi} = \sum_{X \in \lquot{\UpperTriangularAdditive}{M}} \sum_{g \in \lquot{\UnipotentSubgroup}{G}} \besselFunctionOfFiniteFieldRepresentation \left( \evenPermutationMatrix \ShalikaUnipotentElement{X}	\ShalikaDiagonalElement{g} \evenPermutationMatrix^{-1}
		\right) \fieldCharacter\left(-\trace X\right) \indicatorFunction{\firstrowvector} \left( \firstrowvector \inverseTranspose{g} \right).$$
	Notice that $ \firstrowvector \inverseTranspose{g} = \firstrowvector $ if and only if $g$ has $\standardColumnVector{1} = \transpose{\firstrowvector}$ as its first column. It follows now (similarly to the proof of \cref{prop:jacquet-shalika-integral-of-bessel-function-even-case}) that $  \DualJSOfFiniteFieldRepresentation{W}{\phi} = q^{\frac{m}{2}} $, and therefore
	\begin{equation}\label{eqn:integral-representation-of-gamma-inverse}
	\begin{split}
	q^{\frac{m}{2}} \cdot \gammaFactorOfFiniteField^{-1} &=  \JSOfFiniteFieldRepresentation{W}{\phi} \\
	&= \sum_{g \in \lquot{\UnipotentSubgroup}{G}} \sum_{X \in \lquot{\UpperTriangularAdditive}{M}} \besselFunctionOfFiniteFieldRepresentation \left( \evenPermutationMatrix \ShalikaUnipotentElement{X} \ShalikaDiagonalElement{g} \evenPermutationMatrix^{-1}
		\right) \fieldCharacter\left(-\trace X\right) \fieldCharacter \left( -g_{m 1} \right).
	\end{split}
	\end{equation}

	Since the Jacquet-Shalika integral runs over cosets of the form $g \in \lquot{\UnipotentSubgroup}{G}$ (and is constant on these), it follows from the Bruhat decomposition that it suffices to consider elements of the form $g = wdu$, where $w$ is a permutation matrix, $d$ is a diagonal matrix and $u$ is an upper triangular unipotent matrix. By \cref{lem:support-of-bessel-function-jacquet-shalika-integral}, we only need to consider $w$, $d$ such that $wd = \antidiag\left(\lambda_1 \IdentityMatrix{m_1}, \dots, \lambda_r \IdentityMatrix{m_r} \right)$.

	We get from \cref{lem:support-of-bessel-function-jacquet-shalika-integral} that $$ \besselFunctionOfFiniteFieldRepresentation \left( \evenPermutationMatrix \ShalikaUnipotentElement{X} \ShalikaDiagonalElement{wdu} \evenPermutationMatrix^{-1} \right) = \besselFunctionOfFiniteFieldRepresentation \left( \antidiag \left( \lambda_1 \IdentityMatrix{2m_1}, \dots, \lambda_r \IdentityMatrix{2m_r} \right) \right).$$
	Implicitly, we see that it does not depend on $X$ and $u$. By \cref{lem:jacquet-shalika-support-orbits-size}, given such $wd$, we have $q^{\binom{m}{2} - \sum_{i=1}^{r}{\binom{m_i}{2}}}$ options for $u \in \UnipotentSubgroup$, and the same number of options for $X \in \NilpotentLowerTriangular$.

	Therefore we get the following formula $$\JSOfFiniteFieldRepresentation{W}{\phi} = \sum_{\substack{
				m_1, \dots, m_r \ge 1\\
				m_1 + \dots + m_r = m\\
				\lambda_1,\dots,\lambda_r \ \in \multiplicativegroup{\finiteField}
			}} q^{2\left( \binom{m}{2} - \sum_{i=1}^{r} \binom{m_i}{2} \right)} \cdot  \besselFunctionOfFiniteFieldRepresentation \left( \antidiag \left( \lambda_1 \IdentityMatrix{2m_1}, \dots, \lambda_r \IdentityMatrix{2m_r} \right) \right) \cdot \fieldCharacter \left( -\lambda_{r} \cdot \indicatorFunction{m_r, 1} \right). $$
	The theorem now follows from \Cref{eqn:integral-representation-of-gamma-inverse}, \cref{prop:argument-inverse-and-conjugate-of-Bessel-function} and \cref{rem:exterior-square-gamma-factor-is-unitrary}.

\end{proof}

\begin{rem}\label{rem:gamma-factor-gauss-sum}
	Let
	$$S_0 = \sum_{\substack{
				m_0 > 1\\
				m_1, \dots, m_{r-1} \ge 1\\
				m_0 + \dots + m_{r-1} = m\\
				\lambda_0,\dots,\lambda_{r-1} \ \in \multiplicativegroup{\finiteField}
			}} q^{- 2\sum_{i=0}^{r-1} \binom{m_i}{2} } \cdot  \besselFunctionOfFiniteFieldRepresentation \left( \antidiag \left( \lambda_0 \IdentityMatrix{2m_0}, \lambda_1 \IdentityMatrix{2m_1}, \dots, \lambda_{r-1} \IdentityMatrix{2m_{r-1}} \right) \right)$$

	Then for every $a \in \multiplicativegroup{F}$, $S_0 = \centralCharacter{\finiteFieldRepresentation}\left(a\right) \cdot  S_0$, and therefore if $\finiteFieldRepresentation$ has a non-trivial central character, then $S_0 = 0$.

	Also let $$S_1 = \sum_{\substack{
				m_1, \dots, m_{r-1} \ge 1\\
				m_1 + \dots + m_{r-1} = m - 1\\
				\lambda_1,\dots,\lambda_{r-1} \ \in \multiplicativegroup{\finiteField}
			}} q^{- 2\sum_{i=1}^{r - 1} \binom{m_i}{2} } \cdot  \besselFunctionOfFiniteFieldRepresentation \left( \antidiag  \left(1 \IdentityMatrix{2},  \lambda_1 \IdentityMatrix{2m_1}, \dots, \lambda_{r-1} \IdentityMatrix{2m_{r-1}} \right) \right).$$
	Then by a change of variables, we have that $$ \gammaFactorOfFiniteField = q^{-\frac{m}{2} + 2 \binom{m}{2}} \left(S_0 + S_1 \cdot \sum_{a \in \multiplicativegroup{\finiteField}}{\centralCharacter{\finiteFieldRepresentation}\left(a^{-1}\right)  \fieldCharacter \left(a\right)}  \right).$$

\end{rem}

We now move to prove the odd case of the theorem:
\begin{proof}[Proof of \cref{thm:mellin-inverse-odd-case}]
	Choose $\phi = \indicatorFunction{0}$ and $W = \grpIndex{G}{\UnipotentSubgroup} \cdot  \grpIndex{M}{\UpperTriangularAdditive} \cdot \sizeof{\Mat{1}{m}{\finiteField}} \cdot \finiteFieldRepresentation\left(\oddPermutationMatrix^{-1}\right) \besselFunctionOfFiniteFieldRepresentation$, as in \cref{prop:jacquet-shalika-integral-of-bessel-function-odd-case}. Then $\JSOfFiniteFieldRepresentation{W}{\phi} = 1$, and $\FourierTransformWithRespectToCharacter{\phi}{\fieldCharacter} = q^{-\frac{m}{2}}$. By \cref{prop:dual-jacquet-shalika-formula-odd}, we have that \begin{equation*}
		\begin{split}
			\gammaFactorOfFiniteField =  & q^{-\frac{m}{2}}  \sum_{g \in \lquot{\UnipotentSubgroup}{G}} \sum_{X \in \lquot{\UpperTriangularAdditive}{M}} \sum_{Z \in \Mat{1}{m}{\finiteField}} \\
			&\besselFunctionOfFiniteFieldRepresentation\left(
			\begin{pmatrix}& 1 \\ I_{2m} & \end{pmatrix}
			\oddPermutationMatrix \ShalikaUnipotentElementOdd{X}
			\ShalikaDiagonalElementOdd{g}
			\ShalikaUpperRightUnipotentElementOdd{-\transpose{Z}}
			\oddPermutationMatrix^{-1}
			\right)
			\fieldCharacter\left(-\trace X\right).
		\end{split}
	\end{equation*}
	Since  $\oddPermutationMatrix \SmallShalikaUpperRightUnipotentElementOdd{-\transpose{Z}} \oddPermutationMatrix^{-1}$ is an upper triangular unipotent matrix with zeros right above its diagonal, we can evaluate the summation over $Z$ (using the right translation $\UnipotentSubgroup_{2m+1}$-equivariance property of the Bessel function), and we are left with the expression
	\begin{equation*}
		\begin{split}
			\gammaFactorOfFiniteField =  & q^{\frac{m}{2}}  \sum_{g \in \lquot{\UnipotentSubgroup}{G}} \sum_{X \in \lquot{\UpperTriangularAdditive}{M}} \\
			&\besselFunctionOfFiniteFieldRepresentation\left(
			\begin{pmatrix}& 1 \\ I_{2m} & \end{pmatrix}
			\oddPermutationMatrix \ShalikaUnipotentElementOdd{X}
			\ShalikaDiagonalElementOdd{g}
			\oddPermutationMatrix^{-1}
			\right)
			\fieldCharacter\left(-\trace X\right).
		\end{split}
	\end{equation*}
	Since $\oddPermutationMatrix = \smallDiagTwo{\evenPermutationMatrix}{1} $, we get that the argument of Bessel function in the summand is $\left(\begin{smallmatrix}
				& 1\\
				Y &
			\end{smallmatrix}\right) $, where $Y = \evenPermutationMatrix \SmallShalikaUnipotentElement{X} \SmallShalikaDiagonalElement{g} \evenPermutationMatrix^{-1}$ . Let $g \in G$, $X \in M$, and denote $Y = \evenPermutationMatrix \SmallShalikaUnipotentElement{X} \SmallShalikaDiagonalElement{g} \evenPermutationMatrix^{-1}$. In order for $g$, $X$ to contribute to the sum, we must have that $$\begin{pmatrix}
			  & 1 \\
			Y &
		\end{pmatrix} \in \UnipotentSubgroup_{2m+1} \antidiag\left(\lambda_0 \IdentityMatrix{n_0}, \lambda_1 \IdentityMatrix{n_1} \dots \lambda_r \IdentityMatrix{n_r} \right) \UnipotentSubgroup_{2m+1},$$
	where $n_0 + \dots + n_r = 2m + 1$ and $\lambda_0, \dots , \lambda_r \in \multiplicativegroup{\finiteField}$. Let $u_1, u_2 \in \UnipotentSubgroup_{2m}$, such that $u_1 Y u_2 = w'd'$, where $w'$ is a permutation matrix and $d'$ is a diagonal matrix. Then $$\begin{pmatrix}
			1 &     \\
			  & u_1
		\end{pmatrix}
		\begin{pmatrix}
			  & 1 \\
			Y &
		\end{pmatrix}
		\begin{pmatrix}
			u_2 &   \\
			    & 1
		\end{pmatrix} = \begin{pmatrix}
			     & 1 \\
			w'd' &
		\end{pmatrix}.$$
	By the uniqueness of the Bruhat decomposition, we get that $n_0 = 1$, $\lambda_0=1$, and
	$$u_1 Y u_2 = \antidiag \left(\lambda_{1} \IdentityMatrix{n_1}, \dots, \lambda_{r} \IdentityMatrix{n_r} \right),$$
	i.e., $\evenPermutationMatrix \SmallShalikaUnipotentElement{X} \SmallShalikaDiagonalElement{g} \evenPermutationMatrix^{-1} \in \UnipotentSubgroup_{2m} \antidiag \left(\lambda_{1} \IdentityMatrix{n_1}, \dots, \lambda_{r} \IdentityMatrix{n_r} \right) \UnipotentSubgroup_{2m}$. As in the even case, since $g$ runs on cosets of $\lquot{\UnipotentSubgroup}{G}$, using the Bruhat decomposition, we may assume its representatives are in the form $g = wdu$, where $w$ is a permutation matrix, $d$ is a diagonal matrix, and $u \in \UnipotentSubgroup_{2m}$. By \cref{lem:support-of-bessel-function-jacquet-shalika-integral}, we have that $wd = \antidiag\left(\lambda_1 \IdentityMatrix{m_1}, \dots, \lambda_r \IdentityMatrix{m_r} \right)$ and then $Y = \antidiag\left(\lambda_1 \IdentityMatrix{m_1}, \dots, \lambda_r \IdentityMatrix{m_r} \right) v$, where $n_i = 2m_i$, for every $1 \le i \le r$, and $v \in \UnipotentSubgroup_{2m}$ has zeros right above its diagonal. Since $\left(\begin{smallmatrix}
				& 1\\
				Y v^{-1} &
			\end{smallmatrix}\right) \cdot \smallDiagTwo{v}{1} = \left(\begin{smallmatrix}
				& 1\\
				Y &
			\end{smallmatrix}\right)$, and the matrix $\smallDiagTwo{v}{1}$ has zeros right above its diagonal, we get that $\besselFunctionOfFiniteFieldRepresentation \left( \begin{smallmatrix}
				& 1\\
				Y &
			\end{smallmatrix} \right) = \besselFunctionOfFiniteFieldRepresentation\left(\antidiag\left(1\IdentityMatrix{1},\lambda_1 \IdentityMatrix{2m_1}, \dots, \lambda_r \IdentityMatrix{2m_r} \right)\right)$. Using \cref{lem:jacquet-shalika-support-orbits-size} to count how many such $u$ and $X$ correspond to a given choice of $wd = \antidiag \left(\lambda_1 \IdentityMatrix{m_1}, \dots \lambda_r \IdentityMatrix{m_r} \right)$, we get that
	$$
		\gammaFactorOfFiniteField = q^{\frac{m}{2}+2 \binom{m}{2}} \sum_{\substack{
				m_1, \dots, m_r \ge 1\\
				m_1 + \dots + m_r = m\\
				\lambda_1, \dots, \lambda_r \in \multiplicativegroup{\finiteField}
			}}
		q^{-\sum_{i=1}^{r}{2 \binom{m_i}{2}}} \cdot {\besselFunctionOfFiniteFieldRepresentation} \left(\antidiag\left(\IdentityMatrix{1}, \lambda_1 \IdentityMatrix{2 m_1}, \dots, \lambda_r \IdentityMatrix{2 m_r}\right)\right),$$
	as required.
\end{proof}

\subsection{Computations for small $n$}\label{subsec:computation-regular-character}
In this section, we use the correspondence between irreducible cuspidal representations of $\GL{n}{\finiteField}$ and equivalence classes of regular characters of $\multiplicativegroup{\finiteFieldExtension{n}}$ in order to find explicit expressions for the exterior gamma factor.

Recall that a character $\cuspidalCharacter : \multiplicativegroup{\finiteFieldExtension{n}} \rightarrow \multiplicativegroup{\cComplex}$ is called regular if its Galois orbit is of size $n$, i.e., for every $1 \le k \le n -1$, $\cuspidalCharacter^{q^k} \ne \cuspidalCharacter$. We define an equivalence relation on the set of regular characters of $\multiplicativegroup{\finiteFieldExtension{n}}$, where $\cuspidalCharacter \sim \cuspidalCharacter'$ if they are in the same Galois orbit, i.e., if $\cuspidalCharacter^{q^k} = \cuspidalCharacter'$ for some $k$. By \cite[Section 6]{gel1970representations}, there exists a bijection between equivalence classes of regular characters of $\multiplicativegroup{\finiteFieldExtension{n}}$ and irreducible representations of $\GL{n}{\finiteField}$. Furthermore, if $\finiteFieldRepresentation$ is an irreducible representation that corresponds under this bijection to the equivalence class of the regular character $\cuspidalCharacter$, then there exists a formula that expresses the character $\representationCharacter{\finiteFieldRepresentation} \left(g\right) = \trace\left( \finiteFieldRepresentation \left(g\right)\right)$ in terms of $\cuspidalCharacter$. See also \cite[Section 3.1]{nien2017n} for more details.

Suppose that $\finiteFieldRepresentation$ is an irreducible cuspidal representation of $\GL{n}{\finiteField}$ corresponding under the above bijection to the equivalence class of the regular character $\cuspidalCharacter : \multiplicativegroup{\finiteFieldExtension{n}} \rightarrow \multiplicativegroup{\cComplex}$. If $n = 2m$ is even, assume that $\finiteFieldRepresentation$ does not admit a Shalika vector. By \cref{thm:equivalent-conditions-for-a-shalika-vector}, this is equivalent to the assumption that $\cuspidalCharacter \restriction_{\multiplicativegroup{\finiteFieldExtension{m}}}$ is non-trivial. By \cref{thm:bessel-function-in-terms-of-representation-character}, the Bessel function $\besselFunctionOfFiniteFieldRepresentation$ can be expressed in terms of  $\representationCharacter{\finiteFieldRepresentation}$, the character of the representation $\finiteFieldRepresentation$, and $\fieldCharacter$ only. Therefore, by \Cref{subsec:melin-transform}, the exterior square gamma factor $\gammaFactorOfFiniteField$ is expressible in terms of $\cuspidalCharacter$ and $\fieldCharacter$ only. It is desirable to find such a precise expression for $\gammaFactorOfFiniteField$. A similar computation has been done by Nien in \cite[Theorem 1.1]{nien2017n} for the gamma factor $\gamma\left(\pi \times \tau, \fieldCharacter\right)$, where $\pi$ and $\tau$ are irreducible representations of $\GL{n}{\finiteField}$ and $\GL{1}{\finiteField}$ respectively.

We were not able to achieve this for general $n$. We specify the results for $n = 2,3,4$. These computations demonstrate that the solution for this problem is not straightforward. Furthermore, since by \cref{rem:exterior-square-gamma-factor-is-unitrary}, $\gammaFactorOfFiniteField$ has absolute value one, our computations actually lead to some special exponential sums having absolute value one, which is a rare situation.

In a subsequent work after this paper, we were able to formulate a conjecture expressing $\gammaFactorOfFiniteField$ as a product of Gauss sums. See \cite[Section 5]{ye2020epsilon}.

\begin{rem}\label{rem:evaluation-of-square-sums}
	Towards these computations, we will encounter sums of the form $$\sum_{\lambda
		\in \multiplicativegroup{\finiteField}} {\sum_{\substack{\xi
				\in \multiplicativegroup{\finiteFieldExtension{n}}\\
				\FieldNorm_{\FieldExtension{\finiteFieldExtension{n}}{\finiteField}} \left(\xi\right) = \lambda^2}}{ J \left(\xi, \lambda\right)} },$$
	where $J$ is some complex valued function. One can show that $ \FieldNorm_{\FieldExtension{\finiteFieldExtension{n}}{\finiteField}} \left(\xi\right) $ is a square in $\multiplicativegroup{\finiteField}$ if and only if $\xi \in \multiplicativegroup{\finiteFieldExtension{n}}$ is a square. We therefore have that $$ \left\{ \left(\xi, \lambda\right) \in \multiplicativegroup{\finiteFieldExtension{n}} \times \multiplicativegroup{\finiteField} \mid \FieldNorm_{\FieldExtension{\finiteFieldExtension{n}}{\finiteField}} \left(\xi\right) = \lambda^2 \right\} = \left\{ \left(\xi'^2, \pm \FieldNorm_{\FieldExtension{\finiteFieldExtension{n}}{\finiteField}}\left(\xi'\right) \right) \mid \xi' \in \multiplicativegroup{\finiteFieldExtension{n}} \right\}.$$ Since the map $\xi \mapsto \xi^2$ is two-to-one for fields with characteristic $\ne 2$, we get that for these fields $$ \sum_{\lambda
		\in \multiplicativegroup{\finiteField}} {\sum_{\substack{\xi
				\in \multiplicativegroup{\finiteFieldExtension{n}}\\
				\FieldNorm_{\FieldExtension{\finiteFieldExtension{n}}{\finiteField}} \left(\xi\right) = \lambda^2}}{ J \left(\xi, \lambda\right)} } =  \frac{1}{2} {\sum_{\xi
			\in \multiplicativegroup{\finiteFieldExtension{n}}}}\left( { J \left(\xi^2, \FieldNorm_{\FieldExtension{\finiteFieldExtension{n}}{\finiteField}} \left(\xi\right) \right) + J \left(\xi^2, -\FieldNorm_{\FieldExtension{\finiteFieldExtension{n}}{\finiteField}} \left(\xi\right) \right)} \right).$$
	This formula is also true for fields with characteristic $2$, as in this case the map $\xi \mapsto \xi^2$ is a bijection, and $ \FieldNorm_{\FieldExtension{\finiteFieldExtension{n}}{\finiteField}} \left(\xi\right) = -\FieldNorm_{\FieldExtension{\finiteFieldExtension{n}}{\finiteField}} \left(\xi\right).$
	
	If $n$ is odd, then $\FieldNorm_{\FieldExtension{\finiteFieldExtension{n}}{\finiteField}}\left(-1\right) = -1$, and then by replacing $\xi$ with $-\xi$ in the second sum we get that $$ \sum_{\lambda
		\in \multiplicativegroup{\finiteField}} {\sum_{\substack{\xi
				\in \multiplicativegroup{\finiteFieldExtension{n}}\\
				\FieldNorm_{\FieldExtension{\finiteFieldExtension{n}}{\finiteField}} \left(\xi\right) = \lambda^2}}{ J \left(\xi, \lambda\right)} } =  {\sum_{\xi
			\in \multiplicativegroup{\finiteFieldExtension{n}}}} { J \left(\xi^2, \FieldNorm_{\FieldExtension{\finiteFieldExtension{n}}{\finiteField}} \left(\xi\right) \right)}.$$
\end{rem}

\subsubsection{Computation for $\GL{2}{\finiteField}$}

Since $\cuspidalCharacter \restriction_{\multiplicativegroup{\finiteField}} = \centralCharacter{\finiteFieldRepresentation} $, we have that $\finiteFieldRepresentation$ has a non-trivial central character. We have in \cref{rem:gamma-factor-gauss-sum}, that $S_0 = 0$, $S_1 = 1$, and therefore we get the following

\begin{thm}
	Let $\representationDeclaration{\finiteFieldRepresentation}$ be an irreducible cuspidal representation of $\GL{2}{\finiteField}$, with a non-trivial central character. Then
	$$\gammaFactorOfFiniteField = q^{-\frac{1}{2}} \sum_{a \in \multiplicativegroup{\finiteField}}{ \centralCharacter{\finiteFieldRepresentation}} \left(a^{-1}\right) \fieldCharacter \left(a\right).$$
\end{thm}

\subsubsection{Computation for $\GL{3}{\finiteField}$}

By \cref{thm:mellin-inverse-odd-case}, we have that $$ \gammaFactorOfFiniteField = q^{\frac{1}{2}} \sum_{a \in \multiplicativegroup{\finiteField}} \besselFunctionOfFiniteFieldRepresentation \begin{pmatrix}
0                   & 1 \\
a\IdentityMatrix{2} & 0
\end{pmatrix}.$$

In order to proceed, we use the formula for the value of the Bessel function from \cite[Section 7]{gel1970representations}.

$$ \besselFunctionOfFiniteFieldRepresentation \left( \antidiag \left( \lambda_1 \IdentityMatrix{1}, \lambda_2 \IdentityMatrix{2} \right) \right) = q^{-2} \sum_{\substack{\xi \in \multiplicativegroup{\finiteFieldExtension{3}}\\
		\FieldNorm_{\FieldExtension{\finiteFieldExtension{3}}{\finiteField}}\left(\xi\right) = \lambda_1 \cdot \lambda_2^2
}}{ \fieldCharacter\left(-\lambda_2^{-1} \FieldTrace_{\FieldExtension{\finiteFieldExtension{3}}{\finiteField}} \left(\xi\right) \right) \cuspidalCharacter \left(\xi\right) }.$$

Using \cref{rem:evaluation-of-square-sums}, we get

\begin{thm}
	Let $\localFieldRepresentation$ be an irreducible cuspidal representation of $\GL{3}{\finiteField}$ associated with the regular multiplicative character $\cuspidalCharacter : \multiplicativegroup{\finiteFieldExtension{3}} \rightarrow \multiplicativegroup{\cComplex}$. Then
	$$\gammaFactorOfFiniteField = q^{-\frac{3}{2}} \sum_{\xi \in \multiplicativegroup{\FiniteField{3}}}{ \fieldCharacter\left(-\frac{ \FieldTrace_{\FieldExtension{\finiteFieldExtension{3}}{\finiteField}}\left(\xi^2\right)}{\FieldNorm_{\FieldExtension{\finiteFieldExtension{3}}{\finiteField}}\left(\xi\right)} \right) \cuspidalCharacter \left(\xi^2\right)}. $$
\end{thm}

\subsubsection{Computation for $\GL{4}{\finiteField}$}
By \cref{rem:gamma-factor-gauss-sum}, we have that $$\gammaFactorOfFiniteField = q \left( S_0 + S_1 \cdot  \sum_{a \in \multiplicativegroup{\finiteField}}{ \centralCharacter{\finiteFieldRepresentation} \left( a^{-1} \right) \fieldCharacter \left(a\right)} \right),$$ where in our case $S_0 = q^{-2} \sum_{\lambda \in \multiplicativegroup{\finiteField}} \besselFunctionOfFiniteFieldRepresentation \left( \lambda \IdentityMatrix{4} \right) = \begin{cases}
q^{-1}-q^{-2} & \centralCharacter{\finiteFieldRepresentation} \text{ is trivial} \\
0             & \text{otherwise}
\end{cases}$, and $S_1 = \sum_{\lambda \in \multiplicativegroup{\finiteField}} \besselFunctionOfFiniteFieldRepresentation \left( \begin{smallmatrix}
0 &  \IdentityMatrix{2}\\
\lambda \IdentityMatrix{2} & 0
\end{smallmatrix} \right)$.

In order to proceed, we use the formulas for the Bessel function of $\GL{4}{\finiteField}$. These are computed by Deriziotis and Gotsis. We have by \cite[Page 103]{deriziotis1998cuspidal} that for $w = w_6 = \left(\begin{smallmatrix}
0 & \IdentityMatrix{2}\\
\IdentityMatrix{2} & 0
\end{smallmatrix}\right)$ and $t = \smallDiagTwo{\mu \IdentityMatrix{2}}{\nu \IdentityMatrix{2}}$, the value $\besselFunctionOfFiniteFieldRepresentation \left(tw\right)$ is given by the formula $$ \besselFunctionOfFiniteFieldRepresentation \left( tw \right) = \sum_{\substack{\xi \in \multiplicativegroup{\finiteFieldExtension{4}}\\
		\FieldNorm_{\FieldExtension{\finiteFieldExtension{4}}{\finiteField}} \left(\xi\right) = \det t }}{ F_6\left(\xi, t\right)\cuspidalCharacter\left(\xi\right) },$$ where $$F_6 \left(\xi, t\right) = -q^{-4} \left( F'_6\left(\xi, t\right) + \sum_{\beta \in \multiplicativegroup{\finiteField}}{\fieldCharacter \left(-\beta + \frac{a_1 \left(\xi\right) + a_3\left(\xi\right) \mu \nu}{\beta \mu \nu^2}\right)} \right),$$
where $\sum_{i=0}^{4} {a_{i} \left(\xi\right) t^i} = \prod_{i=0}^{3} \left(t - \xi^{q^i}\right)$ and $$F'_6\left(\xi, t\right) = \begin{cases}
-q & \xi \in \finiteFieldExtension{2} \setminus \finiteField \text{ and } \mu \nu = -\FieldNorm_{\FieldExtension{\finiteFieldExtension{2}}{\finiteField}} \left(\xi\right) \\
0  & \text{otherwise}
\end{cases}.$$
Denote \begin{align*}
	I \left(\xi, \lambda, \beta \right) = \cuspidalCharacter \left( \xi \right) \fieldCharacter \left(-\beta + \frac{a_1 \left(\xi\right) + a_3\left(\xi\right) \lambda}{\beta \lambda^2}\right) &  & \xi \in \multiplicativegroup{\finiteFieldExtension{4}}; \, \beta, \lambda \in \multiplicativegroup{\finiteField}.
\end{align*}

We therefore have that $$S_1 = -q^{-4} \sum_{\lambda \in \multiplicativegroup{\finiteField}}\left({\sum_{\substack{\xi \in \multiplicativegroup{\finiteFieldExtension{4}}\\
			\FieldNorm_{\FieldExtension{\finiteFieldExtension{4}}{\finiteField}} \left(\xi\right) = \lambda^2 }}{ \sum_{\beta \in \multiplicativegroup{\finiteField}}{I \left(\xi, \lambda, \beta \right)} }} -q \sum_{\substack{\xi \in \finiteFieldExtension{2} \setminus \finiteField \\
		\FieldNorm_{\FieldExtension{\finiteFieldExtension{2}}{\finiteField}}\left(\xi\right) = -\lambda}}{ \cuspidalCharacter \left(\xi\right)} \right).$$
Regarding the second summand, since $-\lambda$ runs over all the possible norms of elements of $\finiteFieldExtension{2}$, we have that $$\sum_{\lambda \in \multiplicativegroup{\finiteField}}{\sum_{\substack{\xi \in \finiteFieldExtension{2} \setminus \finiteField\\
			\FieldNorm_{\FieldExtension{\finiteFieldExtension{2}}{\finiteField}} \left(\xi\right) = -\lambda}
	} \cuspidalCharacter \left(\xi\right)} = \sum_{\xi \in \multiplicativegroup{\finiteFieldExtension{2}} }{\cuspidalCharacter\left(\xi\right)} - \sum_{\xi \in \multiplicativegroup{\finiteField }}{\cuspidalCharacter\left(\xi\right)}.$$
Since $\cuspidalCharacter \restriction_{\multiplicativegroup{\finiteFieldExtension{2}}}$ is non-trivial, and since $\cuspidalCharacter\restriction_{\multiplicativegroup{\finiteField}} = \centralCharacter{\finiteFieldRepresentation}$ we have that $$\sum_{\lambda \in \multiplicativegroup{\finiteField}}{\sum_{\xi \in \finiteFieldExtension{2} \setminus \finiteField} \cuspidalCharacter \left(\xi\right)} = -\sum_{\xi \in \multiplicativegroup{\finiteField}}{\centralCharacter{\finiteFieldRepresentation}\left(\xi\right)} = \begin{cases}
1-q & \centralCharacter{\finiteFieldRepresentation}\text{ is trivial} \\
0   & \text{otherwise}
\end{cases}.$$
Note that if $\centralCharacter{\finiteFieldRepresentation}$ is trivial, $\sum_{a \in \multiplicativegroup{\finiteField}}{ \centralCharacter{\finiteFieldRepresentation} \left( a^{-1} \right) \fieldCharacter \left(a\right) } = -1$, and therefore we have that $$ \gammaFactorOfFiniteField = q^{-3} \left( q T_0 - S'_1 \cdot \sum_{a \in \multiplicativegroup{\finiteField}}{ \centralCharacter{\finiteFieldRepresentation} \left( a^{-1} \right) \fieldCharacter \left(a\right) }   \right),$$ where $T_0 = \begin{cases}
q^2 - 1 & \centralCharacter{\finiteFieldRepresentation} \text{ is trivial} \\
0       & \text{otherwise}
\end{cases} $, and $$S'_1 = \sum_{\lambda \in \multiplicativegroup{\finiteField}} {\sum_{\substack{\xi \in \multiplicativegroup{\finiteFieldExtension{4}}\\
			\FieldNorm_{\FieldExtension{\finiteFieldExtension{4}}{\finiteField}} \left(\xi\right) = \lambda^2 }}{ \sum_{\beta \in \multiplicativegroup{\finiteField}}{I \left(\xi, \lambda, \beta \right)} }}.$$
Using \cref{rem:evaluation-of-square-sums}, and the formulas $a_3 \left(\xi\right) = -\FieldTrace_{\FieldExtension{\finiteFieldExtension{4}}{\finiteField}} \left(\xi\right)$, $a_1 \left(\xi\right) = -\FieldTrace_{\FieldExtension{\finiteFieldExtension{4}}{\finiteField}}\left(\frac{1}{\xi}\right) \FieldNorm_{\FieldExtension{\finiteFieldExtension{4}}{\finiteField}} \left(\xi\right) $, we get the expression $S_1' = \frac{1}{2} \left(S_1^{'+} + S_1^{'-}\right)$, where
$$ S_1^{'\pm} = \sum_{\xi \in \multiplicativegroup{\finiteFieldExtension{4}} }{\cuspidalCharacter \left(\xi^2 \right) { K_{\fieldCharacter}\left(1,  \FieldTrace_{\FieldExtension{\finiteFieldExtension{4}}{\finiteField}} \left(\frac{1}{\xi^2}\right) \pm \frac{\FieldTrace_{\FieldExtension{\finiteFieldExtension{4}}{\finiteField}}\left(\xi^2\right)}{\FieldNorm_{\FieldExtension{\finiteFieldExtension{4}}{\finiteField}} \left(\xi\right)} \right)}},$$
where $$K_{\fieldCharacter}\left(a, b\right) = \sum_{x \in \multiplicativegroup{\finiteField}}{ \fieldCharacter\left(a x \right) }\fieldCharacter\left(\frac{b}{x}\right)$$ is the Kloosterman sum.

\begin{thm}Let $\representationDeclaration{\finiteFieldRepresentation}$ be an irreducible cuspidal representation of $\GL{4}{\finiteField}$, associated with the regular multiplicative character $\cuspidalCharacter : \multiplicativegroup{\finiteFieldExtension{4}} \rightarrow \multiplicativegroup{\cComplex}$, satisfying that $\cuspidalCharacter \restriction_{\multiplicativegroup{\finiteFieldExtension{2}}} $ is non trivial. Then
	
	\begin{align*}
		\gammaFactorOfFiniteField = & \, q^{-2} T_0 - \frac{1}{2}q^{-3} G_{\fieldCharacter} \left(\cuspidalCharacter\right) \cdot \sum_{\xi \in \multiplicativegroup{\finiteFieldExtension{4}} }{\cuspidalCharacter \left(\xi^2 \right) { K_{\fieldCharacter}\left(1,  \FieldTrace_{\FieldExtension{\finiteFieldExtension{4}}{\finiteField}} \left(\frac{1}{\xi^2}\right) + \frac{\FieldTrace_{\FieldExtension{\finiteFieldExtension{4}}{\finiteField}}\left(\xi^2\right)}{\FieldNorm_{\FieldExtension{\finiteFieldExtension{4}}{\finiteField}} \left(\xi\right)} \right)}} \\
		& - \frac{1}{2}q^{-3} G_{\fieldCharacter} \left(\cuspidalCharacter\right) \cdot \sum_{\xi \in \multiplicativegroup{\finiteFieldExtension{4}} }{\cuspidalCharacter \left(\xi^2 \right) { K_{\fieldCharacter}\left(1,  \FieldTrace_{\FieldExtension{\finiteFieldExtension{4}}{\finiteField}} \left(\frac{1}{\xi^2}\right) - \frac{\FieldTrace_{\FieldExtension{\finiteFieldExtension{4}}{\finiteField}}\left(\xi^2\right)}{\FieldNorm_{\FieldExtension{\finiteFieldExtension{4}}{\finiteField}} \left(\xi\right)} \right)}},
	\end{align*}
	where $G_{\fieldCharacter} \left(\cuspidalCharacter\right) = \sum_{a \in \multiplicativegroup{\finiteField}}{ \cuspidalCharacter\left(a^{-1}\right) \fieldCharacter \left(a\right) }$ is the Gauss sum, $K_{\fieldCharacter}\left(a,b\right) = \sum_{x \in \multiplicativegroup{\finiteField}}{\fieldCharacter \left(ax\right) \fieldCharacter\left( b x^{-1}\right)}$ is the Kloosterman sum, and $T_0 = \begin{cases}
	q^2 - 1 & \centralCharacter{\finiteFieldRepresentation} \text{ is trivial} \\
	0       & \text{otherwise}
	\end{cases} $.
\end{thm}

\section{The relation with the Jacquet-Shalika integral over a local field}\label{sec:gamma-factor}

In this section we relate our theory of the Jacquet-Shalika integral over a finite field, to the theory of the Jacquet-Shalika integral over a local non-archimedean field, via level zero representations.

\subsection{Preliminaries and notations}

\subsubsection{Notations} \label{subsection:local-field-notations}

Let $\localField$ be a local non-archimedean field. Denote by $\integersring$ the ring of integers of $\localField$, by $\maximalideal$ the unique prime ideal of $\integersring$. Let $\uniformizer$ be a uniformizer of $\localField$, a generator of $\maximalideal$. Denote by $\residueField = \rquot{\integersring}{\maximalideal}$ the residue field of $\localField$ and by $q = \sizeof{\residueField}$. We use the standard normalization for the absolute value, so that $\abs{\uniformizer} = \frac{1}{q}$.

Denote by $\quotientMap : \integersring \rightarrow \residueField$ the quotient map. We continue to denote by $\quotientMap$ the maps induced by $\quotientMap$  from $\integersring^m \rightarrow \residueField^m$, $\SquareMat{m}{\integersring} \rightarrow \SquareMat{m}{\residueField}$, $\GL{m}{\integersring} \rightarrow \GL{m}{\residueField}$, etc.

Let $\fieldCharacter : \localField \rightarrow \multiplicativegroup{\cComplex}$ be a non-trivial additive character, with conductor $\maximalideal$, i.e., $\fieldCharacter$ is trivial on $\maximalideal$ but not on $\integersring$. Then $\fieldCharacter$ defines a non-trivial additive character $\residueFieldCharacter : \residueField \rightarrow \multiplicativegroup{\cComplex}$ by $\fieldCharacter_0 \circ \quotientMap = \fieldCharacter\restriction_{\integersring}$.

We denote by $\Schwartz \left( \localField^m \right)$ the space of Schwartz functions $f : \localField^m \rightarrow \cComplex$ - locally constant and compactly supported functions.

We choose the standard normalizations for the Haar measures on $\localField$ and $\multiplicativegroup{\localField}$, i.e.,
\begin{equation*}
	\begin{split}
		\int_{\integersring} dx = 1,\\
		\int_{\multiplicativegroup{\integersring}} \multiplicativeMeasure{x} = 1.
	\end{split}
\end{equation*}

For $f \in \Schwartz \left( \localField^m \right)$, we denote its Fourier transform $$ \FourierTransformWithRespectToCharacter{f}{\fieldCharacter} \left(y\right) = {q^{\frac{m}{2}}} \int_{\localField^m} f \left(x\right) \cdot \fieldCharacter \left( \standardForm{x}{y} \right) dx,$$
where $\standardForm{x}{y}$ is the standard bilinear form on $\localField^m$.
The Fourier inversion formulas for this normalization are given by
\begin{equation*}
	\begin{split}
		\FourierTransformWithRespectToCharacter{\FourierTransformWithRespectToCharacter{f}{\fieldCharacter}}{\fieldCharacter^{-1}} \left(x\right) &= f\left(x\right), \\
		\FourierTransformWithRespectToCharacter{\FourierTransformWithRespectToCharacter{f}{\fieldCharacter}}{\fieldCharacter} \left(x\right) &= f \left(-x\right).
	\end{split}
\end{equation*}

If $\representationDeclaration{\localFieldRepresentation}$ is an irreducible generic representation of $\GL{n}{F}$, we denote by $\WhittakerModelOfLocalFieldRepresentation$, as in \Cref{subsection:finite-field-notations}, its Whittaker model with respect to the character $\fieldCharacter$. We also denote for an element $W \in \WhittakerModelOfLocalFieldRepresentation$, an element $\Contragradient{W} \in \Whittaker\left(\Contragradient{\localFieldRepresentation}, \fieldCharacter^{-1} \right)$, defined by $ \Contragradient{W} \left(g\right) = W\left(\weylElement{n} \inverseTranspose{g}\right),$ where $\inverseTranspose{g} = \transpose{g^{-1}}$ and $\weylElement{n} = \left(\begin{smallmatrix}
			& & 1\\
			& \iddots &\\
			1
		\end{smallmatrix}\right)$.

\subsubsection{The local Jacquet-Shalika integral}
We briefly review the theory of the local Jacquet-Shalika integral.

Let $\representationDeclaration{\localFieldRepresentation}$ be a generic irreducible representation of $\GL{n}{\localField}$. We first give the formal formulas for the Jacquet-Shalika integral and its dual. These should initially be treated as formal expressions. \cref{thm:convergence-of-jacquet-shalika-integral} and the discussion afterwards explain how to interpret these integrals for general $s \in \cComplex$.

Suppose $n = 2m$. The Jacquet-Shalika integral of $\localFieldRepresentation$ with respect to the character $\fieldCharacter$ is defined as follows: for every $s \in \cComplex$, $W \in \WhittakerModelOfLocalFieldRepresentation$, $\phi \in \Schwartz\left(\localField^m\right)$
$$ \JSOfLocalFieldRepresentation{s}{W}{\phi} = \int_{\lquot{\UnipotentSubgroup}{G}} \int_{\lquot{\UpperTriangularAdditive}{M}} W\left(\evenPermutationMatrix \ShalikaUnipotentElement{X} \ShalikaDiagonalElement{g} \right) \fieldCharacter \left(-\trace X \right) dX \cdot \abs{\det g}^s \phi\left( \lastrowvector g \right) \multiplicativeMeasure{g}, $$
where the notations are the same as in \Cref{subsec:jacquet-shalika-integral}, this time defined over $\localField$.

We define the dual Jacquet-Shalika integral as $$ \DualJSOfLocalFieldRepresentation{s}{W}{\phi} = \JS{\Contragradient{\finiteFieldRepresentation}}{\fieldCharacter^{-1}}\left(1-s,\Contragradient{\finiteFieldRepresentation} \begin{pmatrix}
		                   & \IdentityMatrix{m} \\
		\IdentityMatrix{m} &
	\end{pmatrix} \Contragradient{W}, \FourierTransformWithRespectToCharacter{\phi}{\fieldCharacter} \right).$$

Now suppose $n=2m+1$. In this case, the Jacquet-Shalika integral of $\localFieldRepresentation$ with respect to the character $\fieldCharacter$ is defined as

\begin{equation*}
	\begin{split}
		\JSOfLocalFieldRepresentation{s}{W}{\phi} =  \int_{\lquot{\UnipotentSubgroup}{G}}\int_{\lquot{\UpperTriangularAdditive}{M}} \int_{\Mat{1}{m}{\localField}}
		&W\left(
		\oddPermutationMatrix
		\ShalikaUnipotentElementOdd{X}
		\ShalikaDiagonalElementOdd{g}
		\ShalikaLowerUnipotentElementOdd{Z}
		\right) \\
		& \cdot
		\fieldCharacter\left(-\trace X\right)
		\phi \left( Z \right)\abs{\det{g}}^{s-1}
		dZ\,dX\,\multiplicativeMeasure{g}.
	\end{split}
\end{equation*}

In this case, we define the dual Jacquet-Shalika integral as
$$\DualJSOfLocalFieldRepresentation{s}{W}{\phi} = \JS{\Contragradient{\finiteFieldRepresentation}}{\fieldCharacter^{-1}}\left(
	1-s,
	\Contragradient{\localFieldRepresentation} \begin{pmatrix}
		                   & \IdentityMatrix{m} &   \\
		\IdentityMatrix{m} &                    &   \\
		                   &                    & 1
	\end{pmatrix} \Contragradient{W}, \FourierTransformWithRespectToCharacter{\phi}{\fieldCharacter} \right).$$

These Jacquet-Shalika integrals converge in suitable half planes. From \cite{jacquet11exterior}, we have the following theorems (for $n=2m$ or $n=2m+1$) regarding their convergence:

\begin{thm}[{\cite[Section 7, Proposition 1; Section 9, Proposition 3]{jacquet11exterior}}]\label{thm:convergence-of-jacquet-shalika-integral}
	There exists $\realPartHalfRightPlaneLocalFieldRepresentation \in \rReal$, such that for every $s \in \cComplex$ with $\RealPart \left(s\right) > \realPartHalfRightPlaneLocalFieldRepresentation$, the integral $\JSOfLocalFieldRepresentation{s}{W}{\phi}$ converges, for every $W \in \WhittakerModelOfLocalFieldRepresentation$ and $\phi \in \Schwartz \left(\localField^m\right)$.
\end{thm}

Correspondingly, the dual Jacquet-Shalika integrals $\DualJSOfLocalFieldRepresentation{s}{W}{\phi}$ converge in a left half plane ($\RealPart \left(s\right) < 1 - \realPartHalfRightPlaneLocalFieldDualRepresentation$).

By \cite[Proposition 2.3]{kewat2011local}, \cite[Lemma 3.1]{CogdellMatringe15}, for fixed $W$ and $\phi$, $\JSOfLocalFieldRepresentation{s}{W}{\phi}$ converges (in its domain of convergence) to an element of $\cComplex \left(q^{-s}\right)$ (a rational function in $q^{-s}$) and therefore has an meromorphic continuation to the entire complex plane, which we continue to denote $\JSOfLocalFieldRepresentation{s}{W}{\phi}$. Similarly, we continue to denote the meromorphic continuation of $\DualJSOfLocalFieldRepresentation{s}{W}{\phi}$ with the same notation. Moreover, the set $$I = \Span_\cComplex \left\{ \JSOfLocalFieldRepresentation{s}{W}{\phi} \mid W \in \WhittakerModelOfLocalFieldRepresentation, \phi \in \Schwartz \left(F^m\right) \right\}$$ is a fractional ideal of $\cComplex \left[q^s, q^{-s} \right]$. By a non-vanishing theorem of D. Belt (see \cite[Theorem 2.2]{belt2011holomorphy}), there exists a unique polynomial $p \in \cComplex \left[ Z \right]$, such that $p \left(0\right) = 1$ and $I = \frac{1}{p \left( q^{-s} \right) }\cComplex \left[q^{-s}, q^s\right]$. We define the exterior square $L$-function $\exteriorSquareLFunction{s}{\localFieldRepresentation} = \frac{1}{p \left( q^{-s} \right) }$ as in \cite[Definition 3.4]{KewatRaghunathan12}.

Similar to \cref{prop:dual-jacquet-shalika-formula-even} and \cref{prop:dual-jacquet-shalika-formula-odd}, we can express $\DualJSOfLocalFieldRepresentation{s}{W}{\phi}$ as in the following

\begin{prop}
	\begin{enumerate}
		\item For $n = 2m$,
		$$ \DualJSOfLocalFieldRepresentation{s}{W}{\phi} = \int_{\lquot{\UnipotentSubgroup}{G}} \int_{\lquot{\UpperTriangularAdditive}{M}} W\left(\evenPermutationMatrix \ShalikaUnipotentElement{X} \ShalikaDiagonalElement{g} \right) \fieldCharacter \left(-\trace X \right) dX \cdot \abs{\det g}^{s-1} \cdot \FourierTransformWithRespectToCharacter{\phi}{\fieldCharacter} \left( \firstrowvector \inverseTranspose{g} \right) \multiplicativeMeasure{g}. $$
		\item For $n = 2m+1$,
		\begin{equation*}
			\begin{split}
				\DualJSOfLocalFieldRepresentation{s}{W}{\phi} =  \int_{\lquot{\UnipotentSubgroup}{G}}\int_{\lquot{\UpperTriangularAdditive}{M}} \int_{\Mat{1}{m}{\localField}}
				&W\left(
				\begin{pmatrix}
					& 1 \\
					\IdentityMatrix{2m}
				\end{pmatrix}
				\oddPermutationMatrix
				\ShalikaUnipotentElementOdd{X}
				\ShalikaDiagonalElementOdd{g}
				\ShalikaUpperRightUnipotentElementOdd{-\transpose{Z}}
				\right) \\
				& \cdot
				\fieldCharacter\left(-\trace X\right)
				\FourierTransformWithRespectToCharacter{\phi}{\fieldCharacter} \left( Z \right)\abs{\det{g}}^{s}
				dZ\,dX\,\multiplicativeMeasure{g}.
			\end{split}
		\end{equation*}
	\end{enumerate}
\end{prop}

We now introduce the important local exterior square factors $\gammaFactorOfLocalField{s}$ and $\epsilonFactorOfLocalField{s}$, relating the Jacquet-Shalika integral $\JSOfLocalFieldRepresentation{s}{W}{\phi}$ to its dual $\DualJSOfLocalFieldRepresentation{s}{W}{\phi}$.

\begin{thm}[{\cite[Theorem 4.1]{matringe2014linear}, \cite[Theorem 3.1]{CogdellMatringe15}}]\label{thm:local-field-functional-equation}
	There exists an element $\gammaFactorOfLocalField{s} \in \cComplex \left(q^{-s}\right)$, such that for every $W \in \WhittakerModelOfLocalFieldRepresentation$ and $\phi \in \Schwartz \left( F^m \right)$,
	$$ \DualJSOfLocalFieldRepresentation{s}{W}{\phi} = \gammaFactorOfLocalField{s} \cdot \JSOfLocalFieldRepresentation{s}{W}{\phi}.$$ Furthermore $$ \gammaFactorOfLocalField{s} = \epsilonFactorOfLocalField{s} \cdot \frac{\exteriorSquareLFunction{1-s}{\Contragradient{\localFieldRepresentation}}}{\exteriorSquareLFunction{s}{\localFieldRepresentation}},$$
	where $\epsilonFactorOfLocalField{s}$ is an invertible element of $\cComplex {\left[q^{-s}, q^s\right]}$.
\end{thm}

In the case where $\localFieldRepresentation$ is  supercuspidal representation, the following result regrading $\exteriorSquareLFunction{s}{\localFieldRepresentation}$ is known

\begin{thm}[{\cite[Theorem 3.6]{jo2018derivatives}}]\label{thm:exterior-square-polynomial-is-a-divisor}
	\begin{enumerate}
		\item If $n = 2m + 1$, then $\exteriorSquareLFunction{s}{\localFieldRepresentation} = 1$.
		\item If $n = 2m$, then $\exteriorSquareLFunction{s}{\localFieldRepresentation} = \frac{1}{p\left(q^{-s}\right)}$, where $p\left(Z\right) \in \cComplex \left[Z\right]$ is a polynomial dividing $1- \centralCharacter{\localFieldRepresentation}\left(\uniformizer\right) Z^m$, satisfying $p\left(0\right) = 1$.
	\end{enumerate}
\end{thm}
Actually a more precise version of this theorem is known, in which $\exteriorSquareLFunction{s}{\localFieldRepresentation}$ is expressed in terms of twisted Shalika functionals in the even case. This is discussed in \Cref{subsubsec:twisted-shalika-functionals}, see \Cref{thm:exterior-square-L-function-in-terms-of-shalika-functionals} for instance.

\subsubsection{Level zero supercuspidal representations}

\newcommand{\levelZeroCentralCharacter}{\chi}
\newcommand{\levelZeroBeforeInduction}{\residueFieldRepresentation \cdot \levelZeroCentralCharacter}

Let $\representationDeclaration{\residueFieldRepresentation}$ be an irreducible cuspidal representation of $\GL{n}{\residueField}$. Let $\levelZeroCentralCharacter : \multiplicativegroup{\localField} \rightarrow \multiplicativegroup{\cComplex}$ be a multiplicative character, such that $\levelZeroCentralCharacter \restriction_{\multiplicativegroup{\integersring}} = \centralCharacter{\residueFieldRepresentation} \circ \quotientMap \restriction_{\multiplicativegroup{\integersring}}$. Let $\left(\levelZeroBeforeInduction, \underlyingVectorSpace{\residueFieldRepresentation}\right)$ be the representation of the group $\multiplicativegroup{\localField} \cdot \GL{n}{\integersring}$, defined by the formula $\levelZeroBeforeInduction \left( zk \right) = \levelZeroCentralCharacter \left(z\right) \residueFieldRepresentation \left(\quotientMap\left(k\right)\right) $, where $z \in \multiplicativegroup{\localField}$ and $k \in \GL{n}{\integersring}$.

\begin{thm}[{\cite[Theorem 6.2]{prasad2000representation}}]\label{thm:level-zero-representation-is-induced-from-cuspidal-reprensentation}
	Let $\localFieldRepresentation = \CompactInd{\multiplicativegroup{\localField} \cdot \GL{n}{\integersring}}{\GL{n}{\localField}}{\levelZeroBeforeInduction}$. Then $\localFieldRepresentation$ is an irreducible supercuspidal representation of $\GL{n}{\localField}$, with central character $\levelZeroCentralCharacter$.
\end{thm}

In fact, this theorem is a special case of the construction of supercuspidal representations using type theory, due to Bushnell and Kutzko. See \cite[Chapter 6]{BushnellKutzko93} for details.

We say that $\localFieldRepresentation$ is a level zero supercuspidal representation of $\GL{n}{\localField}$ constructed from the representation $\residueFieldRepresentation$, with respect to the central character $\levelZeroCentralCharacter$, or simply a level zero supercuspidal representation constructed from the representation $\residueFieldRepresentation$ (as we can recover $\levelZeroCentralCharacter$ via the central character $\centralCharacter{\localFieldRepresentation}$).

\subsubsection{Lifts}

In order to relate Jacquet-Shalika integrals of cuspidal representations of $\GL{n}{\residueField}$ to the local Jacquet-Shalika integrals of their corresponding level-zero representations, we need to be able to lift Schwartz functions and Whittaker functions to corresponding functions defined over the local field. We describe here briefly the process of doing so, leaving the details to the reader.

\subsubsection*{Lifts of Schwartz functions.}
Denote for a function $\phi_0 \in \Schwartz \left( \residueField^m \right)$ a lift $\lift{\phi_0} \in \Schwartz \left( \localField^m  \right)$ defined by $\lift{\phi_0}\left(x\right) = \begin{cases}
		\phi_0 \left( \quotientMap \left( x \right) \right) & x \in \integersring^m, \\
		0                                                   & \text{otherwise.}
	\end{cases}$

It is easy to verify the following relation between the Fourier transforms and the lifts.

\begin{prop}
	Let $\phi_0 \in \Schwartz\left( \residueField^m \right)$. Then
	$$\FourierTransformWithRespectToCharacter{\lift{\phi_0}}{\fieldCharacter} = \lift{\FourierTransformWithRespectToCharacter{\phi_0}{\residueFieldCharacter}}.$$
\end{prop}

\subsection*{Lifts of Whittaker functions}

Let $\representationDeclaration{\residueFieldRepresentation}$ be an irreducible cuspidal representation of $\GL{n}{\residueField}$. Let $\localFieldRepresentation$ be a level zero supercuspidal constructed from $\residueFieldRepresentation$.
Let $0 \ne \residueFieldWhittakerFunctional \in \Hom_{\UnipotentRadical{n}\left(\residueField\right)}\left(\residueFieldRepresentation , \residueFieldCharacter \right)$ be a Whittaker functional of $\residueFieldRepresentation$. The following proposition explains how to lift $\residueFieldWhittakerFunctional$ to a Whittaker functional of $\localFieldRepresentation$.

\begin{prop}[{\cite[Theorem 4.3]{Zelingher17}}]\label{prop:lift-of-whittaker-functional}
	Denote by $\localFieldWhittakerFunctional : \underlyingVectorSpace{\localFieldRepresentation} \rightarrow \cComplex$ the functional
	$$ \standardForm{\localFieldWhittakerFunctional}{f} = \int_{\lquot{\UnipotentRadical{n}\left( \integersring \right)}{\UnipotentRadical{n} \left( \localField  \right)}}{ \standardForm{\residueFieldWhittakerFunctional}{f\left(u\right)} \fieldCharacter^{-1}\left(u\right) } \multiplicativeMeasure{u},$$
	where $ f \in \underlyingVectorSpace{\localFieldRepresentation}$ (recall from \cref{thm:level-zero-representation-is-induced-from-cuspidal-reprensentation} that $f : \GL{n}{\localField} \rightarrow \underlyingVectorSpace{\residueFieldRepresentation}$).
	Then this integral converges and $0 \ne \localFieldWhittakerFunctional \in \Hom_{\UnipotentRadical{n} \left( \localField \right)}\left(\localFieldRepresentation, \fieldCharacter\right)$ is a Whittaker functional.
\end{prop}
\begin{proof}
	Let $f \in \underlyingVectorSpace{\pi}$. We have that $f$ is supported on a set of the form $\multiplicativegroup{\localField} \cdot \GL{n}{\integersring} \cdot K_f$, where $K_f$ is a compact subset of $\GL{n}{\localField}$. Suppose that $u \in \UnipotentSubgroup_n\left(\localField\right) \cap \Supp\left(f\right)$. Then $u = z k k'$, where $z \in \multiplicativegroup{\localField}$, $k \in \GL{n}{\integersring}$ and $k' \in K_f$. Taking determinants, we get from $z = u \left(k'\right)^{-1} k^{-1}$ that $\abs{z}^n = \abs{\det k'}^{-1}$. Since $K_f$ is a compact set, we have that there exist $c_f, C_f > 0$, such that $c_f \le \abs{\det k'}^{-1} \le C_f$ for all $k' \in K_f$. Therefore, we have for $z$ as above, $c_f^{\frac{1}{n}} \le \abs{z} \le C_f^{\frac{1}{n}}$. This implies that $$\UnipotentSubgroup_n\left(\localField\right) \cap \Supp\left(f\right) \subseteq \left\{ z \in \multiplicativegroup{\localField} \mid c_f^{\frac{1}{n}} \le \abs{z} \le C_f^{\frac{1}{n}} \right\} \cdot \GL{n}{\integersring} \cdot K_f.$$
	The right hand side is a compact subset of $\GL{n}{\localField}$ as a product of compact subsets, and the left hand side is a closed subset, and hence compact. We proved that the integral defining $\localFieldWhittakerFunctional$ is over a compact domain, and therefore converges.
	
	The functional $\localFieldWhittakerFunctional$ is not zero: let $v_0 \in \underlyingVectorSpace{\residueFieldRepresentation}$. We consider the function $f_{v_0} : \GL{n}{\localField} \rightarrow \underlyingVectorSpace{\residueFieldRepresentation}$ defined by $$f_{v_0}\left(g\right) = \begin{cases}
	\centralCharacter{\localFieldRepresentation}\left(z\right) \residueFieldRepresentation\left(\quotientMap\left(k\right)\right) v_0  & g = zk,\, z \in \multiplicativegroup{\localField},\, k \in \GL{n}{\integersring},\\
	0 & \text{otherwise}.
	\end{cases}$$
	Then $f_{v_0} \in \underlyingVectorSpace{\localFieldRepresentation}$ with $\Supp f_{v_0} = \multiplicativegroup{\localField} \cdot \GL{n}{\integersring}$. Since $\Supp f_{v_0} \cap \UnipotentSubgroup_n \left(\localField\right) = \UnipotentSubgroup_n\left(\integersring\right)$, we have that  $$\standardForm{\localFieldWhittakerFunctional}{f_{v_0}} = \int_{\lquot{\UnipotentRadical{n}\left( \integersring \right)}{\UnipotentRadical{n} \left( \integersring  \right)}}{ \standardForm{\residueFieldWhittakerFunctional}{f_{v_0}\left(u\right)} \fieldCharacter^{-1}\left(u\right) } \multiplicativeMeasure{u} = \standardForm{\residueFieldWhittakerFunctional}{f_{v_0}\left(\IdentityMatrix{n}\right)} = \standardForm{\residueFieldWhittakerFunctional}{v_0}.$$
	Choosing $v_0$ such that $\standardForm{\residueFieldWhittakerFunctional}{v_0} \ne 0$, we get that $\standardForm{\localFieldWhittakerFunctional}{f_{v_0}} \ne 0$, and therefore $\localFieldWhittakerFunctional \ne 0$, as required.
	
	The functional $T$ is a Whittaker functional: let $f \in \underlyingVectorSpace{\pi}$, $u_0 \in \UnipotentSubgroup_n \left(\localField\right)$ and $g \in \GL{n}{\localField}$. Then $\left(\localFieldRepresentation\left(u_0\right)f\right)\left(g\right) = f\left(g u_0\right)$. Therefore
	$$ \standardForm{\localFieldWhittakerFunctional}{\localFieldRepresentation\left(u_0\right) f} = \int_{\lquot{\UnipotentRadical{n}\left( \integersring \right)}{\UnipotentRadical{n} \left( \localField  \right)}}{ \standardForm{\residueFieldWhittakerFunctional}{f\left(u u_0\right)} \fieldCharacter^{-1}\left(u\right) } \multiplicativeMeasure{u}.$$
	Substituting $u' = u u_0$, we get
	$$ \standardForm{\localFieldWhittakerFunctional}{\localFieldRepresentation\left(u_0\right) f} = \int_{\lquot{\UnipotentRadical{n}\left( \integersring \right)}{\UnipotentRadical{n} \left( \localField  \right)}}{ \standardForm{\residueFieldWhittakerFunctional}{f\left(u'\right)} \fieldCharacter^{-1}\left(u' u_0^{-1}\right) } \multiplicativeMeasure{u} = \fieldCharacter\left( u_0 \right) \standardForm{\localFieldWhittakerFunctional}{f},$$
	as required.
\end{proof}

Using the lifted Whittaker functional $\localFieldWhittakerFunctional$, we are now able to define a lift of a Whittaker function.
\begin{defn}
	Let $W_0 \in \WhittakerModelOfResidueFieldRepresentation$. Let $v_{W_0} \in \underlyingVectorSpace{\residueFieldRepresentation}$ be the unique vector such that $W_0\left(g\right) = \standardForm{\residueFieldWhittakerFunctional}{\residueFieldRepresentation\left(g\right)v_{W_0}}$, for every $g \in \GL{n}{\residueField}$. Let $f_{W_0} \in \underlyingVectorSpace{\localFieldRepresentation}$ be defined as $$f_{W_0} \left( g \right) =
		\begin{cases}
			\centralCharacter{\localFieldRepresentation}\left(z\right) \residueFieldRepresentation \left( \quotientMap\left(k\right) \right) v_{W_0} & g = zk,\,z \in \multiplicativegroup{\localField},\,k\in\GL{n}{\integersring} \\
			0                                                                                                                                        & \text{otherwise}
		\end{cases}.$$
	We define $\lift{W_0} \in \WhittakerModelOfLocalFieldRepresentation$ by $\lift{W_0} \left(g\right) = \standardForm{\localFieldWhittakerFunctional}{\localFieldRepresentation\left(g\right) f_{W_0}}$ for $g \in \GL{n}{\localField}$.
\end{defn}

We have that $\lift{W_0}$ is given by a simple formula:

\begin{prop}[{\cite[Proposition 4.4]{Zelingher17}}]\label{prop:support-of-whittaker-lift}
	$\lift{W_0}$ is supported on $\UnipotentRadical{n} \left( \localField \right) \cdot \multiplicativegroup{F} \cdot \GL{n}{\integersring}$ and $$\lift{W_0}\left( u_0 zk \right) = \fieldCharacter\left( u_0 \right) \cdot \centralCharacter{\localFieldRepresentation}\left(z\right) \cdot W_0\left( \quotientMap\left( k \right) \right),$$ for any $u_0 \in \UnipotentRadical{n} \left( \localField \right)$, $z \in \multiplicativegroup{\localField}$, $k \in \GL{n}{\integersring}$.
\end{prop}
\begin{proof}
	Let $g \in \GL{n}{\localField}$ with $f\left(g\right) \ne 0$. Then
	$$ \lift{W_0}\left(g\right) = \standardForm{\localFieldWhittakerFunctional}{\localFieldRepresentation\left(g\right) f_{W_0}} = \int_{\lquot{\UnipotentRadical{n}\left( \integersring \right)}{\UnipotentRadical{n} \left( \localField  \right)}}{ \standardForm{\residueFieldWhittakerFunctional}{f_{W_0}\left(u g\right)} \fieldCharacter^{-1}\left(u\right) } \multiplicativeMeasure{u}.$$
	Since $f_{W_0}$ is supported on $\multiplicativegroup{\localField} \cdot \GL{n}{\integersring}$, we must have $u g \in \multiplicativegroup{\localField} \cdot \GL{n}{\integersring}$ for some $u \in \UnipotentSubgroup_n \left(\localField\right)$, i.e., $g \in \UnipotentSubgroup_n\left(\localField\right) \cdot \multiplicativegroup{\localField} \cdot \GL{n}{\integersring}$.
	
	Write $g = u_0 z k$ for $u_0 \in \UnipotentSubgroup_n\left(\localField\right)$, $z \in \multiplicativegroup{\localField}$, $k \in \GL{n}{\integersring}$. Then $$\lift{W_0}\left( u_0 z k\right) = \standardForm{\localFieldWhittakerFunctional}{\pi\left(u_0 z k\right) f_{W_0}} = \fieldCharacter\left( u_0 \right) \centralCharacter{\localFieldRepresentation}\left(z\right) \standardForm{\localFieldWhittakerFunctional}{\localFieldRepresentation\left(k\right) f_{W_0}},$$ where we used the fact that $\localFieldWhittakerFunctional$ is a Whittaker functional and that $\localFieldRepresentation\left(z\right) = \centralCharacter{\localFieldRepresentation}\left(z\right) \idmap_{\underlyingVectorSpace{\localFieldRepresentation}}$.
	Finally, write $$\standardForm{\localFieldWhittakerFunctional}{\localFieldRepresentation\left(k\right) f_{W_0}} = \int_{\lquot{\UnipotentRadical{n}\left( \integersring \right)}{\UnipotentRadical{n} \left( \localField  \right)}}{ \standardForm{\residueFieldWhittakerFunctional}{f_{W_0}\left(u k\right)} \fieldCharacter^{-1}\left(u\right) } \multiplicativeMeasure{u}.$$
	This integral is supported on $uk \in \multiplicativegroup{\localField} \cdot \GL{n}{\integersring}$, i.e., $u \in \multiplicativegroup{\localField} \cdot \GL{n}{\integersring}$. We have that $u \in \left(\multiplicativegroup{\localField} \cdot \GL{n}{\integersring}\right) \cap \UnipotentSubgroup_n\left(\localField\right) = \UnipotentSubgroup_n\left(\integersring\right)$, and therefore we get
	\begin{align*}
	\standardForm{\localFieldWhittakerFunctional}{\localFieldRepresentation\left(k\right) f_{W_0}} &= \int_{\lquot{\UnipotentRadical{n}\left( \integersring \right)}{\UnipotentRadical{n} \left( \integersring  \right)}}{ \standardForm{\residueFieldWhittakerFunctional}{f_{W_0}\left(u k\right)} \fieldCharacter^{-1}\left(u\right) } \multiplicativeMeasure{u} \\
	&= \standardForm{\residueFieldWhittakerFunctional}{f_{W_0}\left(k\right)} = \standardForm{\residueFieldWhittakerFunctional}{\residueFieldRepresentation\left( \quotientMap\left(k\right) \right) v_{W_0}} = W_0\left( \quotientMap\left(k\right) \right).
	\end{align*}
	Therefore we have $\lift{W_0}\left(u_0 zk\right) = \fieldCharacter\left(u_0\right) \centralCharacter{\localFieldRepresentation}\left(z\right) W_0\left( \quotientMap\left(k\right) \right)$, as required.
\end{proof}

\subsection{A relation between the Jacquet-Shalika integrals}

In this section, we find a relation between local Jacquet-Shalika integrals of level zero representations and their corresponding residue field Jacquet-Shalika integrals. Our main result is the following:

\begin{thm} \label{thm:Jacquet-Shalika-integral-of-a-level-zero-supercuspidal-representation}
	Let $n = 2m$ or $n = 2m+1$.
	Let $\representationDeclaration{\residueFieldRepresentation}$ be an irreducible cuspidal representation of $\GL{n}{\residueField}$. If $n$ is even, suppose that $\residueFieldRepresentation$ does not admit a Shalika vector. Let $\localFieldRepresentation$ be a level zero supercuspidal (irreducible) representation constructed from $\residueFieldRepresentation$. Then for every $W_0 \in \WhittakerModelOfResidueFieldRepresentation$, $ \phi_0 \in \Schwartz \left( \residueField^m \right) $, $s \in \cComplex$,
	\begin{equation*}
		\begin{split}
			\JSOfLocalFieldRepresentation{s}{\lift{W_0}}{\lift{\phi_0}} &= \JSOfResidueFieldRepresentation{W_0}{\phi_0}, \\
			\DualJSOfLocalFieldRepresentation{s}{\lift{W_0}}{\lift{\phi_0}} &= \DualJSOfResidueFieldRepresentation{W_0}{\phi_0}.
		\end{split}
	\end{equation*}

\end{thm}
\cref{thm:Jacquet-Shalika-integral-of-a-level-zero-supercuspidal-representation} implies that if $\residueFieldRepresentation$ does not admit a Shalika vector, the Jacquet-Shalika integrals of its corresponding level zero representation for fixed lifted functions result in constant elements of $\cComplex \left(q^{-s}\right)$, i.e., they are independent of $s$. 
In the case where $n$ is even, we have a modified theorem that also handles the case in which the representation $\residueFieldRepresentation$ admits a Shalika vector:
\begin{thm} \label{thm:Jacquet-Shalika-integral-of-a-level-zero-supercuspidal-representation-even-case}
	Let $\representationDeclaration{\residueFieldRepresentation}$ be an irreducible cuspidal representation of $\GL{2m}{\residueField}$. Then for every $W_0 \in \WhittakerModelOfResidueFieldRepresentation$, $ \phi_0 \in \Schwartz \left( \residueField^m \right) $, $s \in \cComplex$,
	\begin{equation*}
		\begin{split}
			\JSOfLocalFieldRepresentation{s}{\lift{W_0}}{\lift{\phi_0}}
			=&\, \JSOfResidueFieldRepresentation{W_0}{\phi_0} +
			q^{-ms}
			\centralCharacter{\localFieldRepresentation} \left( \uniformizer \right)
			\phi_0\left(0\right)  \representationLFunction{ms}{\centralCharacter{\localFieldRepresentation}}
			\JSOfResidueFieldRepresentation{W_0}{1}, \\
			\DualJS{\localFieldRepresentation}{\fieldCharacter}\left(s, \lift{W_0}, \lift{\phi_0}\right)
			=&\,  \DualJS{\residueFieldRepresentation}{\residueFieldCharacter}\left(W_0, \phi_0 \right)  + \\
			& + q^{-m \left(1-s\right)}
			\centralCharacter{\localFieldRepresentation}^{-1} \left( \uniformizer \right)
			\FourierTransformWithRespectToCharacter{\phi_0}{\residueFieldCharacter}\left(0\right)  \representationLFunction{m \left(1-s\right)}{\centralCharacter{\localFieldRepresentation}^{-1}}
			\JSOfResidueFieldRepresentation{W_0}{1}.
		\end{split}
	\end{equation*}
\end{thm}
By \cref{thm:equivalent-conditions-for-a-shalika-vector}, we have that \cref{thm:Jacquet-Shalika-integral-of-a-level-zero-supercuspidal-representation-even-case} implies \cref{thm:Jacquet-Shalika-integral-of-a-level-zero-supercuspidal-representation} in the even case.

As usual, we will treat the even case and the odd case separately. Before the proof, we recall the following lemmas from \cite{jacquet11exterior}, which will be useful throughout the proof.

\begin{lem}[{\cite{jacquet11exterior}}, Section 5, Proposition 4]\label{lem:diagonal-part-of-iwasawa-decomposition-of-u_Z}
	Let $Z \in \NilpotentLowerTriangular_m\left({\localField}\right)$ be a lower triangular nilpotent matrix. Denote $u_Z = \evenPermutationMatrix \SmallShalikaUnipotentElement{Z} \evenPermutationMatrix^{-1}$ - a lower triangular unipotent matrix. Let $u_Z = n_Z \cdot t_Z \cdot k_Z$ be an Iwasawa decomposition of $u_Z$, where $n_Z \in \UnipotentRadical{2m}$, $k_Z \in \GL{2m}{\integersring}$ and $t_Z$ is a diagonal matrix. Write $t_Z = \diag \left(t_1, \dots, t_{2m}\right)$. Then $\abs{t_1} = \abs{t_{2m}} = 1$, and $\abs{t_{i}} \ge 1$ for odd $i$, and $\abs{t_{i}} \le 1$ for even $i$.
\end{lem}

\begin{lem}[{\cite{jacquet11exterior}}, Section 5, Proposition 5]\label{lem:norm-of-Z-in-terms-of-tz}
	For the same notations as in \cref{lem:diagonal-part-of-iwasawa-decomposition-of-u_Z}, we have that there exists some $\alpha > 0$, such that $ \Norm{Z}^{\frac{1}{\alpha}} \le \prod_{\substack{1 \le i \le 2m \\
				i \text{ is odd}}} \abs{t_i} $, where $\Norm{Z} = \max_{1 \le i,j \le m} \abs{z_{i j}}$.
\end{lem}

The following lemma will be useful for the proofs of \cref{thm:Jacquet-Shalika-integral-of-a-level-zero-supercuspidal-representation} and \cref{thm:Jacquet-Shalika-integral-of-a-level-zero-supercuspidal-representation-even-case}. Let $\diagonalSubgroup_{m}$ be the diagonal subgroup of $\GL{m}{\localField}$.

\begin{lem}\label{lem:support-of-level-zero-jacquet-shalika-integrals}
	Suppose that $\evenPermutationMatrix \SmallShalikaUnipotentElement{X} \SmallShalikaDiagonalElement{a} = \lambda \cdot u \cdot k$, where $a = \diag\left(a_1,\dots,a_m\right) \in \diagonalSubgroup_m$ is an invertible diagonal matrix, $X \in \NilpotentLowerTriangular_m \left(\localField\right)$ is a lower triangular nilpotent matrix,  $\lambda \in \multiplicativegroup{\localField}$, $u \in \UnipotentRadical{2m}$ and $k \in \maximalCompactSubgroup_{2m} = \GL{2m}{\integersring}$. Then
	\begin{enumerate}
		\item{$\abs{a_1} = \dots = \abs{a_m} = \abs{\lambda}$.}
		\item{$X \in \SquareMat{m}{\integersring}.$}
	\end{enumerate}
\end{lem}
\begin{proof}
	Denote $Z = a^{-1} X a$, and $u_Z = \evenPermutationMatrix \SmallShalikaUnipotentElement{Z} \evenPermutationMatrix^{-1}$. Also denote $b = \evenPermutationMatrix \SmallShalikaDiagonalElement{a} \evenPermutationMatrix^{-1} = \diag\left(a_1,a_1,a_2,a_2\dots,a_m,a_m\right)$. Then we have that $b u_Z \evenPermutationMatrix = \lambda u k$. Writing $u_Z = n_Z t_Z k_Z$ as in \cref{lem:diagonal-part-of-iwasawa-decomposition-of-u_Z}, we get that $\lambda^{-1} b t_Z = \left(b n_Z^{-1} b^{-1} u\right) \cdot \left(k \evenPermutationMatrix^{-1} k_Z^{-1}\right) $. Since $\diagonalSubgroup_{2m} \cap \left( \UnipotentRadical{2m} \cdot \maximalCompactSubgroup_{2m} \right) = \diagonalSubgroup_{2m} \cap \maximalCompactSubgroup_{2m} = \left(\multiplicativegroup{\integersring}\right)^{2m}$, we get that $\lambda^{-1} b t_Z $ is a diagonal matrix having units on its diagonal. Writing $t_Z = \diag \left(t_1, \dots, t_{2m}\right)$, we have that $\abs{t_{2i-1}} = \abs{t_{2i}} = \frac{\abs{\lambda}}{\abs{a_i}}$, for every $1 \le i \le m$. By \cref{lem:diagonal-part-of-iwasawa-decomposition-of-u_Z}, we get that $\abs{t_i} = 1$ for every $1 \le i \le 2m$, and therefore $\abs{a_i} = \abs{\lambda}$, for every $1 \le i \le m$. Finally, by \cref{lem:norm-of-Z-in-terms-of-tz}, we get that $Z \in \SquareMat{m}{\integersring}$. Since $\frac{\abs{a_i}}{\abs{a_j}} = \frac{\abs{\lambda}}{\abs{\lambda}}= 1$, for every $1 \le i, j \le m$, we get that $X \in \SquareMat{m}{\integersring}$, as required.
\end{proof}

We now move to prove \cref{thm:Jacquet-Shalika-integral-of-a-level-zero-supercuspidal-representation} and \cref{thm:Jacquet-Shalika-integral-of-a-level-zero-supercuspidal-representation-even-case}. Throughout both proofs, we will use the following symbols:

$W = \lift{W_0}$, $\phi = \lift{\phi_0}$. $A = \diagonalSubgroup_{m}$ is the diagonal subgroup of $G$. $K = \GL{m}{\integersring}$. $\NilpotentLowerTriangular = \NilpotentLowerTriangular_m \left(\localField\right) \le \SquareMat{m}{\localField}$ is the subspace consisting of lower triangular nilpotent matrices.

Let $X \in \NilpotentLowerTriangular$. Also let $g = a k$, where $a = \diag \left(a_1, \dots, a_{m}\right) \in \diagonalSubgroup$, $k \in \maximalCompactSubgroup$. Then by the Iwasawa decomposition of $\GL{m}{\localField}$, we have $\multiplicativeMeasure{g} = \rightHaarMeasureModulus{\borelSubgroup_{m}} \left(a\right) \prod_{i = 1}^{m}{\multiplicativeMeasure{a_{i}}} \cdot \multiplicativeMeasure{k} $, where $\rightHaarMeasureModulus{\borelSubgroup_{m}} \left(a\right) = \prod_{1 \le i < j \le {m}} \abs{\frac{a_j}{a_i}}$.

\begin{proof}[Proof of \cref{thm:Jacquet-Shalika-integral-of-a-level-zero-supercuspidal-representation-even-case}]
	We prove only the equation regarding $\JS{\localFieldRepresentation}{\fieldCharacter}$. The proof of the equation regarding $\DualJS{\localFieldRepresentation}{\fieldCharacter}$ is similar. It can also be deduced from the first equation.

	The proof consists of two parts. In the first part we find the supports of $g$, $X$ in the Jacquet-Shalika integral. In the second part, we evaluate the integral on the supports.

	Suppose that $\evenPermutationMatrix \SmallShalikaUnipotentElement{X} \SmallShalikaDiagonalElement{a} \SmallShalikaDiagonalElement{k} \in \Supp \left(W\right) \subseteq \multiplicativegroup{F} \UnipotentRadical{2m} \maximalCompactSubgroup_{2m} $ (see \cref{prop:support-of-whittaker-lift}). Then by \cref{lem:support-of-level-zero-jacquet-shalika-integrals}, we have that $\abs{a_i} = \abs{a_m}$, for every $1 \le  i \le m$, and $X \in \SquareMat{m}{\integersring}$. Therefore, we get that $a_i = u_i \cdot a_m$, where $u_i \in \multiplicativegroup{\integersring}$, $\multiplicativeMeasure{u_i} = \multiplicativeMeasure{a_i} $, for every $1 \le i \le m-1$. We also get $\rightHaarMeasureModulus{\borelSubgroup_{m}}\left(a\right) = \prod_{1 \le i < j \le {m}} \abs{\frac{a_j}{a_i}} = 1$.

	Therefore, the Jacquet Shalika integral is integrated on $X \in \NilpotentLowerTriangular_m\left(\integersring\right)$, $g = a k$, where $k \in \maximalCompactSubgroup$, $a = a_m \cdot \diag \left(u_1, \dots, u_{m-1}, 1\right)$, $u_i \in \multiplicativegroup{\integersring}$ for every $1 \le i \le m -1$, $\multiplicativeMeasure{g} = \multiplicativeMeasure{a_m} \cdot \prod_{i=1}^{m-1}{\multiplicativeMeasure{u_i}} \cdot \multiplicativeMeasure{k}$, and by replacing the variable $k$ with $\diag\left(u_1, \dots, u_{m-1}, 1\right)^{-1} \cdot k$, we have that $ \JSOfLocalFieldRepresentation{s}{W}{\phi} $ is given by
	$$ \int_{\multiplicativegroup{\localField}} \int_{\maximalCompactSubgroup} \int_{\NilpotentLowerTriangular_m \left(\integersring\right)} W \left(\evenPermutationMatrix \ShalikaUnipotentElement{X} \ShalikaDiagonalElement{k} \right) \phi \left( \lastrowvector k a_m \right) \abs{a_m}^{ms} \centralCharacter{\localFieldRepresentation}\left(a_m\right)  d X \multiplicativeMeasure{k} \multiplicativeMeasure{a_m}.$$
	Since $\phi = \lift{\phi_0}$ is a lift of the Schwartz function $\phi_0$, we have that for a fixed $k$, $\phi \left(\lastrowvector k a_m \right) = 0$ for
	$\abs{a_m} > 1$ and $\phi \left(\lastrowvector k a_m \right) = \phi_0 \left(0\right)$ for $\abs{a_m} < 1$.
	Therefore
	$$ \int_{\multiplicativegroup{\localField}} \phi \left( \lastrowvector k a_m \right) \abs{a_m}^{ms} \centralCharacter{\localFieldRepresentation}\left(a_m\right) \multiplicativeMeasure{a_m} =  \int_{\multiplicativegroup{\integersring}} \phi \left( \lastrowvector k a_m \right) \centralCharacter{\localFieldRepresentation}\left(a_m\right) \multiplicativeMeasure{a_m} + \phi_0\left(0\right) I\left(s\right),$$
	where
	$$I\left(s\right) = \sum_{i=1}^{\infty} \int_{\uniformizer^i \cdot \multiplicativegroup{\integersring}} \abs{a_m}^{ms} \centralCharacter{\localFieldRepresentation}\left(a_m\right) \multiplicativeMeasure{a_m} = \int_{\multiplicativegroup{\integersring}} \centralCharacter{\localFieldRepresentation} \left( a_m \right) \multiplicativeMeasure{a_m} \cdot \sum_{i=1}^{\infty}{q^{-i m s} \centralCharacter{\localFieldRepresentation}\left(\uniformizer\right)}^{i}.$$
	Therefore, $I\left(s\right) = q^{-ms} \centralCharacter{\localFieldRepresentation} \left(\uniformizer\right) \representationLFunction{ms}{\centralCharacter{\localFieldRepresentation}}$ if $\centralCharacter{\localFieldRepresentation}$ is unramified (i.e., $\centralCharacter{\localFieldRepresentation}\restriction_{\multiplicativegroup{\integersring}}$ trivial, which happens if and only if $\centralCharacter{\residueFieldRepresentation}$ is trivial), and otherwise $\int_{\multiplicativegroup{\integersring}} \centralCharacter{\localFieldRepresentation} \left( a_m \right) \multiplicativeMeasure{a_m} = 0$, so $I\left(s\right)=0$.

	We are left to evaluate
	$$ \int_{\multiplicativegroup{\integersring}} \int_{\maximalCompactSubgroup} \int_{\NilpotentLowerTriangular_m \left(\integersring\right)} W \left(\evenPermutationMatrix \ShalikaUnipotentElement{X} \ShalikaDiagonalElement{k} \right)  \centralCharacter{\localFieldRepresentation}\left(a_m\right) \cdot \left(\phi \left( \lastrowvector k a_m \right) + \phi_0\left(0\right) I\left(s\right) \right)  d X \multiplicativeMeasure{k} \multiplicativeMeasure{a_m}.$$
	Since $W$, $\phi$ and $\centralCharacter{\localFieldRepresentation}$ are lifts of $W_0$, $\phi_0$ and $\centralCharacter{\residueFieldRepresentation}$ respectively, and the expression is constant on the quotient spaces $\rquot{\GL{m}{\integersring}}{1 + \uniformizer\SquareMat{m}{\integersring}} \isomorphic \GL{m}{\residueField}$, $\rquot{\NilpotentLowerTriangular \left( \integersring \right)}{\NilpotentLowerTriangular \left( \maximalideal \right)} \isomorphic \NilpotentLowerTriangular \left( \residueField \right)$ and $\rquot{\multiplicativegroup{\integersring}}{1 + \maximalideal} \isomorphic \multiplicativegroup{\residueField}$, we get that this sum equals
	\begin{align*}
		\frac{1}{\sizeof{\multiplicativegroup{\residueField}}} \frac{1}{\sizeof{\GL{m}{\residueField}}} \frac{1}{\sizeof{\NilpotentLowerTriangular \left(\residueField\right)}} \sum_{a_m \in \multiplicativegroup{\residueField}}\sum_{k \in \GL{m}{\residueField}}\sum_{
		X \in \NilpotentLowerTriangular \left(\residueField\right)} &
		W_0 \left(\evenPermutationMatrix \ShalikaUnipotentElement{X} \ShalikaDiagonalElement{k} \right) \cdot                                                                                                                                       \\
		                                                            & \cdot \centralCharacter{\residueFieldRepresentation}\left(a_m\right) \cdot \left(\phi_0 \left( \lastrowvector k a_m \right) + \phi_0 \left(0\right) I\left(s\right) \right) .
	\end{align*}
	By replacing the variable $k$ with $k = k' a_m^{-1}$, we get
	$$ \frac{1}{\sizeof{\GL{m}{\residueField}}} \frac{1}{\sizeof{\NilpotentLowerTriangular \left(\residueField\right)}} \sum_{k' \in \GL{m}{\residueField}}\sum_{X \in \NilpotentLowerTriangular \left(\residueField\right)}  W_0 \left(\evenPermutationMatrix \ShalikaUnipotentElement{X} \ShalikaDiagonalElement{k'} \right)
		\left( \phi_0 \left( \lastrowvector k' \right) + \phi_0 \left(0\right) I\left(s\right) \right).$$
	Since this expression is constant on cosets of  $k' \in \lquot{\UnipotentRadical{m}\left(\residueField\right)}{\GL{m}{\residueField}}$ and since $\lquot{\UpperTriangularAdditive_m\left(\residueField\right)}{\SquareMat{m}{\residueField}} \isomorphic \NilpotentLowerTriangular_m \left( \residueField \right)$, we get that $$ \JSOfLocalFieldRepresentation{s}{W}{\phi} = \JSOfResidueFieldRepresentation{W_0}{\phi_0} + \phi_0 \left(0\right) I\left(s\right) \JSOfResidueFieldRepresentation{W_0}{1}.$$

	Finally, we claim that $$ \phi_0 \left(0\right) I\left(s\right) \JSOfResidueFieldRepresentation{W_0}{1} = \phi_0 \left(0\right) q^{-ms} \centralCharacter{\localFieldRepresentation} \left(\uniformizer\right) \representationLFunction{ms}{\centralCharacter{\localFieldRepresentation}} \JSOfResidueFieldRepresentation{W_0}{1}.$$
	If $\centralCharacter{\residueFieldRepresentation}$ is trivial, we already saw that this is true. If $\centralCharacter{\residueFieldRepresentation}$ is non-trivial, $\residueFieldRepresentation$ does not admit a Shalika vector, and we get from \cref{thm:equivalent-conditions-for-a-shalika-vector}, that $\JSOfResidueFieldRepresentation{W}{1} = 0$, and therefore the result follows.
\end{proof}

\begin{proof}[Proof of \cref{thm:Jacquet-Shalika-integral-of-a-level-zero-supercuspidal-representation}]
	We are left to prove the odd case, $n = 2m+1$. Again, we only prove the equation regarding $\JS{\localFieldRepresentation}{\fieldCharacter}$.
	Let $Z \in \Mat{1}{m}{\localField}$. Suppose that $Z \in \Supp \left(\phi\right) \subseteq \integersring^m$. Therefore $Z$ is integrated on $\Mat{1}{m}{\integersring}$. Suppose that $$\oddPermutationMatrix \ShalikaUnipotentElementOdd{X} \ShalikaDiagonalElementOdd{a} \ShalikaDiagonalElementOdd{k} \ShalikaLowerUnipotentElementOdd{Z} \in \Supp \left(W\right).$$
	Since $\Supp \left(W\right) \subseteq \UnipotentRadical{2m+1} \cdot \multiplicativegroup{F} \cdot \maximalCompactSubgroup_{2m+1}$, we get that $\oddPermutationMatrix \SmallShalikaUnipotentElementOdd{X} \SmallShalikaDiagonalElementOdd{a} = u' \cdot \lambda \cdot k'$, where $u' \in \UnipotentRadical{2m+1}$, $\lambda \in \multiplicativegroup{\localField}$, $k' \in \maximalCompactSubgroup_{2m+1}$. Writing $\oddPermutationMatrix = \smallDiagTwo{\evenPermutationMatrix}{1}$, comparing the last row of the matrices, and modifying $\lambda$ by a unit, we get that $k' = \left(\begin{smallmatrix}
				k'' & Y\\
				0 & 1
			\end{smallmatrix}\right)$, where $k'' \in \maximalCompactSubgroup_{2m}$ and $Y \in \Mat{2m}{1}{\integersring}$. Replacing $u'$ by $u' \left(\begin{smallmatrix}
				\IdentityMatrix{2m} & Y\\
				0 & 1
			\end{smallmatrix}\right)$, we may assume that $k' = \smallDiagTwo{k''}{1}$, which implies that $u' = \smallDiagTwo{u''}{1}$, where $u'' \in \UnipotentRadical{2m}$. We also get $\lambda = 1$. Therefore we have that $\evenPermutationMatrix \SmallShalikaUnipotentElement{X} \SmallShalikaDiagonalElement{a} = 1 \cdot u''  \cdot k''$. By \cref{lem:support-of-level-zero-jacquet-shalika-integrals}, this implies that $X \in \SquareMat{m}{\integersring}$ and $\abs{a_i} = 1$, for every $1 \le i \le m$, which implies that $\rightHaarMeasureModulus{\borelSubgroup_{m}}\left(a\right) = 1$. Therefore, $\JSOfLocalFieldRepresentation{s}{W}{\phi}$ is given by

	\begin{equation*}
		\begin{split}
			\int_{\maximalCompactSubgroup}\int_{\NilpotentLowerTriangular_m\left(\integersring\right)} \int_{\Mat{1}{m}{\integersring}}
			&W\left(
			\oddPermutationMatrix
			\ShalikaUnipotentElementOdd{X}
			\ShalikaDiagonalElementOdd{k}
			\ShalikaLowerUnipotentElementOdd{Z}
			\right) \\
			& \cdot
			\fieldCharacter\left(-\trace X\right)
			\phi \left( Z \right)
			dZ\,dX\,\multiplicativeMeasure{k}.
		\end{split}
	\end{equation*}
	Since the integrand is constant on $k \in \rquot{\maximalCompactSubgroup}{1 + \uniformizer \SquareMat{m}{\integersring}} \isomorphic \GL{m}{\residueField}$, $X \in \rquot{\NilpotentLowerTriangular_m\left(\integersring\right)}{ \NilpotentLowerTriangular_m\left(\maximalideal\right)} \isomorphic \NilpotentLowerTriangular_m\left(\residueField\right)$ and $Z \in \rquot{\Mat{1}{m}{\integersring}}{ \Mat{1}{m}{\maximalideal
			}} \isomorphic \Mat{1}{m}{\residueField}$, and since $W$ and $\phi$ are lifts of $W_0$ and $\phi_0$ respectively, we have that this expression equals \begin{align*}
		 & \frac{1}{\sizeof{\GL{m}{\residueField}} \sizeof{\NilpotentLowerTriangular_m \left( \residueField \right)} \sizeof{\Mat{1}{m}{\residueField}}} \sum_{k \in \GL{m}{\residueField}} \sum_{X \in \NilpotentLowerTriangular_m\left(\residueField\right)} \sum_{Z \in \Mat{1}{m}{\residueField}} \\
		 & W_0\left( \oddPermutationMatrix \ShalikaUnipotentElementOdd{X} \ShalikaDiagonalElementOdd{k} \ShalikaLowerUnipotentElementOdd{Z}
		\right) \fieldCharacter\left(-\trace X\right) \phi_0 \left( Z \right),
	\end{align*}
	and it equals $\JSOfResidueFieldRepresentation{W_0}{\phi_0}$, as the summand is constant on the cosets of $\lquot{\UnipotentRadical{m}\left(\residueField\right)}{\GL{m}{\residueField}} $, and since $\NilpotentLowerTriangular_m\left(\residueField\right) \isomorphic \lquot{\UpperTriangularAdditive_m\left(\residueField\right)}{\SquareMat{m}{\residueField}}$. Hence, we proved that equality $$\JSOfLocalFieldRepresentation{s}{\lift{W_0}}{\lift{\phi_0}} = \JSOfResidueFieldRepresentation{W_0}{\phi_0}.$$

\end{proof}

\subsection{The modified functional equation in the even case}

As a result of \cref{thm:Jacquet-Shalika-integral-of-a-level-zero-supercuspidal-representation-even-case}, we get the following modified functional equation, a generalization of \cref{thm:functional-equation-finite-field}, this time valid for all irreducible cuspidal representations of $\GL{2m}{\finiteField}$, regardless whether they admit a Shalika vector.

\begin{thm}[The modified functional equation]\label{thm:modified-functional-equation}
	Let $\finiteField$ be a finite field with $\sizeof{\finiteField} = q$. Let $\residueFieldCharacter : \finiteField \rightarrow \multiplicativegroup{\cComplex}$ be a non-trivial (additive) character. Let $\representationDeclaration{\residueFieldRepresentation}$ be an irreducible cuspidal representation of $\GL{m}{\finiteField}$. Then there exists $\modifiedGammaFactorOfResidueField{s} \in \cComplex \left( q^{-s} \right)$, such that for every $W_0 \in \WhittakerModelOfResidueFieldRepresentation$ and $\phi_0 \in \Schwartz\left(\finiteField^m\right)$, we have
	\begin{align*}
		\DualJS{\residueFieldRepresentation}{\residueFieldCharacter}\left(W_0, \phi_0 \right)  + q^{-m \left(1-s\right)}
		\FourierTransformWithRespectToCharacter{\phi_0}{\residueFieldCharacter}\left(0\right)  \representationLFunction{m \left(1-s\right)}{1}
		\JSOfResidueFieldRepresentation{W_0}{1} = \\
		\modifiedGammaFactorOfResidueField{s} \cdot \left( \JSOfResidueFieldRepresentation{W_0}{\phi_0} +
		q^{-ms}
		\phi_0\left(0\right)  \representationLFunction{ms}{1}
		\JSOfResidueFieldRepresentation{W_0}{1} \right).
	\end{align*}
\end{thm}

\begin{proof}
	If $\residueFieldRepresentation$ does not admit a Shalika vector, then from \cref{thm:equivalent-conditions-for-a-shalika-vector} for every $W_0 \in \WhittakerModelOfResidueFieldRepresentation$, we have that $\JSOfResidueFieldRepresentation{W_0}{1} = 0$, and therefore we get the same functional equation as in \cref{thm:functional-equation-finite-field}.

	Suppose that $\residueFieldRepresentation$ admits a Shalika vector. Then by \cref{rem:shalika-vector-trivial-character}, $\residueFieldRepresentation$ has a trivial central character. Choose any local field $\localField$ with $\finiteField$ as its residue field, and $\fieldCharacter : \localField \rightarrow \multiplicativegroup{\cComplex}$, an additive character, such that $\fieldCharacter\restriction_{\integersring} = \residueFieldCharacter \circ \quotientMap$. Let $\representationDeclaration{\localFieldRepresentation}$ be the level zero supercuspidal representation constructed from $\residueFieldRepresentation$, with respect to the trivial central character. The statement now follows from \cref{thm:Jacquet-Shalika-integral-of-a-level-zero-supercuspidal-representation-even-case} and \cref{thm:local-field-functional-equation}.
\end{proof}

\subsection{Exterior square gamma factors for level zero supercuspidal representations}
Let $\representationDeclaration{\residueFieldRepresentation}$ be an irreducible cuspidal representation of $\GL{n}{\residueField}$, and let $\representationDeclaration{\localFieldRepresentation}$ be a level zero representation of $\GL{n}{\localField}$ constructed from $\residueFieldRepresentation$.

As a corollary of \cref{thm:Jacquet-Shalika-integral-of-a-level-zero-supercuspidal-representation} we obtain the following main theorems of the paper.

\begin{thm}\label{thm:equality-of-gamma-factors-non-shalika-vector}
	If $\residueFieldRepresentation$ does not admit a Shalika vector, then $\gammaFactorOfLocalField{s}$ is an invertible constant (i.e., independent of $s$), and $\gammaFactorOfLocalField{s} = \gammaFactorOfResidueField$.
	Furthermore, $\exteriorSquareLFunction{s}{\localFieldRepresentation} = 1$, $\epsilonFactorOfLocalField{s} = \gammaFactorOfResidueField$.
\end{thm}

\begin{proof}
	We can choose $W_0 \in \WhittakerModelOfResidueFieldRepresentation$ and $\phi_0 \in \Schwartz \left(\residueField^m\right)$, such that $ \JSOfResidueFieldRepresentation{W_0}{\phi_0} \ne 0$ (for instance take the functions in \cref{prop:jacquet-shalika-integral-of-bessel-function-even-case}, \cref{prop:jacquet-shalika-integral-of-bessel-function-odd-case}). By \cref{thm:Jacquet-Shalika-integral-of-a-level-zero-supercuspidal-representation}, we have that
	$\JSOfLocalFieldRepresentation{s}{\lift{W_0}}{\lift{\phi_0}} = \JSOfResidueFieldRepresentation{W_0}{\phi_0}$ and $\DualJSOfLocalFieldRepresentation{s}{\lift{W_0}}{\lift{\phi_0}} = \DualJSOfResidueFieldRepresentation{W_0}{\phi_0}$.
	We get from the functional equations in \cref{thm:functional-equation-finite-field} and \cref{thm:local-field-functional-equation} that
	$$\gammaFactorOfLocalField{s} = \frac{\DualJSOfLocalFieldRepresentation{s}{\lift{W_0}}{\lift{\phi_0}}}{\JSOfLocalFieldRepresentation{s}{\lift{W_0}}{\lift{\phi_0}}} = \frac{\DualJSOfResidueFieldRepresentation{W_0}{\phi_0}}{\JSOfResidueFieldRepresentation{W_0}{\phi_0}} = \gammaFactorOfResidueField.$$
	This proves the result regarding the gamma factors.
	
	If $n = 2m+1$, we have from \cref{thm:exterior-square-polynomial-is-a-divisor} that $\exteriorSquareLFunction{s}{\localFieldRepresentation} = 1$. Suppose $n = 2m$, and denote $\exteriorSquareLFunction{s}{\localFieldRepresentation} = \frac{1}{p_1 \left( q^{-s} \right)}$, $\exteriorSquareLFunction{s}{\Contragradient{\localFieldRepresentation}} = \frac{1}{p_2 \left( q^{-s} \right)}$, where $p_1\left(Z\right), p_2\left(Z\right) \in \cComplex \left[Z\right]$ are polynomials with $p_1\left(0\right)=p_2\left(0\right)=1$. Since $\gammaFactorOfLocalField{s}$ is a constant, by \cref{thm:local-field-functional-equation} we must have that $$\frac{p_1\left(q^{-s}\right)}{p_2 \left(q^{-\left(1-s\right)}\right)} = \frac{\exteriorSquareLFunction{1-s}{\Contragradient{\localFieldRepresentation}}}{\exteriorSquareLFunction{s}{{\localFieldRepresentation}}} = c \cdot q^{ks},$$ where $k \in \zIntegers$ and $c \in \multiplicativegroup{\cComplex}$. This implies that $p_1 \left(Z\right)$ and $p_2 \left(q^{-1} Z^{-1}\right)$ have the same non-zero roots. By \cref{thm:exterior-square-polynomial-is-a-divisor}, we have that $p_1 \left(Z\right)$ divides $1 - \centralCharacter{\localFieldRepresentation}\left(\uniformizer\right) Z^m$ and $p_2$ divides $1 - \centralCharacter{\localFieldRepresentation}\left(\uniformizer\right)^{-1} Z^m$, and therefore $p_1 \left(Z\right)$ and $p_2 \left(q^{-1} Z^{-1}\right)$ can't have mutual roots, as roots $r$ of $p_1 \left(Z\right)$ satisfy $r^m = \centralCharacter{\localFieldRepresentation}\left(\uniformizer\right)^{-1}$, while roots $r'$ of $p_2 \left(q^{-1}Z^{-1}\right)$ satisfy $r'^m = q^{-m} \centralCharacter{\localFieldRepresentation}\left(\uniformizer\right)^{-1}$. Therefore $p_1 \left(Z\right), p_2 \left(Z\right)$ are constants and $p_1\left(Z\right) = p_2\left(Z\right) = 1$, which implies that $\exteriorSquareLFunction{s}{\localFieldRepresentation} = \frac{1}{p \left(q^{-s}\right)} = 1$.
	The result regarding $\epsilonFactorOfLocalField{s}$ now follows from the equation in \cref{thm:local-field-functional-equation}.
\end{proof}

\begin{thm}\label{thm:equality-of-gamma-factors-with-shalika-vector}
	If $n = 2m$ and $\residueFieldRepresentation$ admits a Shalika vector, then $$ \gammaFactorOfLocalField{s} = \frac{q^{ms}}{q^{\frac{m}{2}}
			\centralCharacter{\localFieldRepresentation}\left( \uniformizer\right)}
		\cdot
		\frac{\representationLFunction{m\left(1-s\right)}{\centralCharacter{\localFieldRepresentation}^{-1}}}{\representationLFunction{ms}{\centralCharacter{\localFieldRepresentation}}}.$$
	Furthermore,  $\exteriorSquareLFunction{s}{\localFieldRepresentation} = \representationLFunction{ms}{\centralCharacter{\localFieldRepresentation}}$, and $\epsilonFactorOfLocalField{s} = \frac{q^{ms}}{q^{\frac{m}{2}} \centralCharacter{\localFieldRepresentation}\left(\uniformizer\right)}$.
	
	Also in this case $\gammaFactorOfLocalField{s} = \modifiedGammaFactorOfResidueField{s - s_0}$, where $q^{m s_0} = \centralCharacter{\localFieldRepresentation} \left(\uniformizer\right)$ (see also \cref{thm:modified-functional-equation}).
\end{thm}

\begin{proof}
	By \cref{rem:shalika-vector-trivial-character}, since $\residueFieldRepresentation$ admits a Shalika vector, the central character $\centralCharacter{\residueFieldRepresentation}$ is trivial. Thus, the central character $\centralCharacter{\localFieldRepresentation}$ is unramified. Therefore,
	$$\representationLFunction{s}{\centralCharacter{\localFieldRepresentation}} = \frac{1}{1-\centralCharacter{\localFieldRepresentation} \left( \uniformizer \right) q^{-s}}.$$
	
	By \cref{thm:equivalent-conditions-for-a-shalika-vector}, there exists $W_0 \in \WhittakerModelOfResidueFieldRepresentation$, such that $\JSOfResidueFieldRepresentation{W_0}{1} = 1$. We substitute in \cref{thm:equivalent-conditions-for-a-shalika-vector} such $W_0$ and $\phi_0 = 1$, $\FourierTransformWithRespectToCharacter{\phi_0}{\residueFieldCharacter} = q^{\frac{m}{2}} \indicatorFunction{0}$ to get
	\begin{equation}\label{eqn:substitution}
		\JSOfLocalFieldRepresentation{s}{\lift{W_0}}{\lift{\phi_0}} = 1 + q^{-ms} \centralCharacter{\localFieldRepresentation} \left(\uniformizer\right) \representationLFunction{ms}{\centralCharacter{\localFieldRepresentation}} =  \representationLFunction{ms}{\centralCharacter{\localFieldRepresentation}},
	\end{equation}
	and $$\DualJSOfLocalFieldRepresentation{s}{\lift{W_0}}{\lift{\phi_0}} = q^{\frac{m}{2}} q^{-m \left( 1 - s \right)}  \centralCharacter{\localFieldRepresentation} \left( \uniformizer \right)^{-1} \representationLFunction{m \left(1 - s\right)}{\centralCharacter{\localFieldRepresentation}^{-1}}.$$
	The result regarding $\gammaFactorOfLocalField{s}$ follows as $\gammaFactorOfLocalField{s} = \frac{\DualJSOfLocalFieldRepresentation{s}{\lift{W_0}}{\lift{\phi_0}}}{\JSOfLocalFieldRepresentation{s}{\lift{W_0}}{\lift{\phi_0}}}$.
	
	Regarding $\exteriorSquareLFunction{s}{\localFieldRepresentation}$, denote $\exteriorSquareLFunction{s}{\localFieldRepresentation} = \frac{1}{p \left( q^{-s} \right)}$, where $p\left(Z\right) \in \cComplex \left[Z\right]$ is a polynomial with $p\left(0\right) = 1$. By \cref{thm:exterior-square-polynomial-is-a-divisor}, we have that $p\left(Z\right) \mid 1 - \centralCharacter{\localFieldRepresentation}\left(\uniformizer\right) Z^m$. From \cref{eqn:substitution}, we have that $$\JSOfLocalFieldRepresentation{s}{\lift{W_0}}{\lift{\phi_0}} = \representationLFunction{ms}{\centralCharacter{\localFieldRepresentation}} = \frac{1}{1 - \centralCharacter{\localFieldRepresentation}\left(\uniformizer\right) q^{-ms}} \in \frac{1}{p\left(q^{-s}\right)} \cComplex\left[q^{s}, q^{-s}\right],$$ so $1 - \centralCharacter{\localFieldRepresentation}\left(\uniformizer\right) Z^m \mid p\left(Z\right)$. Therefore we must have $p \left(Z\right) = 1 - \centralCharacter{\localFieldRepresentation}\left(\uniformizer\right) Z^m$, and the result $\exteriorSquareLFunction{s}{\localFieldRepresentation} = \representationLFunction{ms}{\centralCharacter{\localFieldRepresentation}}$ follows.	
	The result regarding $\epsilonFactorOfLocalField{s}$ now follows from the equation in \cref{thm:local-field-functional-equation}.
\end{proof}

\cref{thm:equality-of-gamma-factors-non-shalika-vector} and \cref{thm:equality-of-gamma-factors-with-shalika-vector} establish a connection between a cuspidal representation $\residueFieldRepresentation$ and its corresponding level zero representation $\localFieldRepresentation$ via the local exterior square factors of $\localFieldRepresentation$. Moreover, these theorems demonstrate a close connection between the existence of Shalika vectors and the existence of poles of the local exterior square $L$-function.

\begin{cor}\label{cor:shalika-vator-and-pole}
Let $\localFieldRepresentation$ be a level zero representation constructed from an irreducible cuspidal representation $\residueFieldRepresentation$. Then $\residueFieldRepresentation$ admits a Shalika vector if and only if $\exteriorSquareLFunction{s}{\localFieldRepresentation}$ has a pole.
\end{cor}

\begin{proof}
On one hand, if $\residueFieldRepresentation$ admits a Shalika vector, then by \cref{thm:equality-of-gamma-factors-with-shalika-vector}, $\exteriorSquareLFunction{s}{\localFieldRepresentation} = \representationLFunction{ms}{\centralCharacter{\localFieldRepresentation}}$ has a pole.

On the other hand, if $\residueFieldRepresentation$ does not admit a Shalika vector, then by \cref{thm:equality-of-gamma-factors-non-shalika-vector}, $\exteriorSquareLFunction{s}{\localFieldRepresentation} = 1$ does not have any poles.
\end{proof}

\begin{rem}
	Once we establish the connection between Shalika vectors of $\residueFieldRepresentation$ and Shalika functionals of $\localFieldRepresentation$ in the next section, we can see that \Cref{cor:shalika-vator-and-pole} is a special case of a more general fact relating existence of Shalika functionals to poles of exterior square $L$ functions. Corollary 4.4 of \cite{kewat2011local}, which is a consequence of \cite[Theorem 4.3]{kewat2011local}, states that for an irreducible square integrable representation $\pi$, a sufficient condition for $\pi$ to have a non-zero Shalika functional is that $L(s, \pi, \wedge^2)$ has a pole at $s = 0$. Actually, it is also a necessary condition, see \cite[Proposition 6.1]{matringe2014linear} or \cite[Propositon 3.12]{jo2018derivatives}.
\end{rem}

\subsubsection{Relation with Shalika functionals}\label{subsubsec:twisted-shalika-functionals}
Shalika vectors in $\underlyingVectorSpace{\residueFieldRepresentation}$ are closely related to poles of $\exteriorSquareLFunction{s}{\localFieldRepresentation}$, which in turn are closely akin to (twisted) Shalika functionals on $\WhittakerModelOfLocalFieldRepresentation$. After we relate Shalika vectors and Shalika periods, we will give another explanation for the computation of $\exteriorSquareLFunction{s}{\localFieldRepresentation}$ in the case where $\residueFieldRepresentation$ admits a Shalika vector. We begin with the introduction of twisted Shalika functionals $\Lambda_s$ on $\WhittakerModelOfLocalFieldRepresentation$.

\begin{defn}\label{defn:twisted-shalika-functional}Let $\representationDeclaration{\localFieldRepresentation}$ be a representation with an unramified central character.
For any $s \in \cComplex$ satisfying $q^{ms} = \centralCharacter{\localFieldRepresentation}\left(\uniformizer\right)$, the twisted Shalika period $\Lambda_s$ on $\WhittakerModelOfLocalFieldRepresentation$ is defined to be the following linear functional $\Lambda_s : \WhittakerModelOfLocalFieldRepresentation \rightarrow \cComplex$
\begin{equation*}
\Lambda_s \left( W \right) = \int_{\lquot{\multiplicativegroup{\localField}\UnipotentSubgroup}{G}} \int_{\lquot{\UpperTriangularAdditive}{M}} W\left(\evenPermutationMatrix \ShalikaUnipotentElement{X} \ShalikaDiagonalElement{g} \right) \fieldCharacter \left(-\trace X \right) dX \cdot \abs{\det g}^s \multiplicativeMeasure{g}.
\end{equation*}
\end{defn}

For any $s \in \cComplex$, we set $\nu^{s}$ to be the one dimensional representation of $\GL{2m}{\localField}$ given by $\nu^{s}(g) = \abs{\det g}^s$ for any $g \in \GL{2m}{\localField}$. Applying \cite[Lemma 4.2]{kewat2011local} to the representation $\localFieldRepresentation \otimes \nu^{\frac{s}{2}}$, we know that $\Lambda_s$ converges absolutely. $\Lambda_s$ is a twisted Shalika functional in the sense that for any $h \in \ShalikaSubgroupEven$ and any $W \in \WhittakerModelOfLocalFieldRepresentation$, we have
\begin{equation}\label{eqn:translation-under-shalika-subgroup}
\Lambda_s \left( \localFieldRepresentation \left( h \right)W \right) = \abs{\det h}^{-s} \ShalikaCharacter \left( h \right) \Lambda_s \left( W \right).
\end{equation}
By \cref{eqn:translation-under-shalika-subgroup}, we have
$$\Lambda_s \in \Hom_{\ShalikaSubgroupEven(\localField)} \left( \localFieldRepresentation \otimes \nu^{\frac{s}{2}}, \ShalikaCharacter \right).$$

By \cite{jo2018derivatives}, we have
\begin{thm}[{\cite[Theorem 3.6]{jo2018derivatives}}]\label{thm:exterior-square-L-function-in-terms-of-shalika-functionals}Let $\representationDeclaration{\localFieldRepresentation}$ be an irreducible supercuspidal representation of $\GL{2m}{\localField}$, with an unramified central character. Then
	$$\exteriorSquareLFunction{s}{\localFieldRepresentation} = \prod_{\alpha}(1-\alpha q^{-s})^{-1},$$
where the product runs over all $\alpha = q^{s_0}$, such that $\alpha^{m} = \centralCharacter{\localFieldRepresentation}(\uniformizer)$ and $\Lambda_{s_0} \ne 0$. Equivalently, the product runs over all $\alpha = q^{s_0}$ such that $\Hom_{\ShalikaSubgroupEven(\localField)} \left( \localFieldRepresentation \otimes \nu^{\frac{s_0}{2}}, \ShalikaCharacter \right) \ne 0$.
\end{thm}

As before, let $\representationDeclaration{\residueFieldRepresentation}$ be an irreducible representation of $\GL{2m}{\residueField}$ and let $\representationDeclaration{\localFieldRepresentation}$ be a level zero representation constructed from $\residueFieldRepresentation$. We recall that the central character $\centralCharacter{\residueFieldRepresentation}$ is trivial if and only if the central character $\centralCharacter{\localFieldRepresentation}$ is unramified. Similarly to the proof of \cref{thm:Jacquet-Shalika-integral-of-a-level-zero-supercuspidal-representation-even-case}, one can show

\begin{prop}\label{prop:evaluate-shalika-functional-on-lifted-whittaker-function}
	Suppose $\centralCharacter{\residueFieldRepresentation}$ is trivial and equivalently $\centralCharacter{\localFieldRepresentation}$ is unramified. For any $W_0 \in \WhittakerModelOfResidueFieldRepresentation$ and any $s \in \cComplex$, such that $q^{ms} = \centralCharacter{\localFieldRepresentation} \left(\uniformizer\right)$, we have
	\begin{equation*}
		\Lambda_s \left( \lift{W_0} \right) = \JSOfResidueFieldRepresentation{W_0}{1}.
	\end{equation*}
\end{prop}
We remark that the right hand side of the equation in the above proposition is independent of $s \in \cComplex$.

We can now use \cref{thm:exterior-square-L-function-in-terms-of-shalika-functionals} and \cref{prop:evaluate-shalika-functional-on-lifted-whittaker-function} to give another explanation why in the case that $\residueFieldRepresentation$ admits a Shalika vector, $\exteriorSquareLFunction{s}{\localFieldRepresentation} = \representationLFunction{ms}{\centralCharacter{\localFieldRepresentation}}$:
Since $\residueFieldRepresentation$ admits a Shalika vector, we have by \cref{rem:shalika-vector-trivial-character} that the central character $\centralCharacter{\residueFieldRepresentation}$ is trivial. We also have by \cref{thm:equivalent-conditions-for-a-shalika-vector} and \cref{prop:evaluate-shalika-functional-on-lifted-whittaker-function} that for any $s \in \cComplex$, with $q^{ms} = \centralCharacter{\localFieldRepresentation}\left(\uniformizer\right)$, $\Lambda_s \ne 0$. Therefore, the condition $\alpha = q^{s_0}$ with $\Lambda_{s_0} \ne 0$ is always valid. Since $\alpha$ in the product in \cref{thm:exterior-square-L-function-in-terms-of-shalika-functionals} runs over the $m$-th roots of $\centralCharacter{\localFieldRepresentation}(\uniformizer)$, we get
$$
L(s, \localFieldRepresentation, \wedge^2) = \prod_{\alpha}(1-\alpha q^{-s})^{-1} = (1-\centralCharacter{\localFieldRepresentation}(\uniformizer)q^{-ms})^{-1} = L(ms, \centralCharacter{\localFieldRepresentation}).
$$

We conclude this section by giving a characterization of the existence of Shalika vectors in terms of the existence of (twisted) Shalika functionals.

\begin{prop}
$\Hom_{\ShalikaSubgroupEven(\residueField)} \left( \residueFieldRepresentation, \residueFieldShalikaCharacter \right) \ne 0$ if and only if $\Hom_{\ShalikaSubgroupEven(\localField)} \left( \localFieldRepresentation \otimes \nu^{s/2}, \ShalikaCharacter \right) \ne 0$ for some $s \in \cComplex$. Here $\residueFieldShalikaCharacter : \ShalikaSubgroup{2m}\left(\residueField\right) \rightarrow \multiplicativegroup{\cComplex}$ is the character on Shalika subgroup, defined by the character $\residueFieldCharacter$ (see \cref{defn:shalika-subgroup-even}).
\end{prop}

\begin{proof}
With a choice of inner product on $\underlyingVectorSpace{\residueFieldRepresentation}$, with respect to which $\residueFieldRepresentation$ is unitary, we can show that the existence of a Shalika vector is equivalent to $\Hom_{\ShalikaSubgroupEven(\residueField)} \left( \residueFieldRepresentation, \residueFieldShalikaCharacter \right) \ne 0$. By \cref{thm:exterior-square-L-function-in-terms-of-shalika-functionals}, $\exteriorSquareLFunction{s}{\localFieldRepresentation}$ has a pole at $s \in \cComplex$ if and only if $\Hom_{\ShalikaSubgroupEven(\localField)} \left( \localFieldRepresentation \otimes \nu^{s/2}, \ShalikaCharacter \right) \ne 0$. Therefore, the proposition now follows from \cref{cor:shalika-vator-and-pole}.
\end{proof}

\appendix

\section{Multiplicity one theorems}\label{sec:multiplicity-one}

In the appendix, we follow \cite{matringe2012cuspidal} in order to prove the following multiplicity one theorems, which are keys to the proofs of the functional equation for the finite field case (\cref{thm:functional-equation-finite-field}).

Let $\finiteField$ be a finite field.

\begin{thm}\label{thm:multiplicity-one-even-case}
	Let $\representationDeclaration{\finiteFieldRepresentation}$ be an irreducible cuspidal representation of $\GL{2m}{\finiteField}$. Then $$\dim_{\cComplex}\Hom_{\LeviSubgroup{m,m} \cap \Mirabolic{2m}}\left(\finiteFieldRepresentation, 1\right) \le 1,$$
	where $\LeviSubgroup{m,m}$ is the Levi subgroup corresponding to the partition $\left(m,m\right)$, and $\Mirabolic{2m}$ is the mirabolic subgroup.
\end{thm}

\begin{thm}\label{thm:multiplicity-one-odd-case}
	Let $\representationDeclaration{\finiteFieldRepresentation}$ be an irreducible cuspidal representation of $\GL{2m+1}{\finiteField}$. Then
	$$\dim_{\cComplex}\Hom_{\NonstandardLeviSubgroup{2m+1} \cap \Mirabolic{2m+1}}\left(\finiteFieldRepresentation, 1\right) \le 1,$$
	where $\Mirabolic{2m+1}$ is the mirabolic subgroup and $\NonstandardLeviSubgroup{2m+1}$ is the following maximal (non-standard) Levi subgroup of $\GL{2m+1}{\finiteField}$ corresponding to the partition $\left(m+1,m\right)$:
	$$\NonstandardLeviSubgroup{2m+1}= \left\{ \begin{pmatrix}
			g_1 &     & u       \\
			    & g_2 &         \\
			v   &     & \lambda
		\end{pmatrix} \mid g_1,g_2 \in \GL{m}{\finiteField}, \,u \in \Mat{m}{1}{\finiteField}, \,v \in \Mat{1}{m}{\finiteField}, \lambda \in \finiteField \right\} \cap \GL{2m+1}{\finiteField}.$$
\end{thm}

The proofs of these theorems require some preparations.

For a finite group $G$ and a vector space $V$ over $\cComplex$, denote $$\Schwartz\left(G, V\right) = \left\{f : G \rightarrow V \right\}.$$
Denote by $\rightTranslation$ and $\leftTranslation$ the right and left actions of $G$ on $\Schwartz\left(G, V\right)$ respectively: $ \left(\rightTranslation\left(g_0\right) f\right)\left(g\right) = f\left(g g_0\right)$, $ \left(\leftTranslation\left(g_0\right) f\right)\left(g\right) = f\left({g_0}^{-1} g \right)$.

Let $\fieldCharacter : \finiteField \rightarrow \cComplex$ be a non-trivial additive character of $\finiteField$.

For every positive integer $n$, we denote $G_n = \GL{n}{\finiteField}$. In the following, we use the convention that for $k < n$, $G_k$ is embedded in $G_n$ by mapping $g \mapsto \smallDiagTwo{g}{\IdentityMatrix{n-k}}$. Suppose that $p + q = n$, where $p \ge q \ge 0$. Let $\LeviPermutation{p}{q}$ be the following permutation:
$$\begin{pmatrix}
		1 & 2 & \cdots & p-q & \mid & p-q+1 & p-q+2 & \cdots & p     & | & p+1   & p+2   & \cdots & p+q \\
		1 & 2 & \cdots & p-q & \mid & p-q+1 & p-q+3 & \cdots & p+q-1 & | & p-q+2 & p-q+4 & \cdots & p+q
	\end{pmatrix}.$$
Let $\LeviPermutationMatrix{p}{q}$ be the column permutation matrix corresponding to $\LeviPermutation{p}{q}$.

Let $\TwistedLeviSubgroup{p}{q} = \LeviPermutationMatrix{p}{q} \LeviSubgroup{p,q} \LeviPermutationMatrix{p}{q}^{-1}$, where $\LeviSubgroup{p,q}$ is the Levi subgroup corresponding to the partition $\left(p,q\right)$. If $q \ge 1$, denote $$\TwistedLeviSubgroup{p}{q-1} = \LeviPermutationMatrix{p}{q} \left\{ \diagTwo{m}{1} \mid m \in \LeviSubgroup{p, q-1} \right\}\LeviPermutationMatrix{p}{q}^{-1},$$ where $\LeviSubgroup{p, q-1}$ is the Levi subgroup of $\GL{n-1}{\finiteField}$ corresponding to the partition $\left(p, q-1\right)$. Note that since $\LeviPermutation{p}{q}\left(n\right) = n$, we have that $\TwistedLeviSubgroup{p}{q-1}$ is a subgroup of $G_{n-1}$. If $q \ge 1$, also denote $$\TwistedLeviSubgroup{p-1}{q-1} = \left\{ \diagTwo{h}{\IdentityMatrix{2}} \mid h \in \TwistedLeviSubgroup[n-2]{p-1}{q-1} \right\}.$$

Let $$\MirabolicUnipotentRadical{n} = \left\{ \begin{pmatrix}
		\IdentityMatrix{n-1} & x \\
		                     & 1
	\end{pmatrix} \mid x \in \Mat{n-1}{1}{\finiteField} \right\},$$ be the unipotent radical of $G_n$, corresponding to the partition $\left(n-1,1\right)$. Let $\Mirabolic{n} \le G_n$ be the mirabolic subgroup. We have that $\Mirabolic{n} = \MirabolicUnipotentRadical{n} \rtimes G_{n-1}$. $\fieldCharacter$ defines a character on $\MirabolicUnipotentRadical{n}$ by $\fieldCharacter\left(u\right) = \fieldCharacter\left(u_{n-1,n}\right)$.

Recall the definition of the following Bernstein-Zelevinsky derivative: for a representation $\finiteFieldRepresentation$ of $\Mirabolic{n-1}$, we define $\MirabolicPlusFunctor \left( \finiteFieldRepresentation \right) = \Ind{\Mirabolic{n-1} \MirabolicUnipotentRadical{n}}{\Mirabolic{n}}{\finiteFieldRepresentation \otimes \fieldCharacter}$, where $\left(\finiteFieldRepresentation \otimes \fieldCharacter\right) \left(p u\right) = \fieldCharacter \left(u\right) \cdot \finiteFieldRepresentation \left(p\right)$, for $p \in \Mirabolic{n-1}$ and $u \in \MirabolicUnipotentRadical{n}$. We have the following relation between this functor and irreducible cuspidal representations of $G_n$:

\begin{thm}[{\cite[Theorem 2.3]{gel1970representations}}] \label{thm:cuspidal-restriction-to-mirabolic}
	Let $\finiteFieldRepresentation$ be an irreducible cuspidal representation of $\GL{n}{\finiteField}$. Then the restriction of $\finiteFieldRepresentation$ to the mirabolic subgroup $\Mirabolic{n}$ is isomorphic to the representation $\left(\MirabolicPlusFunctor\right)^{n-1}\left(1\right)$.
\end{thm}

We start with the following lemmas from \cite{matringe2012cuspidal}. These are purely algebraic statements, which are proved exactly as in \cite{matringe2012cuspidal}.

\begin{lem}[{\cite[Lemma 3.1]{matringe2012cuspidal}}]\label{lem:stabilizer-of-character-p-q}
	Let $p \ge q \ge 1$ with $p+q = n$, and let $$\CharacterStabilizerSubgroup{p}{q} = \left\{ g \in G_{n-1} \mid \fieldCharacter \left(gug^{-1}\right) = 1, \forall u \in \MirabolicUnipotentRadical{n} \cap \TwistedLeviSubgroup{p}{q} \right\}.$$
	Then $\CharacterStabilizerSubgroup{p}{q} = \Mirabolic{n-1} \cdot \TwistedLeviSubgroup{p}{q-1}$.
\end{lem}

\begin{lem}[{\cite[Lemma 3.2]{matringe2012cuspidal}}]\label{lem:stabilizer-of-character-p-q-minus-1}
	Let $p \ge q \ge 1$ with $p+q = n$, and let $$\CharacterStabilizerSubgroup{p}{q-1} = \left\{ g \in G_{n-2} \mid \fieldCharacter \left(gug^{-1}\right) = 1, \forall u \in \MirabolicUnipotentRadical{n-1} \cap \TwistedLeviSubgroup{p}{q-1} \right\}.$$
	Then $\CharacterStabilizerSubgroup{p}{q-1} = \Mirabolic{n-2} \cdot \TwistedLeviSubgroup{p-1}{q-1}$.
\end{lem}

The proofs of both multiplicity one theorems rely on the following propositions:

\begin{prop}\label{prop:embedding-n-minus-1}
	Suppose $p \ge q \ge 1$ with $p+q = n$. Let $\representationDeclaration{\MirabolicRepresentation}$ be a representation of $\Mirabolic{n-1}$. Then there exists an embedding
	$$\Hom_{\Mirabolic{n} \cap \TwistedLeviSubgroup{p}{q}}\left(\MirabolicPlusFunctor\left(\MirabolicRepresentation\right), 1\right) \hookrightarrow \Hom_{\Mirabolic{n-1} \cap \TwistedLeviSubgroup{p}{q-1}} \left(\MirabolicRepresentation, 1\right). $$
\end{prop}

\begin{prop}\label{prop:embedding-n-minus-2}
	Suppose $p \ge q \ge 1$ with $p+q = n$. Let $\representationDeclaration{\MirabolicRepresentation}$ be a representation of $\Mirabolic{n-2}$. Then there exists an embedding
	$$\Hom_{\Mirabolic{n-1} \cap \TwistedLeviSubgroup{p}{q-1}}\left(\MirabolicPlusFunctor\left(\MirabolicRepresentation\right), 1\right) \hookrightarrow \Hom_{\Mirabolic{n-2} \cap \TwistedLeviSubgroup{p-1}{q-1}} \left(\MirabolicRepresentation, 1\right). $$
\end{prop}

We prove only \cref{prop:embedding-n-minus-1}. The proof of \cref{prop:embedding-n-minus-2} is similar.

\begin{proof}[Proof of \cref{prop:embedding-n-minus-1}]
	Denote $W = \MirabolicPlusFunctor \left( \MirabolicRepresentation \right) = \Ind{\Mirabolic{n-1} \MirabolicUnipotentRadical{n}}{\Mirabolic{n}}{\MirabolicTensorRepresentation}$, where $\MirabolicTensorRepresentation = \MirabolicRepresentation \otimes \fieldCharacter$.

	Let $A : \Schwartz\left(\Mirabolic{n},\MirabolicRepresentationVectorSpace\right) \rightarrow W$ be the projection operator $$\left(Af\right)\left(p\right) = \frac{1}{\sizeof{\Mirabolic{n-1} \MirabolicUnipotentRadical{n}}} \sum_{y \in \Mirabolic{n-1} \MirabolicUnipotentRadical{n}}{\MirabolicTensorRepresentation\left(y\right)^{-1} f\left(y p\right)}.$$

	Let $L \in \Hom_{\Mirabolic{n} \cap \TwistedLeviSubgroup{p}{q}}\left(\MirabolicPlusFunctor\left(\MirabolicRepresentation\right), 1\right)$. We can define using $A$ and $L$ a distribution $T = L \circ A : \Schwartz\left( \Mirabolic{n}, \MirabolicRepresentationVectorSpace \right) \rightarrow \cComplex$. One easily checks that this distribution satisfies
	\begin{align}
		\label{eq:mirabolic-right-translation-invariance} & \standardForm{T}{\rightTranslation \left(h\right)f} = \standardForm{T}{f}                                                     & \forall h \in \Mirabolic{n} \cap \TwistedLeviSubgroup{p}{q},   \\
		\label{eq:mirabolic-left-translation-invariance}  & \standardForm{T}{\leftTranslation\left(y_0\right) f} = \standardForm{T}{\MirabolicTensorRepresentation \left(y_0\right)^{-1} f} & \forall y_0 \in \Mirabolic{n-1} \MirabolicUnipotentRadical{n}.
	\end{align}

	Since $A$ is onto, we have that the map $L \mapsto L \circ A$ is an injection from $\Hom_{\Mirabolic{n} \cap \TwistedLeviSubgroup{p}{q}} \left(\MirabolicPlusFunctor \left(\MirabolicRepresentation\right), 1\right)$ to the space of all distributions satisfying \cref{eq:mirabolic-right-translation-invariance} and \cref{eq:mirabolic-left-translation-invariance}.

	Let $\MirabolicCharacter : \Mirabolic{n} \rightarrow \cComplex$ be the function defined by $\MirabolicCharacter \left(ug\right) = \fieldCharacter\left(u\right)$, where $u\in \MirabolicUnipotentRadical{n}$ and $g \in G_{n-1}$. It is well defined, since $\MirabolicUnipotentRadical{n} \cap G_{n-1} = \left\{\IdentityMatrix{n}\right\}$. One has $\MirabolicCharacter \left(up\right) = \fieldCharacter \left(u\right) \MirabolicCharacter \left(p\right)$ for $u \in \MirabolicUnipotentRadical{n} $ and $p \in \Mirabolic{n}$.

	Let $T$ be a distribution on $\Schwartz \left(\Mirabolic{n},\MirabolicRepresentationVectorSpace\right)$, satisfying \cref{eq:mirabolic-right-translation-invariance} and \cref{eq:mirabolic-left-translation-invariance}. We denote by $\MirabolicCharacter \cdot T$ the distribution defined by $\standardForm{\MirabolicCharacter \cdot T}{f} = \standardForm{T}{\MirabolicCharacter \cdot f}$. One easily checks using \cref{eq:mirabolic-left-translation-invariance} that \begin{equation}\label{eq:mirabolic-character-times-distribution-invariance}
		\standardForm{\leftTranslation \left(u\right)\left(\MirabolicCharacter \cdot T\right)}{f} = \standardForm{\MirabolicCharacter \cdot T}{f},
	\end{equation} for every $u \in \MirabolicUnipotentRadical{n}$. Define for $f : \Mirabolic{n} \rightarrow \MirabolicRepresentationVectorSpace$, $f' : G_{n - 1} \rightarrow \MirabolicRepresentationVectorSpace$ by $$f' \left(g\right) = \sum_{u \in \MirabolicUnipotentRadical{n}}{f \left(ug\right)},$$ and define a distribution $S$ on $\Schwartz\left(G_{n-1}, \MirabolicRepresentationVectorSpace\right)$ by defining it on the basis by $\standardForm{S}{\indicatorFunction{g_0}} = \standardForm{\MirabolicCharacter \cdot T}{\indicatorFunction{g_0}}$, for every $g_0 \in G_{n-1}$. Then one easily checks using \cref{eq:mirabolic-character-times-distribution-invariance} that $\standardForm{S}{f'} = \standardForm{\MirabolicCharacter \cdot T}{f}$, for every $f \in \Schwartz\left(G_{n-1}, \MirabolicRepresentationVectorSpace\right)$. It follows that $\standardForm{\MirabolicCharacter \cdot T}{\rightTranslation \left(u_0\right) f} =\standardForm{\MirabolicCharacter \cdot T}{f}$, for every $f \in \Schwartz \left(\Mirabolic{n}, \MirabolicRepresentationVectorSpace\right)$ and $u_0 \in \MirabolicUnipotentRadical{n}$.

	Suppose that $g_0 \in G_{n-1}$ is in the support of $S$, i.e., $\standardForm{S}{\indicatorFunction{g_0}} \ne 0$. Then for every $u_0 \in \MirabolicUnipotentRadical{n} \cap \TwistedLeviSubgroup{p}{q}$, we have $$\MirabolicCharacter \left(g_0\right) \standardForm{T}{\indicatorFunction{g_0}} = \standardForm{\MirabolicCharacter \cdot T}{\indicatorFunction{g_0}} = \standardForm{\MirabolicCharacter \cdot T}{\rightTranslation\left(u_0\right) \indicatorFunction{g_0}} = \MirabolicCharacter\left(g_0 u_0^{-1}\right) \standardForm{T}{\indicatorFunction{g_0}}.$$
	Since $\standardForm{S}{\indicatorFunction{g_0}} \ne 0$, we get that $\standardForm{T}{\indicatorFunction{g_0}} \ne 0$, and therefore $\MirabolicCharacter\left(g_0 u_0^{-1}\right) = \MirabolicCharacter \left(g_0\right)$, which implies that $\fieldCharacter \left(g_0 u_0 g_0^{-1}\right) = 1$. Thus we have that $\Supp S \subseteq \CharacterStabilizerSubgroup{p}{q}$, and by \cref{lem:stabilizer-of-character-p-q}, $\Supp S \subseteq \Mirabolic{n-1} \cdot \TwistedLeviSubgroup{p}{q-1}$, and hence $\Supp T \subseteq \MirabolicUnipotentRadical{n} \Mirabolic{n-1} \TwistedLeviSubgroup{p}{q-1} =  \Mirabolic{n-1} \MirabolicUnipotentRadical{n} \TwistedLeviSubgroup{p}{q-1}$. Hence the restriction map $T \rightarrow T \restriction_{\Schwartz\left(\Mirabolic{n-1} \MirabolicUnipotentRadical{n} \TwistedLeviSubgroup{p}{q-1}, \MirabolicRepresentationVectorSpace\right)}$, from the space of distributions on $\Schwartz \left(\Mirabolic{n},\MirabolicRepresentationVectorSpace\right)$  satisfying \cref{eq:mirabolic-right-translation-invariance} and \cref{eq:mirabolic-left-translation-invariance} to the space on distributions of $\Schwartz\left(\Mirabolic{n-1} \MirabolicUnipotentRadical{n} \TwistedLeviSubgroup{p}{q-1}, \MirabolicRepresentationVectorSpace\right)$ satisfying \cref{eq:mirabolic-right-translation-invariance}, \cref{eq:mirabolic-left-translation-invariance} is injective.

	Consider the projection $B : \Schwartz\left(\Mirabolic{n-1} \MirabolicUnipotentRadical{n} \times \TwistedLeviSubgroup{p}{q-1}, \MirabolicRepresentationVectorSpace \right) \rightarrow \Schwartz\left(\Mirabolic{n-1} \MirabolicUnipotentRadical{n}  \TwistedLeviSubgroup{p}{q-1}, \MirabolicRepresentationVectorSpace \right)$ given by $$\left(Bf\right)\left(y^{-1} h\right) = \frac{1}{\sizeof{\Mirabolic{n-1} \cap \TwistedLeviSubgroup{p}{q-1}}}\sum_{a \in \Mirabolic{n-1} \cap \TwistedLeviSubgroup{p}{q-1} }{f \left(ay, ah\right)},$$
	for $y \in \Mirabolic{n-1} \MirabolicUnipotentRadical{n}$, $h \in \TwistedLeviSubgroup{p}{q-1}$. This is well defined, and $B$ is a projection in the sense that if $f \left(y,h\right) = g\left(y^{-1} h\right)$ for $g \in \Schwartz \left(\Mirabolic{n-1} \MirabolicUnipotentRadical{n} \TwistedLeviSubgroup{p}{q-1}, \MirabolicRepresentationVectorSpace\right)$, then $Bf = g$. In particular $B$ is onto.

	Consider the linear isomorphism from $\Schwartz \left(\Mirabolic{n-1} \MirabolicUnipotentRadical{n} \times \TwistedLeviSubgroup{p}{q-1}, \MirabolicRepresentationVectorSpace \right) $ to itself, $\phi \mapsto \tilde{\phi}$, given by $\tilde{\phi}\left(y,h\right) =  \MirabolicTensorRepresentation \left(y\right)^{-1} \phi \left(y, h\right) $. Let $y_0 \in \Mirabolic{n-1} \MirabolicUnipotentRadical{n}$, $h_0 \in \TwistedLeviSubgroup{p}{q-1}$, $\phi \in \Schwartz \left( \Mirabolic{n-1}\MirabolicUnipotentRadical{n} \times \TwistedLeviSubgroup{p}{q-1}, \MirabolicRepresentationVectorSpace \right)$ and denote $\phi_1 = \rightTranslation\left(y_0, h_0\right) \phi$. One checks that $\widetilde{\phi_1} = \MirabolicTensorRepresentation \left(y_0\right) \left(\rightTranslation \left(y_0, h_0\right) \tilde{\phi} \right)$ and that $B \left( \widetilde{\phi_1} \right) = \rightTranslation \left(h_0\right) \leftTranslation \left(y_0\right) \MirabolicTensorRepresentation \left(y_0\right) B \left(\tilde{\phi}\right)$, which implies for a distribution $T$ on $\Schwartz\left(\Mirabolic{n-1} \MirabolicUnipotentRadical{n} \TwistedLeviSubgroup{p}{q-1}, \MirabolicRepresentationVectorSpace\right)$ satisfying \cref{eq:mirabolic-right-translation-invariance}, \cref{eq:mirabolic-left-translation-invariance}, we have $\standardForm{T}{B\left( \tilde{\phi} \right)} = \standardForm{T}{B\left( \widetilde{\phi_1} \right)}$.

	For $T$ as above, we define a distribution $D_T$ on the space $\Schwartz\left(\Mirabolic{n-1} \MirabolicUnipotentRadical{n} \times \TwistedLeviSubgroup{p}{q-1}, \MirabolicRepresentationVectorSpace \right)$ by $\standardForm{D_T}{\phi} = \standardForm{T}{B \left(\tilde{\phi}\right)}$. The map $T \mapsto D_T$ is injective, as $B$ is surjective and $\phi \mapsto \tilde{\phi} $ is an isomorphism. The above discussion shows that $\rightTranslation \left(y_0, h_0\right) D_T = D_T$, for every $y_0 \in \Mirabolic{n-1} \MirabolicUnipotentRadical{n}$, $h_0 \in \TwistedLeviSubgroup{p}{q-1}$. This means that $D_T$ is determined by the functional $\xi_T : \MirabolicRepresentationVectorSpace \rightarrow \cComplex$, defined by $\standardForm{\xi_T}{v} = \standardForm{D_T}{v \cdot  \indicatorFunction{\left( \IdentityMatrix{n}, \IdentityMatrix{n} \right)}}$, and is given by the formula \begin{equation}\label{eq:distribution-in-terms-of-functional}
		\standardForm{D_T}{\phi} = \sum_{\substack{
				y \in \Mirabolic{n - 1} \MirabolicUnipotentRadical{n}\\
				h \in \TwistedLeviSubgroup{p}{q-1}
			}}{\standardForm{\xi_T}{\phi \left(y, h\right)}}.
	\end{equation}

	We now show that $\xi_T \in \Hom_{\Mirabolic{n-1} \cap \TwistedLeviSubgroup{p}{q-1}} \left(\MirabolicRepresentation ,1\right)$. Let $\phi \in \Schwartz \left(\Mirabolic{n-1} \MirabolicUnipotentRadical{n} \times \TwistedLeviSubgroup{p}{q-1}\right)$, let $b \in \Mirabolic{n-1} \cap \TwistedLeviSubgroup{p}{q-1}$, and let $\phi_1 = \leftTranslation \left(b,b\right) \phi$. Then an easy computation shows that $\leftTranslation\left(b,b\right) \reallywidetilde{\MirabolicTensorRepresentation \left(b\right)^{-1} \phi} = \widetilde{\phi_1}$. Since $B \left(\leftTranslation \left(b,b\right)f \right) = B \left(f\right)$, for every $f \in \Schwartz\left(\Mirabolic{n-1} \MirabolicUnipotentRadical{n} \times \TwistedLeviSubgroup{p}{q-1}, \MirabolicRepresentationVectorSpace \right)$, we get that $\standardForm{T}{\leftTranslation \left(b,b\right) \phi} = \standardForm{T}{\MirabolicTensorRepresentation \left(b\right)^{-1} \phi}$. It follows from \cref{eq:distribution-in-terms-of-functional} that $\standardForm{T}{\leftTranslation \left(b,b\right) \phi} = \standardForm{T}{\phi}$. Therefore $\standardForm{\xi_T}{v} = \standardForm{\xi_T}{\MirabolicTensorRepresentation \left(b\right) v}$, for every $v \in \MirabolicRepresentationVectorSpace$ and every $b \in \Mirabolic{n-1} \cap \TwistedLeviSubgroup{p}{q-1}$, as required.

\end{proof}

We are now able to prove \cref{thm:multiplicity-one-even-case} and \cref{thm:multiplicity-one-odd-case}.

\begin{proof}[Proof of \cref{thm:multiplicity-one-even-case}]
	Since $\TwistedLeviSubgroup[2m]{m}{m} = \LeviPermutationMatrix{m}{m} \LeviSubgroup{m,m} \LeviPermutationMatrix{m}{m}^{-1}$ and $\Mirabolic{2m} = \LeviPermutationMatrix{m}{m} \Mirabolic{2m} \LeviPermutationMatrix{m}{m}^{-1}$, we get that $$\Hom_{\Mirabolic{2m} \cap \LeviSubgroup{m,m}} \left(\finiteFieldRepresentation, 1\right) \isomorphic \Hom_{\Mirabolic{2m} \cap \TwistedLeviSubgroup[2m]{m}{m}} \left(\finiteFieldRepresentation, 1\right)$$
	by mapping $L \in \Hom_{\Mirabolic{2m} \cap \TwistedLeviSubgroup[2m]{m}{m}} \left(\finiteFieldRepresentation, 1\right)$ to $L \finiteFieldRepresentation \left( \LeviPermutationMatrix{m}{m} \right) \in \Hom_{\Mirabolic{2m} \cap \LeviSubgroup{m,m}} \left(\finiteFieldRepresentation, 1\right)$. Therefore it suffices to prove that $$ \dim_{\cComplex} \Hom_{\Mirabolic{2m} \cap \TwistedLeviSubgroup[2m]{m}{m}} \left(\finiteFieldRepresentation, 1\right) \le 1.$$

	Since $\finiteFieldRepresentation$ is an irreducible cuspidal representation, we have from \cref{thm:cuspidal-restriction-to-mirabolic} that $\finiteFieldRepresentation \restriction_{\Mirabolic{2m}} \isomorphic \left(\MirabolicPlusFunctor\right)^{2m - 1}\left(1\right)$.	Using \cref{prop:embedding-n-minus-1} and \cref{prop:embedding-n-minus-2} repeatedly, and using the fact that $ \TwistedLeviSubgroup[2m]{m-1}{m-1} \isomorphic \TwistedLeviSubgroup[2m - 2]{m-1}{m-1} $, one gets an embedding
	$$\Hom_{\Mirabolic{2m} \cap \TwistedLeviSubgroup[2m]{m}{m}} \left(\left(\MirabolicPlusFunctor\right)^{2m - 1} \left(1\right), 1\right) \hookrightarrow \Hom_{\Mirabolic{2} \cap \TwistedLeviSubgroup[2]{1}{0}} \left(1, 1\right).$$
	The last space is one dimensional, and therefore we get the required result.
\end{proof}

\begin{proof}[Proof of \cref{thm:multiplicity-one-odd-case}]
	Let $\tau$ be the permutation $$\begin{pmatrix}
			1 & 2 & 3 & \dots & m & \mid & m+1  & \mid & m+2 & m+3 & \dots & 2m+1 \\
			1 & 2 & 3 & \dots & m & \mid & 2m+1 & \mid & m+1 & m+2 & \dots & 2m
		\end{pmatrix},$$
	and let $\columnPermutationMatrix{\tau}$ be the column permutation matrix corresponding to $\tau$. Then $\NonstandardLeviSubgroup{2m+1} = \columnPermutationMatrix{\tau} \LeviSubgroup{m+1,m} \columnPermutationMatrix{\tau}^{-1}$, where $\LeviSubgroup{m+1,m}$ is the standard Levi subgroup corresponding to the partition $\left(m+1,m\right)$. A simple calculation shows that $$\NonstandardLeviSubgroup{2m+1} \cap \Mirabolic{2m + 1} = \columnPermutationMatrix{\tau} \left\{ \diagTwo{p}{g} \mid p \in \Mirabolic{m+1}, g \in \GL{m}{\finiteField} \right\} \columnPermutationMatrix{\tau}^{-1}.$$
	Similarly, an easy calculation shows that $$\TwistedLeviSubgroup[2m + 2]{m+1}{m} \cap \Mirabolic{2m + 1} = \LeviPermutationMatrix{m+1}{m+1} \left\{ \begin{pmatrix}
			p &   &   \\
			  & g &   \\
			  &   & 1
		\end{pmatrix} \mid p \in \Mirabolic{m+1}, g \in \GL{m}{\finiteField} \right\} \LeviPermutationMatrix{m+1}{m+1}^{-1}.$$
	Therefore $\Mirabolic{2m + 1}  \cap \TwistedLeviSubgroup[2m + 2]{m+1}{m}$ and $\Mirabolic{2m + 1} \cap \NonstandardLeviSubgroup{2m + 1}$ are conjugate as subgroups of $\GL{2m+1}{\finiteField}$ (actually even as subgroups of $\Mirabolic{2m+1}$), which implies that $$\Hom_{\Mirabolic{2m+1} \cap \NonstandardLeviSubgroup{2m+1}} \left(\finiteFieldRepresentation, 1 \right) \isomorphic \Hom_{\Mirabolic{2m + 1} \cap \TwistedLeviSubgroup[2m + 2]{m+1}{m}} \left(\finiteFieldRepresentation, 1\right).$$
	Thus it suffices to prove that $$\dim_{\cComplex} \Hom_{\Mirabolic{2m + 1} \cap \TwistedLeviSubgroup[2m + 2]{m+1}{m}}\left(\finiteFieldRepresentation, 1\right) \le 1.$$

	As in the previous proof, by \cref{thm:cuspidal-restriction-to-mirabolic}, and by using \cref{prop:embedding-n-minus-1} and \cref{prop:embedding-n-minus-2} repeatedly, we get an embedding
	$$\Hom_{\Mirabolic{2m + 1} \cap \TwistedLeviSubgroup[2m + 2]{m+1}{m}}\left(\finiteFieldRepresentation, 1\right) \hookrightarrow \Hom_{\Mirabolic{2} \cap \TwistedLeviSubgroup[2]{1}{0}}\left(1,1\right),$$ and the statement follows.
\end{proof}

\subsection*{Acknowledgments}

We thank our advisors for their tremendous support. We appreciate the first author's advisor, James Cogdell, for his support and his comments on this paper. We are grateful to the second author's advisor, David Soudry, for suggesting the problem and for many helpful discussions during our work on the even case.

We are thankful to Ofir Gorodetsky for useful discussions and for the observation in \cref{rem:evaluation-of-square-sums}.

\bibliographystyle{amsalpha}
\bibliography{bibliography}

\end{document}